\documentclass[12pt]{amsart}
\usepackage{amsmath, amsthm, amscd, amssymb, amsfonts, latexsym}
\usepackage{fullpage}
\usepackage{bm}
\usepackage[all]{xy}
\usepackage{tikz}
\usepackage{mathrsfs}

\newcommand{\kommentar}[1]{}
\newcommand\cube{\begin{tikzpicture}[scale=2.3]
    \coordinate (A1) at (0, 0);
    \coordinate (A2) at (0, 0.1);
    \coordinate (A3) at (0.1, 0.1);
    \coordinate (A4) at (0.1, 0);
    \coordinate (B1) at (0.03, 0.03);
    \coordinate (B2) at (0.03, 0.13);
    \coordinate (B3) at (0.13, 0.13);
    \coordinate (B4) at (0.13, 0.03);

    \draw (A1) -- (A2);
    \draw (A2) -- (A3);
    \draw (A3) -- (A4);
    \draw (A4) -- (A1);
    \draw[densely dotted] (A1) -- (B1);
    \draw[densely dotted] (B1) -- (B2);
    \draw (A2) -- (B2);
    \draw (B2) -- (B3);
    \draw (A3) -- (B3);
    \draw (A4) -- (B4);
    \draw (B4) -- (B3);
    \draw[densely dotted] (B1) -- (B4);
\end{tikzpicture}}

\newcommand\tinycube{\begin{tikzpicture}[scale=1.3]
    \coordinate (A1) at (0, 0);
    \coordinate (A2) at (0, 0.1);
    \coordinate (A3) at (0.1, 0.1);
    \coordinate (A4) at (0.1, 0);
    \coordinate (B1) at (0.03, 0.03);
    \coordinate (B2) at (0.03, 0.13);
    \coordinate (B3) at (0.13, 0.13);
    \coordinate (B4) at (0.13, 0.03);

    \draw (A1) -- (A2);
    \draw (A2) -- (A3);
    \draw (A3) -- (A4);
    \draw (A4) -- (A1);
    \draw[densely dotted] (A1) -- (B1);
    \draw[densely dotted] (B1) -- (B2);
    \draw (A2) -- (B2);
    \draw (B2) -- (B3);
    \draw (A3) -- (B3);
    \draw (A4) -- (B4);
    \draw (B4) -- (B3);
    \draw[densely dotted] (B1) -- (B4);
\end{tikzpicture}}

\newcommand{\ccom}[1]{{\color{red}{CD: #1}} }

\newcommand{\mcom}[1]{{\color{orange}{Matilde: #1}} }

\newcommand{\acom}[1]{{\color{blue}{Alexandra: #1}} }

\newcommand{\F}{\mathbb F}

\newcommand{\kk}{k}
\newcommand{\JJ}{J}
\newcommand{\el}{l}

\newcommand{\Q}{\mathbb Q}
\newcommand{\R}{\mathbb R}

\newcommand{\C}{\mathbb C}

\newcommand{\sumstar}{\sideset{}{^*}\sum}

\newcommand{\A}{R_1}
\newcommand{\B}{R_2}

\DeclareMathOperator{\tr}{tr}

\DeclareMathOperator{\re}{Re}

\DeclareMathOperator{\ord}{ord}

\renewcommand{\pmod}[1]{\,(\mathrm{mod}\,#1)}

\newtheorem{lem}{Lemma}[section]
\newtheorem{prop}[lem]{Proposition}
\newtheorem{thm}[lem]{Theorem}

\newtheorem{cor}[lem]{Corollary}
\newtheorem{conj}[lem]{Conjecture}
\theoremstyle{definition}

\newtheorem{rem}[lem]{Remark}

\begin{document}

\author{Chantal David}
\address{Department of Mathematics and Statistics, Concordia University, 1455 de Maisonneuve West, Montr\'eal, Qu\'ebec, Canada H3G 1M8}
\email{chantal.david@concordia.ca}

\author{Alexandra Florea}
\address{Columbia University, Mathematics Department, Rm 606, MC 4417, 2990 Broadway, New York NY 10027, USA}
\email{aflorea@math.columbia.edu}

\author{Matilde Lalin}
\address{Universit\'e de Montr\'eal, Pavillon Andr\'e-Aisenstadt, D\'epartement de math\'ematiques et de statistique, CP 6128, succ. Centre-ville, Montr\'eal, Qu\'ebec, Canada H3C 3J7}
\email{mlalin@dms.umontreal.ca}

\title{Non-vanishing for cubic $L$--functions}

\begin{abstract}
We prove that there is a positive proportion of $L$-functions associated to cubic characters over $\F_q[T]$  that do not vanish at the critical point $s=1/2$. This is achieved by computing the first mollified moment 
using techniques previously developed by the authors in their work on the first moment of cubic $L$--functions, \kommentar{\acom{I'd rephrase: [...] by computing the first mollified moment using techniques previously developed by the authors in their work on the first moment of cubic $L$--functions[...]},\mcom{Changed} }and by obtaining a sharp upper  bound for the second mollified moment, building on work of Lester--Radziwi\l\l, which in turn develops further ideas from the work of Soundararajan, Harper, and Radziwi\l\l--Soundararajan. We work in the non-Kummer setting when $q\equiv 2 \pmod{3}$, but our results could be translated into the Kummer setting when $q\equiv 1\pmod{3}$ as well as into the number field case (assuming the Generalized Riemann Hypothesis).  Our positive proportion of non-vanishing is explicit, but extremely small, due to the fact that the implied constant in the upper bound for the mollified second moment is very large.

 \kommentar{\acom{Should we also add something like: Our positive proportion of non-vanishing is explicit, but extremely small, due to the fact that the implied constant in the upper bound for the mollified second moment is very large?}\mcom{Added.}}
\end{abstract}


\maketitle

\section{Introduction}
A famous conjecture of Chowla predicts that ${\textstyle L(\frac{1}{2}, \chi)} \neq 0$ for Dirichlet $L$-functions attached to  primitive characters $\chi$. It was first conjectured when $\chi$ is a quadratic character, which is the most studied case. For quadratic Dirichlet $L$-functions, 
\"Ozl\"uk and Snyder \cite{OS} showed, under the Generalized Riemann Hypothesis (GRH), that at least $15/16$ of the $L$-functions ${\textstyle L(\frac{1}{2}, \chi)} $ attached to quadratic characters $\chi$ do not vanish, by computing the one-level density for the low-lying zeroes in the family. The conjectures of Katz and Sarnak \cite{KS} on the zeroes of $L$-functions imply 
that ${\textstyle L(\frac{1}{2}, \chi)} \neq 0$ for almost all quadratic Dirichlet $L$-functions. 
Without assuming GRH, Soundararajan \cite{Sound} proved that at least $87.5 \%$ of the quadratic Dirichlet $L$-functions
do not vanish at $s=1/2$, by computing the first two {\it mollified} moments. It is well-known that using the first two (non-mollified) moments does not lead to a positive proportion of non-vanishing, as they grow too fast (see Conjecture $1.5.3$ in \cite{CFKRS} and the work of Jutila \cite{Jutila}.) Soundararajan \cite{Sound} also computed asymptotics for the first three moments, and Shen \cite{shen} obtained an asymptotic formula with the leading order term for the fourth moment, building on work of Soundararajan and Young \cite{SY}. 
A different approach was used by Diaconu, Goldfeld, and Hoffstein \cite{DGH} to compute the third moment.  Over function fields, asymptotics for the first four moments were obtained by Florea \cite{Florea1, Florea2, Florea3}. We refer the reader to those papers for more details.
Moreover, in the function field case, Bui and Florea \cite{Bui-Florea} obtained a proportion of non-vanishing of  at least $94 \%$ for quadratic Dirichlet $L$-functions, by computing the one-level density (those results are unconditional, as GRH is true over function fields).

In this paper, we consider the case of cubic Dirichlet $L$-functions. There are few articles in the literature about cubic Dirichlet $L$-functions, compared to the  abundance of papers on quadratic Dirichlet $L$-functions, as this family is more difficult, in part because of the presence of cubic Gauss sums.
The first moment of ${\textstyle L(\frac{1}{2}, \chi)} $, where $\chi$ is a primitive cubic character, was computed by Baier and Young over $\Q$ \cite{BY} (the non-Kummer case), by Luo for a thin sub-family over $\Q(\sqrt{-3})$ \cite{Luo} (the Kummer case), and by David, Florea, and Lalin  \cite{DFL}{ over function fields, in both the Kummer and the non-Kummer case, and for the full families.  \kommentar{\ccom{I added some words about the thin subfamily.} \acom{Not really relevant here, but is the second moment for the thin subfamily known or doable?}
\ccom{Very good question, but I would have to try it to know. The second moment, or the first moment in absolute value would be. nice too. For the one-level density, we can get support past (-1,1) for the thin-subfamily under GRH. Does this mean that morally we get the second moment?} \acom{I'd guess so. }}

In these three papers, the authors obtained lower bounds for the number of non-vanishing cubic twists, but not  positive proportions, by using upper bounds on higher moments. 
Ellenberg, Li, and Shusterman \cite{ELS} use algebraic geometry techniques to extend the results of \cite{DFL} to $\ell$-twists over function fields and  improve upon the lower bound for the number of non-vanishing cubic twists (but the proportion is still nonpositive). Obtaining an asymptotic for the second moment for cubic Dirichlet $L$-functions is still an open question, over functions fields or number fields. 
Moreover, for the case of cubic Dirichlet $L$-functions, computing the one-level density can only be done for limited support of the Fourier transform of the test function,
and that is not enough to lead to a positive proportion of non-vanishing, even under GRH \cite{Cho-Park, Meisner}.

We prove in this paper that there is a positive proportion of non-vanishing for cubic Dirichlet $L$-functions at $s=1/2$ over function fields, in the non-Kummer case. 
\begin{thm}\label{positive-prop}
 Let $q \equiv 2 \mod 3$. Let $\mathcal{C}(g)$ be the set of primitive cubic Dirichlet characters of genus $g$ over $\F_q[T]$. Then, as $g \rightarrow \infty$,
 $$\# \left\{ \chi \in \mathcal{C}(g) \;:\: {\textstyle L(\frac{1}{2}, \chi)} \neq 0 \right\} \gg \# \mathcal{C}(g).$$
 \end{thm}


Theorem \ref{positive-prop} is obtained  by using the breakthrough work on sharp upper bounds for {moments} of $|\zeta(1/2+it)|$ by Soundararajan \cite{Sound-again} and Harper \cite{Harper}, under GRH. Their techniques, together with ideas appearing in the work of Radziwi\l\l~ and Soundararajan \cite{R-S} on distributions of central $L$-values of quadratic twists of elliptic curves, were further developed by Lester and Radziwi\l\l~ in  \cite{L-R}, where they obtained sharp upper bounds for mollified moments of quadratic twists of modular forms. Our work owes a lot to these papers and circles of ideas. 

To obtain Theorem  \ref{positive-prop}, we need to compute the first mollified moment, generalizing our previous work \cite{DFL} (Theorem \ref{first-moment}), and to obtain a sharp upper bound for the second mollified moment (Theorem \ref{second-moment}). In fact,  we obtain upper bounds for all mollified moments, not only the second moment and integral moments. Using Theorems \ref{first-moment} and \ref{second-moment}, the positive proportion of Theorem \ref{positive-prop} follows from a simple application of the Cauchy--Schwarz inequality.

As noted by Harper in the case of the Riemann zeta function, \kommentar{\acom{in the case of the Riemann-zeta function}\mcom{Added}} the sharp upper bound for the $\kk^{\text{th}}$ moment is obtained at the cost of an enormous constant of the order $e^{e^{c \kk}}$, for some absolute constant $c > 0$. Hence our positive proportion of non-vanishing is extremely small, but explicit nonetheless. 

The method we used for the family of cubic $L$--functions would be expected to work in general for families where one can compute the first moment with a power saving error term, and it is useful in families where the second moment is not known. The method used in the paper allows to get a sharp upper bound for the second mollified moment, which is enough to obtain a positive proportion of non-vanishing (under GRH). For the family of cubic twists, we expect that the Kummer case would be similar, and the results would hold in that setting as well.
Our result should also transfer over to number fields, but it would be conditional on GRH.
\kommentar{\acom{Here maybe we can mention that the method works in general for families where we know how to compute the first moment with a power saving error term, and it is most useful of course for families where we don't know the second moment. Do you know other examples where we can compute the first moment with a power saving ET, but we don't know the second moment?}
\ccom{Is that OK? I moved the 2 lines about Kummer and number fields which were above.}}

We first state the standard conjecture for moments of the family of cubic Dirichlet $L$-functions.  We refer the reader to Section \ref{function-fields} for more information about the family of cubic Dirichlet $L$-functions over function fields in the non-Kummer case (i.e. $q \equiv 2 \pmod 3$).

\begin{conj}  \label{conj-moments}
 Let $q \equiv 2 \pmod 3$. Let $\mathcal{C}(g)$ be the set of primitive cubic Dirichlet characters of genus $g$ over $\F_q[T]$. Then as $g \rightarrow \infty$, 
$$ \frac{1}{\# \mathcal{C}(g)} 
\sum_{\chi \in {\mathcal{C}}(g)} {\textstyle|L(\frac{1}{2},\chi) |^{2\kk}}  \sim \frac{{a_\kk \mathfrak{g}_\kk}}{(\kk^2)!} P_\kk(g),$$ 
where $P_\kk(g)$ is a monic polynomial of degree $\kk^2$, $a_\kk$ is an arithmetic factor depending on the family, and 
$$\mathfrak{g}_\kk= (\kk^2)! \prod_{j=0}^{\kk-1}\frac{j!}{(j+\kk)!}.$$
\end{conj}
A testament to the fact that moments of $L$-functions are hard to compute is the fact that simply conjecturing an asymptotic is very difficult. The constants $\mathfrak{g}_\kk$ were obtained by Keating and Snaith based on considerations from random matrix theory \cite{Keating-Snaith}. Number theoretic heuristic arguments were used in the work of Conrey, Farmer, Keating, Rubinstein, and Snaith \cite{CFKRS} to generalize Conjecture \ref{conj-moments} to include lower order terms, and more recently by Conrey and Keating \cite{CK1,CK2,CK3,CK4,CK5}. 
The order of magnitude $g^{\kk^2}$ is easy to conjecture, as it comes from the size  of the contribution of the {\it diagonal terms}. In the case of cubic characters, this will come from the fact that cubic characters are trivial on cubes. For the first moment, only diagonal terms contribute to the asymptotic of the previously cited work \cite{BY, Luo, DFL}. For the second (and higher) moments, there will be a contribution from the  {\it off-diagonal terms}. The off-diagonal term contribution can be estimated in the case of quadratic characters, but it is open for the family of cubic characters, where only the first moment without absolute value (the sum of 
${\textstyle L(\frac{1}{2},\chi)}$) is known. In the work of Soundararajan and Harper,  which provides an upper bound of the exact order of magnitude for all moments of $\zeta(s)$, the upper bound is built 
only from the contribution of the diagonal terms. This follows from a key result of Soundararajan who proved that one can upper bound $\log 
 |L(\frac{1}{2},\chi) |$ by a short sum over primes. In our setting, we use Lemma  \ref{like-chandee} which is the analogue of Soundararajan's key inequality. Considering only the diagonal terms and neglecting the rest leads to a constant for the upper bound which is much larger than $\mathfrak{g}_\kk$.  In particular, shortening the Dirichlet polynomial produces a large contribution from the $(g+2)/N$ term.  The techniques used to get the upper bound generate a constant of size $e^{e^{ck}}$, as noted by Harper \cite{Harper}. 
 
\subsection{Statement of the results}

We state our two results about the mollified moments. Let $\kappa>0$. The mollifier we use,  
$M(\chi; \frac{1}{\kappa})$, is defined in Section \ref{mollifier}, and depends on the parameter $\kappa$. We will later choose $\kappa=1$ in the application to Theorem \ref{first-moment}.

\begin{thm} \label{first-moment}  Let $q \equiv 2 \pmod 3$. Let $\mathcal{C}(g)$ be the set of primitive cubic Dirichlet characters of genus $g$ over $\F_q[T]$. Then as $g \rightarrow \infty$, 
$$ 
\sum_{\chi \in \mathcal{C}(g)} 
{\textstyle L(\frac{1}{2},\chi) }{\textstyle M(\chi;1)}  = A q^{g+2}+ O (q^{\delta g} ),$$ for some $0<\delta<1$ (see equation \eqref{et_moll} for more details on $\delta$) and where the constant $A$ is given in equation \eqref{arghh}.
\end{thm}


\kommentar{\acom{Here I don't know if we want to write the error term in terms of $\theta_j$ and $\ell_j$. Maybe we could write the error term as $o(q^g)$, and at the end of the proof explain why that is $o(q^g)$. Also, we have the corollary below on the lower bound for the first moment. I don't know if we want to write it like that or as a remark.}
\ccom{Agree, I think we can do without the sum $\sum_{j=0}^{\JJ} \ell_j \theta_j$ in this statement. We can write $\sim$, or $O(q^{\delta g})$ for some $\delta<1$ if we want to stress that we have a power saving. And we can refer to the last section, page 50.}\mcom{I agree}}
\begin{cor}
\label{lb}
With the same notation as before, we have that
$$ \sum_{\chi \in \mathcal{C}(g)} 
{\textstyle L(\frac{1}{2},\chi) }{\textstyle M(\chi;1)} \geq 0.6143 q^{g+2}.$$
\end{cor} 
\begin{rem} It is easy to estimate that $\# \mathcal{C}(g) \sim c_3 q^{g+2}$ for some explicit constant $c_3$ (see \cite{DFL}.) Dividing by the size of the family, we then prove that the first mollified moment of $L(\frac{1}{2}, \chi)$ is asymptotic to a constant, which is the conjectural asymptotic. 
This is also the asymptotic for the non-mollified moment (with a different constant) as proven in \cite{DFL}.
This asymptotic is not included in Conjecture \ref{conj-moments} which is concerned with the moments of the {absolute value} of the $L$-functions. The moments of
$L(\frac{1}{2}, \chi)^{\kk_1}  {\overline{L(\frac{1}{2}, \chi)^{\kk_2}}}$, for general positive $k_1, k_2$,  are conjectured to grow as a polynomial of degree $\kk_1 \kk_2$ in $g$, see \cite{DLN}. Note that the conjectures in \cite{DLN} hold for cubic twists of elliptic curves, but both families have the same symmetry, so the main terms will have a similar shape. 
Theorem \ref{first-moment} corresponds to the case $\kk_1=1, \kk_2=0$, and Conjecture \ref{conj-moments} to the case $\kk_1=\kk_2=\kk$.
\end{rem}

The following upper bound for the second moment is the analogue of Proposition 4.1 in \cite{L-R}.
\begin{thm}\label{moll_ub} \label{second-moment} Let $\kk, \kappa > 0$ such that $\kk \kappa $ is an even integer and $\kk \kappa \leq C$ for some absolute constant $C$.  
Let $q \equiv 2 \pmod 3$. Let $\mathcal{C}(g)$ be the set of primitive cubic Dirichlet characters of genus $g$ over $\F_q[T]$. Then as $g \rightarrow \infty$, 
$$
\sum_{\chi \in \mathcal{C}(g)}  {\textstyle|L(\frac{1}{2},\chi) |^{ \kk} | \; M(\chi;\frac{1}{\kappa})|^{ \kk \kappa }} \ll q^g.$$ \label{ub}
\end{thm}

\begin{rem}  Because of the presence of the mollifier, dividing by $\# \mathcal{C}(g)$, all moments are bounded by a constant, and they do not grow. Using the first and second moment then leads to a positive proportion of non-vanishing.
\end{rem}

\subsection{Proof of Theorem \ref{positive-prop}}
The proof of Theorem \ref{positive-prop}} follows from a simple application of Cauchy--Schwarz and Theorems \ref{first-moment} and \ref{second-moment} for $\kappa=1$. Indeed, 
$$
\sum_{\substack{ \chi \in \mathcal{C}(g) \\L(\frac{1}{2}, \chi) \neq 0}} 1  \geq 
\frac{ \Big| \sum_{\chi \in \mathcal{C}(g)} \displaystyle {\textstyle L(\frac{1}{2},\chi) }{\textstyle M(\chi;1)}\Big| ^2}{ \displaystyle 
\sum_{\chi \in \mathcal{C}(g)} 
 |{\textstyle L(\frac{1}{2},\chi) }{\textstyle M(\chi;1)}|^2} \gg q^{g}. \qed$$
\begin{rem}
Combining Corollary \ref{lb} and equation \eqref{big_bound}, we get the explicit proportion
$$ \#\{ \chi \in \mathcal{C}(g) \, |\,{\textstyle L(\frac{1}{2}, \chi)} \neq 0\}   \geq  0.3773 e^{-e^{182}} q^{g+2},$$
and using \eqref{boundc3}, 
\[\frac{\#\{\chi \in \mathcal{C}(g) \,|\, L(\frac{1}{2}, \chi) \neq 0\}}{\#\mathcal{C}(g)}\geq \left(1-e^{-e^{84}}\right)^2\frac{e^{-e^{182}}}{\zeta_q(2)^3\zeta_q(3)^2}\geq 0.4718e^{-e^{182}}.\]
\end{rem}

\subsection{Overview of the paper}

This paper contains two main results, which are proven with different techniques. 

We first prove the upper bound for 
the mollified moments, adapting the setting and notation of \cite{L-R} to the case of cubic characters (and to the function field case). 
The combinatorics to give an upper bound to the contribution of the diagonal terms is significantly more complicated, in part because the special values of the cubic $L$-functions are not real numbers, and they have to be considered in absolute value, and in part because we are identifying cubes and not squares. This also applies to the proof of the almost-sharp upper bound for the $L$-functions we consider, which is needed as a starting point to prove the sharp upper bound. Because we are dealing with cubic characters, we also have to bound the contribution of the squares of the primes, unlike the case of quadratic characters, where the squares of the primes contribute to the main term. In the language of random matrix theory, the family of cubic characters is a unitary family, and the family of quadratic characters is a symplectic family (for Dirichlet twists) or an orthogonal family (for twists of a modular form). In \cite{Harper}, the author also bounds the contribution of the squares of the primes to get sharp upper bounds on the moments of ${\textstyle|\zeta(\frac12 + it)|}$, which is a unitary family. In our case, because of the presence of the mollifier, mixing the square of the primes with the primes is very cumbersome, and we treat them separately with an additional use of the Cauchy--Schwarz inequality.
The contribution from the squares of the primes morally behaves like $L(1,\overline{\chi})$. Bounding this contribution is similar to getting an upper bound for the average of $L(1,\chi)$, which is much simpler than the original problem of bounding the average of $L(\frac12, \chi).$ 


We then proceed to the evaluation of the first mollified moment. Because the mollifier is a finite Dirichlet polynomial, this amounts to the computation of a ``twisted first moment", see Proposition \ref{twist}. 
The evaluation of this twisted first moment is similar to the evaluation of the first moment for the non-Kummer family in \cite{DFL}, relying on the approximate functional equation and powerful results on the distribution of cubic Gauss sums.



The structure of the paper is as follows. Section \ref{function-fields} contains the standard properties of cubic characters over function fields that are used throughout the paper. Section \ref{setting} contains the proof of Theorem \ref{moll_ub} modulo three important results proven in three subsequent sections: a technical lemma proven in Section \ref{friendly-lemma}, an upper bound for the contribution of the square of the primes in Section \ref{square-primes}, and the proof of a proposition  giving an almost-sharp upper bound for the unmollified moments of 
${\textstyle L(\frac{1}{2},\chi)}$ in Section \ref{ub_sound}. 
In Section \ref{explicit-UB}, we give some estimates on the (extremely small) positive proportion of Theorem \ref{positive-prop}. Finally, 
Section \ref{sec:first-moment} contains the asymptotic for the first mollified moment, following the lines of \cite{DFL} where the first moment is computed.

\medskip
\noindent{\bf Acknowledgements.} The authors would like to thank Maksym Radziwi\l\l~  
for drawing our attention to his work with Lester and  for very helpful discussions, and to Stephen Lester for interesting insigths and comments. 
The research of the first and third authors is supported by the National Science and Engineering Research Council of Canada (NSERC) and 
the Fonds de recherche du Qu\'ebec -- Nature et technologies (FRQNT). The second author of the paper was supported by a National Science Foundation (NSF) Postdoctoral Fellowship during part of the research which led to this paper and she wishes to thank the initiative ``A Room of One's Own'' for focused time.
\section{Background}
\label{function-fields}

Let $q$ be an odd prime power. We denote by $\mathcal{M}_q$ the set of monic polynomials of $\F_q[T]$, by $\mathcal{M}_{q,\leq d}$ the subset of degree less or equal to $d$, and by  $\mathcal{M}_{q,d}$ the subset of degree exactly $d$. Similarly, $\mathcal{H}_q$, $\mathcal{H}_{q, \leq d}$, and $\mathcal{H}_{q,d}$ denote the analogous sets of monic square-free polynomials. 
In general, all sums over polynomials in $\F_q[T]$ are always taken over monic polynomials.
The norm of a polynomial $f(T)\in \F_q[T]$ is given by 
\[|f|_q=q^{\deg(f)}.\] 
In particular, if $f(T)\in \F_{q^n}[T]$, we have $|f|_{q^n}=q^{n\deg(f)}$ for any positive $n$. We will write $|f|$ instead of $|f|_q$ when there is no ambiguity.

The primes of $\F_q[T]$ are the monic irreducible polynomials. Let $\pi(n)$ be the number of primes of $\F_q[T]$ of degree $n$. By considering all the roots of these polynomials, we see that  $n\pi(n)$ counts the number of elements in $\F_{q^n}$ of degree exactly $n$ over the base field $\F_q$, which is less or equal than the total number of elements in $\F_{q^n}$. Therefore 
\begin{equation}\label{ppt-bound}
 \pi(n)\leq \frac{q^n}{n}.
\end{equation}
More precisely, the Prime Polynomial Theorem 
(\cite{rosen}, Theorem 2.2) states that the number $\pi(n)$  of primes of $\F_q[T]$ of degree $n$ satisfies
\begin{equation}
\pi(n)=\frac{q^n}{n}+O\Big(\frac{q^{n/2}}{n}\Big).
\label{ppt}
\end{equation}

The von-Mangoldt function is defined as
$$ 
\Lambda(f) = 
\begin{cases}
\deg(P) & \mbox{ if } f=cP^k, c \in \F_q^{*}, P \text{ prime}, \\ 
0 & \mbox{ otherwise.}
\end{cases}
$$
Recall that for $f\in \F_q[T]$ the M\"obius function $\mu(f)$ is $0$ if $f$ is not square-free and $(-1)^t$ if $f$ is a constant times a product of $t$ different primes. 
The Euler $\phi_q$ function is defined as $\# (\F_q[T]/( f \F_q[T]))^{*}$. It satisfies
\[\phi_q(f)=|f|_q\prod_{P\mid f}(1-|P|_q^{-1}),\]
and
\[\sum_{\substack{d\in \mathcal{M}_q\\d\mid f}}\frac{\mu(d)}{|d|_q}=\frac{\phi_q(f)}{|f|_q}.\] 
When $f(T) \in \F_{q^n}[T]$ we may consider $\phi_{q^n}$ defined similarly.

In this paper we consider the non-Kummer case of cubic Dirichlet character over $\F_q[T]$, where $q \equiv 2 \pmod 3$. These characters are best described as a subset of the cubic characters over $\F_{q^2}[T]$. Notice that $q^2 \equiv 1 \pmod 3$, so 
let $q\equiv 1 \pmod{3}$ momentarily, and we proceed to construct cubic Dirichlet characters over $\F_{q}[T]$ as follows. We fix  an isomorphism $\Omega$ between the third roots of unity $\mu_3\subset \C^*$ and the cubic roots of 1 in $\F_{q}^*$.  Let $P$ be a prime polynomial in $\F_{q}[T]$, and let $f \in \F_{q}[T]$ be such that $P\nmid f$. Then there is  a unique $\alpha \in \mu_3$ such that 
\[f^{\frac{q^{\deg(P)}-1}{3}}\equiv \Omega(\alpha) \pmod{P}.\]
Remark that the above equation is solvable because $q \equiv 1 \pmod{3}$. 
Then we set 
\[\chi_P(f):=\alpha.\]
We remark that there are two such characters, $\chi_P$ and $\overline{\chi_P}=\chi_P^2$, depending on the choice of $\Omega$.  

This construction is extended by multiplicativity to any monic polynomial $F\in \F_q[T]$. In other words, if $F=P_1^{e_1}\cdots P_s^{e_s}$ where the $P_i$ are distinct primes, then 
\[\chi_F=\chi_{P_1}^{e_1}\cdots \chi_{P_s}^{e_s}.\]
We have that $\chi_F$ is a cubic character modulo $P_1\cdots P_s$. It is primitive if and only if $e_i=1$ or $e_i=2$ for all $i$. 

If $q \equiv 1 \pmod 6$, then we have perfect Cubic Reciprocity. Namely, let $a, b \in \F_q[T]$ be relatively prime monic polynomials, and let $\chi_a$ and $\chi_b$ be the cubic residue symbols defined above. If $q \equiv 1 \pmod 6$, then
\begin{equation} \label{cubic-reciprocity}
\chi_a(b) = \chi_b(a).
\end{equation}
\kommentar{
\begin{proof} This is Theorem 3.5 in \cite{Rosen} in the case where $a$ and $b$ are monic and $q \equiv 1 \pmod 6$. \end{proof}
\acom{I would probably just write this as an equation instead of a lemma, but I don't have a strong preference.}
}

When $q \equiv 2 \pmod 3$, the above construction of $\chi_P$ will also give a cubic character as long as $P$ has {\em even} degree, 
and the character can be extended by multiplicativity.
In the non-Kummer case, a better way to describe cubic characters is to see them as restriction characters defined over $\F_{q^2}[T]$. This 
description was formulated by Bary-Soroker and Meisner \cite{BSM}, who generalized the work of Baier and Young \cite{BY} from number fields to function fields. We summarize their work here. Let $\pi$ be a prime in $\F_{q^2}[T]$ lying over a prime $P\in \F_q[T]$ of even degree. Then $P$ splits and we can write $P=\pi \tilde{\pi}$, where 
$\tilde{\pi}$ denotes the  Galois conjugate of $\pi$. Remark that $P \in \F_q[T]$ splits if and only if $\deg(P)$ is even. Then the restrictions of $\chi_\pi$ and $\chi_{\tilde{\pi}}$ to $\F_q[T]$ are $\chi_P$ and $\overline{\chi_P}$ (possibly exchanging the order of characters). Using multiplicativity, it follows that the cubic characters over $\F_q[T]$ are given by the characters $\chi_F$  where $F \in \F_{q^2}[T]$ is square-free and not divisible by any prime $P(T) \in \F_q[T]$.

\kommentar{\acom{As in the introduction, if $\mathcal{C}(g)$ denotes the set of primitive cubic Dirichlet characters over $F_q[T]$ of genus $g$, then from Lemma $2.10$ in \cite{DFL}, we have that for $g$ even,
$$ | \mathcal{C}(g)| =Bq^g+O(q^{\frac{g}{2}(1+\varepsilon)}),$$ for some explicit constant $B$.}
\mcom{I agree to add this, but we should do it after we talk about genus and its relation with the degree of the conductor.}}

Given a primitive cubic Dirichlet character $\chi$ of conductor $F=P_1\cdots P_s$, the $L$-function is defined by
\[L(s,\chi):=\sum_{f \in \mathcal{M}_q}\frac{\chi(f)}{|f|_q^s}= \sum_{\deg(F)<d} q^{-ds} \sum_{f \in \mathcal{M}_{q,d}} \chi(f),\] 
where the second equality follows from the orthogonality relations for $\chi$. The $L$-function above can be written as a polynomial by making the change of variables $u=q^{-s}$, namely,
\[\mathcal{L}(u,\chi)= \sum_{\deg(F)<d} u^d \sum_{f \in \mathcal{M}_{q,d}} \chi(f).\]
Let $C$ be a curve of genus $g$ over $\F_q(T)$ whose function field is a cyclic cubic extension of $\F_q(T)$. From the Weil conjectures, the zeta function of the curve is given by 
\[\mathcal{Z}_C(u)=\frac{\mathcal{P}_C(u)}{(1-u)(1-qu)}.\]
In the case under consideration (that is, $q\equiv 2 \pmod{3}$), we have that \begin{equation} \label{poly-PC}
\mathcal{P}_C(u)=\frac{\mathcal{L}(u,\chi)\mathcal{L}(u,\overline{\chi})}{(1-u)^2},\end{equation}
where $\chi$ and $\overline{\chi}$ are the two cubic Dirichlet characters corresponding to the function field of $C$. 
Because of the additional factors of $(1-u)$ in the denominator of \eqref{poly-PC}, there are extra sums in the approximate functional equation for
$\mathcal{L}(u,\chi)$ in this case (see  Proposition \ref{prop-AFE}).
Furthermore, the Riemann--Hurwitz formula implies that the conductor $F$ of $\chi$ and $\overline{\chi}$ satisfies  $\deg(F)=g+2.$

As in the introduction, let $\mathcal{C}(g)$ denote the set of primitive cubic Dirichlet characters of genus $g$ over $\F_q[T]$.
From the above discussion, we have that
\begin{equation}  \label{square-free-sum}
\mathcal{C}(g) = \left\{ \chi_F \in {\mathcal{H}}_{q^2, g/2+1} \;:\; P \mid F \Rightarrow P \not\in \F_q[T] \right\},
\end{equation}
and in particular $g$ is even.
 In that case, from Lemma $2.10$ in \cite{DFL}, we have
$$ \# \mathcal{C}(g) =c_3q^{g+2}+O(q^{\frac{g}{2}(1+\varepsilon)}),$$ where
\[c_3=\prod_{\substack{R \in \F_q[T]\\\deg(R) \,\mathrm{ odd}} }\left(1-\frac{1}{|R|^2}\right)\prod_{\substack{R \in \F_q[T]\\\deg(R) \,\mathrm{ even}} }\left(1-\frac{3}{|R|^2}+\frac{2}{|R|^3}\right).\]
We remark that
\begin{equation}\label{boundc3}
c_3\leq  \prod_{\substack{R \in \F_q[T]} }\left(1-\frac{1}{|R|^2}\right)=\zeta_q(2)^{-1}.
\end{equation}


The following statement (Proposition $2.5$ from \cite{DFL}) provides the approximate functional equation of the $L$--function. 
\begin{prop} [Approximate Functional Equation, Proposition 2.5 \cite{DFL}] \label{prop-AFE} Let $q\equiv 2 \pmod{3}$ 
and let $\chi$ be a primitive cubic character of modulus $F$. Let $X \leq g$. Then
\begin{align*}
{\textstyle\mathcal{L}( \frac{1}{\sqrt{q}} ,\chi )} =&\sum_{f \in \mathcal{M}_{q, \leq X}} \frac{ \chi(f)}{ q^{\deg(f)/2}} + \omega (\chi) \sum_{f \in \mathcal{M}_{q, \leq g-X-1}} \frac{ \overline{\chi}(f)}{q^{\deg(f)/2}} \\
&+ \frac{1}{1-\sqrt{q}} \sum_{f \in \mathcal{M}_{q, X+1}} \frac{ \chi(f)}{q^{\deg(f)/2}} + \frac{\omega(\chi)}{1-\sqrt{q}} \sum_{f \in \mathcal{M}_{q,g-X}} \frac{ \overline{\chi}(f)}{q^{\deg(f)/2}},
\end{align*}
where 
\begin{equation}\label{rootnumber}\omega(\chi) = -q^{-(\deg(F)-2)/2} \sum_{f \in \mathcal{M}_{\deg(F)-1}} \chi(f)
 \end{equation}
 is the root number, and $g=\deg(F)-2$.
\end{prop}

Now let $\chi$ be a primitive cubic character of conductor $F$ defined over $\F_q[T]$. Then, for 
$\re(s) \geq 1/2$ and for all $\varepsilon > 0$, we have the following upper bound
\begin{equation}
\left|L(s, \chi) \right| \ll q^{ \varepsilon \deg(F)}.
\label{lindelof}
\end{equation}
For $\re(s) \geq 1$ and for all $\varepsilon > 0$, we also have the lower bound
\begin{equation}
\left| L(s, \chi) \right| \gg q^{ - \varepsilon \deg(F)}.
\label{folednil}
\end{equation}
(See Lemmas $2.6$ and $2.7$ in \cite{DFL}.)
\kommentar{The following lemmas provide upper and lower bounds for $L$--functions. 
\begin{lem}[Lemma 2.6, \cite{DFL}] \label{lindelof}
Let $\chi$ be a primitive cubic character of conductor $F$ defined over $\F_q[T]$. Then, for 
$\re(s) \geq 1/2$ and for all $\varepsilon > 0$,
$$\left|L(s, \chi) \right| \ll q^{ \varepsilon \deg(F)}.$$
\end{lem}

\begin{lem}[Lemma 2.7, \cite{DFL}]\label{folednil}
Let $\chi$ be a primitive cubic character of conductor $F$ defined over $\F_q[T]$. Then, for $\re(s) \geq 1$ and for all $\varepsilon > 0$, 
$$\left| L(s, \chi) \right| \gg q^{ - \varepsilon \deg(F)}.$$
\end{lem}

\acom{Here I would just display the bounds as equations instead of writing them as lemmas, but again, I don't have a strong preference.}
}

\kommentar{As previously remarked, when $q\equiv 2 \pmod{3}$, $\chi$ is a restriction of a character over $\F_{q^2}[T]$ of conductor $\pi_1\cdots \pi_s$, where each $\pi_i$ is a prime lying over $P_i$.  Then the degree of its conductor over $\mathbb{F}_q[T]$ is $g/2+1$. Thus we obtain the following. 
\begin{lem} \label{square-free-sum}
Let $q\equiv 2 \pmod{3}$ and $h \in \F_q[T]$. Then 
\[
\sum_{\chi \in \mathcal{C}(g)}\chi(h){\textstyle L(\frac{1}{2}, \chi)} =  \sum_{\substack{F \in \mathcal{H}_{q^2, g/2+1} \\
P \mid F \Rightarrow P \not\in \F_q[T]}} \chi(h)\textstyle{L(\frac{1}{2}, \chi_F)}.\]
\end{lem}}

We recall Perron's formula over $\F_q[T]$ which will be used several times in Section \ref{sec:first-moment}.

\begin{lem}[Perron's Formula] \label{perron}
If the generating series $\mathcal{A}(u) =\sum_{f \in \mathcal{M}_q} a(f) u^{\deg(f)}$ 
 is absolutely convergent in $|u|\leq r<1$, then
\[\sum_{{f} \in \mathcal{M}_{q,n}} a(f)=\frac{1}{2\pi i}\oint_{|u|=r} \frac{\mathcal{A}(u)}{u^n} \; \frac{du}{u}\]
and
\[\sum_{{f} \in \mathcal{M}_{q,\leq n}} a(f)=\frac{1}{2\pi i}\oint_{|u|=r} \frac{\mathcal{A}(u)}{u^{n}(1-u)} \; \frac{du}{u},\]
where, in the usual notation, we take $\oint$ to signify the integral over the circle around the origin oriented counterclockwise. 
\end{lem}

Finally, we recall the Weil bound for sums over primes. Let $\chi$ be a character modulo $B$, where $B$ is not a cube. Then
\begin{equation} \label{weil}
 \Big|\sum_{P \in \mathcal{P}_n} \chi(P)  \Big| \ll q^{n/2} \frac{\deg(B)}{n},
 \end{equation} where the sum is over monic, irreducible polynomials of degree $n$.

The following notation will be used often. We will write 
 $A \leq_\varepsilon B$ to mean $A\leq (1+\varepsilon) B$ for any $\varepsilon >0$ as $g\rightarrow \infty$.

\subsection{Cubic Gauss Sums} 

Let  $q \equiv 1 \pmod 3$. We now define cubic Gauss sums, and we state the result for the distribution of cubic Gauss sums that we are using in Section \ref{sec:first-moment}. 

Let $\chi$ be a (not necessarily primitive) cubic character of modulus $F$. 
The generalized cubic Gauss sum is defined by
\begin{eqnarray} \label{gen-GS} 
G_q(V,F) = \sum_{u \pmod F} \chi_F(u) e_q\left( \frac{uV}{F} \right ), \end{eqnarray}
where
\[e_q(a) = e^{\frac{2 \pi i \tr_{\F_q/\F_p}(a_1) }{p}},\]
is the exponential defined by Hayes \cite{hayes}, for any $a \in \mathbb{F}_q((1/T))$. 

When $(A, F)=1$, it is easy to see that
\begin{equation} \label{gauss-sum-prop}
G_q(AV, f) = \overline{\chi_f}(A) G_q(V, f).
\end{equation}
Furthermore, the shifted Gauss sum is almost multiplicative as a function of $F$. Namely, if $q \equiv 1 \pmod 6$, and  if $(F_1 ,F_2)=1$, then 
\begin{eqnarray*} G_q(V,F_1 F_2) &=& \chi_{F_1}(F_2)^2 G_q(V, F_1 )G_q(V, F_2). \end{eqnarray*}

The generating series of the Gauss sums are given by 
\[\Psi_q(f,u)= \sum_{F \in \mathcal{M}_q} G_q(f,F)u^{\deg(F)}\]
and 
\begin{equation} \label{tilde-F}
\Tilde{\Psi}_q(f,u) = \sum_{\substack{F \in \mathcal{M}_q \\ (F,f)=1}} G_q(f,F) u^{\deg(F)}.
\end{equation}
The function $\Psi_q(f,u)$ was studied by Hoffstein \cite{hoffstein} and Patterson \cite{patterson}. In \cite{DFL} we worked with $\Tilde{\Psi}_q(f,u)$ and proved the following results. 

\begin{prop}[Proposition 3.1 and Lemmas 3.9 and 3.11, \cite{DFL}] \label{big-F-tilde-corrected}
Let $f=f_1 f_2^2 f_3^3$, where $f_1$ and $f_2$ are square-free and coprime, and let $f_3^*$ be the product of the primes dividing $f_3$ but not dividing $f_1f_2$. Then
\begin{align*} \nonumber
\sum_{\substack{F \in \mathcal{M}_{q,d}\\(F,f)=1}} G_q(f,F) = & 
\delta_{f_2=1} \frac{ q^{\frac{4d}{3} - \frac{4}{3} [d+ \deg(f_1)]_3} }{ \zeta_q(2)|f_1|_q^{2/3}} \overline{G_q(1,f_1)} \rho(1, [d+ \deg(f_1)]_3)\prod_{P\mid f_1 f_3^*} \left ( 1+\frac{1}{|P|_q}\right )^{-1} \\ \nonumber
&+ O \left ( \delta_{f_2=1} \frac{q^{\frac{d}{3}+\varepsilon d}} {|f_1|_q^{\frac{1}{6}}}\right )
+ \frac{1}{2\pi i} \oint_{|u|=q^{-\sigma}} \frac{\tilde{\Psi}_q(f,u)}{u^d}\frac{du}{u} 
\end{align*}
with $2/3<\sigma< 4/3$ and where $\Tilde{\Psi}_q(f,u)$ is given by \eqref{tilde-F},  $[x]_3$ denotes  an integer $a\in \{0,1,2\}$ such that $x\equiv a\pmod{3}$,  
\begin{equation*}
\rho(1, a)=\left\{ \begin{array}{cl}1 & a=0,\\
\tau(\chi_3) q &a=1,\\0 &a=2,
\end{array}\right.
\end{equation*}
and
\[\tau(\chi_3) =\sum_{a\in \F_q^*} \chi_3(a)e^{2\pi i \tr_{\F_q/\F_p}(a)/p}, \qquad \chi_3(a)=\Omega^{-1}\left (a^\frac{q-1}{3} \right).\]

Moreover, we have 
\[\frac{1}{2\pi i} \oint_{|u|=q^{-\sigma}} \frac{\tilde{\Psi}_q(f,u)}{u^d}\frac{du}{u} \ll q^{\sigma d} |f|_q^{\frac{1}{2} ( \frac{3}{2} - \sigma)} \mbox{  and  }\Tilde{\Psi}_q(f,u) \ll |f|_q^{{\frac{1}{2}}(\frac{3}{2}-\sigma)+\varepsilon}.\]

\end{prop}

When $q \equiv 2 \pmod 3$, the root number in equation \eqref{rootnumber} can be expressed in terms of cubic Gauss sums over $\F_{q^2}[T]$, as proven in (\cite{DFL}, Section 4.4). 
Let  $\chi$ be a primitive character of conductor $F \in \F_q[T]$. Then, $F$ is square-free and divisible only by primes $P(T)$ of even degree and 
\begin{equation}\label{Gauss-root}
\omega(\chi)=q^{-\frac{g}{2}-1}G_{q^2}(1,F),
\end{equation}
where
\begin{equation}\label{Gauss-size}
G_{q^2}(1,F)=q^{\deg(F)} 
\end{equation}
for $F\in \F_q[T]$ square-free (See Lemma $4.4$ in \cite{DFL}.)



\section{Setting and proof of Theorem \ref{moll_ub}}
\label{setting}
\subsection{Setting} 

Following the work of Soundararajan  on upper bounds for the Riemann zeta function \cite{Sound-again}, we first show that we can bound 
$\log | {\textstyle L(\frac{1}{2},\chi) }|$ by a short Dirichlet polynomial. The following statement is analogous to Proposition 4.3 from \cite{BFKRG}. 
\begin{lem} \label{like-chandee} Let $q \equiv 2 \pmod 3$ and let $\chi$ be a cubic Dirichlet character of genus $g$ over $\F_q[T]$. Then for $N\leq g+2$ we have 
\begin{align*}
\log | {\textstyle L(\frac{1}{2},\chi) }| & \leq \Re \Big(  \sum_{\deg(f) \leq N} \frac{ \Lambda(f) \chi(f)(N-\deg(f))}{N |f|^{\frac{1}{2}+\frac{1}{N\log q}} \deg(f)} \Big)+ \frac{g+2}{N} \nonumber \\
&=  \sum_{\deg(f) \leq N} \frac{ \Lambda(f)  ( \chi(f)+ \overline{\chi(f)})(N-\deg(f))}{2 N |f|^{\frac{1}{2}+\frac{1}{N\log q}} \deg(f)} + \frac{g+2}{N}.
\end{align*}
\end{lem}

\begin{proof} The proof follows that of Proposition 4.3 from \cite{BFKRG}, by setting $z=0$, $N=h$
and using the fact that $m=g+2$.
\end{proof}

Since $\Lambda(f)=0$ unless $f$ is a prime power and $\Lambda(P^j)=\deg(P)$,  we have
\begin{eqnarray*}
\label{ineq-1} 
\log | {\textstyle L(\frac{1}{2},\chi) }|  &\leq& \frac{1}{2} \sum_{\deg(P) \leq N} \frac{( \chi(P) + \overline{\chi}(P))(N-\deg(P))}{N|P|^{\frac{1}{2}+\frac{1}{N \log q}}}  + \frac{g+2}{N}\\
&&+   \frac{1}{4}  \sum_{\deg(P) \leq N/2} \frac{( \chi(P) + \overline{\chi}(P))(N-2\deg(P))}{N|P|^{1+\frac{2}{N \log q}}}\\
&&+\sum_{l \geq 3} \frac{1}{2l}\sum_{\deg(P) \leq N/l} \frac{ (\chi(P)^l + \overline{\chi}(P)^l)(N-l \deg(P))}{N|P|^{\frac{l}{2}+\frac{l}{N \log q}}}.
\end{eqnarray*}
It is easy to see that the powers of primes with $l \geq 3$ contribute $O(1)$ to the expression above. More precisely, using the Prime Polynomial Theorem \eqref{ppt-bound}, we have
\begin{align*}
\sum_{l \geq 3} &\sum_{\deg(P) \leq N/l} \frac{ (\chi(P)^l + \overline{\chi}(P)^l)(N-l \deg(P))}{2lN|P|^{\frac{l}{2}+\frac{l}{N \log q}}} \leq \sum_{l=3}^N \sum_{j \leq N/l} \frac{q^j(N-lj)}{lj Nq^{lj(\frac{1}{2}+\frac{1}{N\log q})}}\\
& \leq \frac{1}{N} \sum_{h=3}^N \frac{N-h}{hq^{h(\frac{1}{2}+\frac{1}{N\log q})}} \sum_{\substack{j|h \\ j \leq h/3}} q^j \leq \frac{1}{N} \sum_{h=3}^{N} \frac{N-h}{hq^{h(\frac{1}{2}+\frac{1}{N \log q})}} q^{h/3} \tau(h) \leq 2 \sum_{h=3}^{N} \frac{1}{q^{h(\frac{1}{6}+\frac{1}{N\log q})} \sqrt{h}} \\
&\leq 2 \sum_{h=3}^{\infty} \frac{1}{5^{\frac{h}{6}} \sqrt{h}}=:\eta=1.676972\dots. 
\end{align*}

Then, for any $\kk > 0$, 
\begin{eqnarray}
\label{ineq} \label{chandee}
 | {\textstyle L(\frac{1}{2},\chi) }|^\kk &\leq& \exp \Big\{  \kk  \Re \Big( \sum_{\deg(P) \leq N} \frac{ \chi(P)(N-\deg(P))}{N|P|^{\frac{1}{2}+\frac{1}{N \log q}}}  \Big) + \frac{\kk (g+2)}{N}+
\kk \eta \nonumber\\
&&+   \kk \Re \Big(  \sum_{\deg(P) \leq N/2} \frac{\chi(P) (N-2\deg(P))}{2N|P|^{1+\frac{2}{N \log q }}} \Big) \Big\}. 
\end{eqnarray}

Similarly as in \cite{L-R}, we separate the sum over primes in $J+1$ sums over  the intervals 
\begin{equation} \label{intervals} I_0 = (0,(g+2) \theta_0], I_1= ((g+2)\theta_0 , (g+2)\theta_1], \ldots, I_\JJ = ((g+2) \theta_{\JJ-1}, (g+2) \theta_\JJ], \end{equation}
where for $0 \leq j \leq \JJ$, we define $$ \theta_j = \frac{e^j}{(\log g)^{1000}},\;\;  \ell_j = 2\left[ \theta_j^{-b}\right],$$ for some $0<b<1$.
In view of \eqref{square-free-sum}, it is natural to use $g+2$ instead of $g$ in the definition of the intervals $I_j$. 

We will choose $\JJ$ such that $\theta_\JJ$  is a small positive constant. We discuss in Section \ref{explicit-UB} explicit upper bounds and how to choose $\theta_\JJ$. We remark that for a given choice of $\theta_\JJ$, we have $\JJ = [ \log (\log g)^{{1000}} + \log{\theta_\JJ}]$. 
The power of 1000 together with the parameters chosen in Section \ref{explicit-UB} guarantee that $\JJ$ is positive for any $g\geq 3$. 

For each interval $I_j$, we define 
$$ P_{I_j} (\chi;u) = \sum_{P \in I_j} \frac{ a(P;u) \chi(P)}{\sqrt{|P|}},$$
where
$$a(P; u) = \frac{1}{|P|^{\frac{1}{(g+2) \theta_u \log q}}} \left( 1 - \frac{\deg{P}}{(g+2) \theta_u} \right),$$
for $0\leq u\leq \JJ$, and we extend this to a completely multiplicative function in the first variable. 
By $P \in I_j$, we always mean that $\deg{P} \in I_j$.

 In order to use Lemma \ref{like-chandee} we need bounds for $\exp\left( \Re P_{I_j}(\chi; u) \right)$ on each interval $I_j$. Let $t \in \R$ and $\ell$ be a positive even integer. Let
\begin{equation}\label{finite-taylor}
E_{\ell}(t) = \sum_{s \leq \ell} \frac{t^s}{s!} .
\end{equation} 
Note that $E_\ell(t) \geq 1$ if $t \geq 0$ and that $E_\ell(t) > 0$ since $\ell$ is even. We also have that for $t \leq \ell/e^2$,
\begin{equation}
e^t \leq (1+e^{-\ell/2}) E_{\ell}(t).
\label{ineq_e}
\end{equation}

Let $\nu(f)$ be the multiplicative function defined by $\nu(P^a) = \frac{1}{a!},$ and let $\nu_j(f)= ( \nu * \ldots * \nu)(f)$ be the $j$-fold convolution of $\nu$. We then have $\nu_j(P^a) = \frac{j^a}{a!}$. 

The following lemma gives a formula for the powers $\left( \Re P_{I_j}(\chi; u) \right)^s$, and will be used frequently in the paper.
\begin{lem} \label{combinatorics}
Let $a(f)$ be a completely multiplicative function from $\F_q[T]$ to $\C$, and let $I$ be some interval. Let $P_I := \sum_{P \in I} a(P)$. Then for any integer $s$, we have
\begin{eqnarray*}
P_I^s &=& s! \sum_{\substack{P|f \Rightarrow P \in I \\ \Omega(f)=s}} a(f) \nu(f) ,\\
\left( \Re P_I \right)^s &=& \frac{ s!}{2^s} \sum_{\substack{P|f h  \Rightarrow P \in I \\ \Omega(f h )=s}} a(f) \overline{a(h)} \nu(f) \nu(h).
\end{eqnarray*}
\end{lem}
\begin{proof}
We have
$$P_I^s=\sum_{\substack{P|f \Rightarrow P \in I\\ \Omega(f)=s}} a(f)  \sum_{P_1 \ldots P_s = f}1.$$
Note that if $f= Q_1^{\alpha_1} \ldots Q_r^{\alpha_r}$ then $s= \alpha_1+\ldots+\alpha_r$, and
$$ \sum_{P_1 \ldots P_s = f}1 = {s\choose \alpha_1} {s-\alpha_1 \choose \alpha_2} \ldots {s-\alpha_1-\ldots-\alpha_{r-1} \choose \alpha_r}= \frac{s!}{\alpha_1! \ldots \alpha_r!}  = s! \, \nu(f),$$ 
 so
$$P_I^s = s! \sum_{\substack{P|f \Rightarrow P \in I \\ \Omega(f)=s}} a(f) \nu(f)  .$$
We also have
\begin{align*}
 \Big(  \Re P_{I}\Big)^s &= \frac{1}{2^s} \sum_{r=0}^s {s \choose r} \sum_{\substack{P|f h  \Rightarrow P \in I_j \\ \Omega(f)=r \\ \Omega(h)=s-r}} a(f) \overline{a(h)}  \sum_{f = P_1 \ldots P_r} 1 \sum_{h= P_1 \ldots P_{s-r}} 1 \nonumber  \\
 &= \frac{1}{2^s} \sum_{r=0}^s {s \choose r} \sum_{\substack{P|f h \Rightarrow P \in I_j \\ \Omega(f)=r \\ \Omega(h)=s-r}}  a(f) a(h) r! \nu(f) (s-r)! \nu(h)  \nonumber \\
 &= \frac{ s!}{2^s} \sum_{\substack{P|f h  \Rightarrow P \in I_j \\ \Omega(f h )=s}} a(f) \overline{a(h)}  \nu(f) \nu(h). 
\end{align*}
\end{proof}

For $j \leq \JJ$, and for any real number $\kk \neq 0$, let
\begin{equation}
 D_{j, \kk} (\chi) = \prod_{r=0}^j (1+e^{-\ell_r/2})E_{\ell_r}( \kk \Re P_{I_r}(\chi;j)).
 \label{d_j}
 \end{equation}
We remark that the weights are $a(\cdot; j)$  for all intervals $I_0, \dots, I_j$ in the formula for $D_{j, \kk} (\chi)$.
 Note that we have
 \begin{eqnarray} \label{approximation-j}
 E_{\ell_r}( \kk \Re P_{I_r}(\chi;j))&= & \sum_{s \leq \ell_r} \frac{ (\kk \Re P_{I_r}(\chi;j))^s}{s!}\nonumber  \\ &=& \sum_{\substack{P|fh  \Rightarrow P \in I_r \\ \Omega(f h)\leq \ell_r}} \frac{(\kk/2)^{\Omega(f h)} a(f; j)a(h;j) \chi(f) \overline{\chi}(h) \nu(f) \nu(h)}{\sqrt{|f h |}},
 \label{elr}
 \end{eqnarray}
 where we have used Lemma \ref{combinatorics}. 

We also define the following term, which corresponds to the sum over the square of primes in equation \eqref{chandee}, 
\begin{eqnarray} \label{S_{j,\kk}}
S_{j,\kk}(\chi)=\exp  \left(\kk
\Re\left( \sum_{\deg(P) \leq (g+2) \theta_j/2} \frac{   \chi(P)b(P;j)}{|P|}\right)\right),\end{eqnarray}
where \begin{equation} \label{def-bp}
b(P;j)=\frac{1}{2|P|^{\frac{2}{(g+2)\theta_j\log q}}}\left(1-\frac{2\deg P}{(g+2) \theta_j}\right).\end{equation}

 \begin{prop} \label{prop-cases}
  Let $\kk$ be positive.
 For each $\chi$ a primitive cubic character of genus $g$, we either have
 $$ \max_{0\leq u \leq \JJ} |\Re P_{I_0}(\chi;u)| > \frac{\ell_0}{\kk e^2},$$ or
\begin{align*}
|{\textstyle L(\frac{1}{2},\chi)}|^\kk \leq&\exp(k(1/\theta_\JJ+\eta)) D_{\JJ,\kk}(\chi) S_{\JJ,\kk}(\chi)\\
 &+\sum_{\substack{0\leq j \leq \JJ-1 \\ j<u\leq \JJ}} \exp(\kk (1/\theta_j+\eta)) D_{j,\kk}(\chi)  S_{j,\kk}(\chi)\Big(\frac{e^2 \kk \Re P_{I_{j+1}}(\chi;u)}{\ell_{j+1}} \Big)^{s_{j+1}},
 \end{align*}
 for any $s_j$ even integers and $\eta=1.676972\dots$.
 \end{prop}

 \begin{proof}
 For $r = 0, 1, \ldots, \JJ$, let
 \begin{equation}
 \label{tr}
 \mathcal{T}_r = \Big\{ \chi \text{ primitive cubic, genus}(\chi)=g \, :  \, \max_{r \leq u \leq \JJ} | \Re P_{I_r}(\chi;u)| \leq \frac{\ell_r}{\kk e^2} \Big\}.
 \end{equation}
 
 For each $\chi$ we have one of the following:
 \begin{enumerate}
 \item $ \chi \notin \mathcal{T}_0$
 \item $\chi \in \mathcal{T}_r$ for each $r \leq \JJ$
 \item There exists a $j <\JJ$ such that $\chi \in \mathcal{T}_r$ for $r \leq j$ and $\chi \notin \mathcal{T}_{j+1}$.
 \end{enumerate}
 If the first condition is satisfied, then we are done. If not, assume that condition $(2)$ is satisfied. 
 Then in equation \eqref{ineq} we take $N= (g+2) \theta_\JJ$, and we get 
 \begin{align*}
 |{\textstyle L(\frac{1}{2},\chi)}|^{\kk}  \leq &\exp(\kk(1/\theta_J+\eta)) \prod_{j=0}^\JJ \exp ( \kk \Re P_{I_j}(\chi;\JJ)) S_{\JJ,\kk}(\chi)\\
 \leq & \exp(\kk(1/\theta_J+\eta)) D_{\JJ,\kk}(\chi) S_{\JJ,\kk}(\chi).
 \end{align*}

 Now assume that condition $(3)$ holds. {Then, there exists $j = j(\chi)$ and $u=u(\chi) > j=j(\chi)$   
 such that $|\Re P_{I_{j+1}} (\chi;u)| > \ell_{j+1}/(\kk e^2)$.}
We then have
 $$1 \leq \Big(  \frac{ \kk e^2 \Re P_{I_{j+1}}(\chi;u)}{\ell_{j+1}}\Big)^{s_{j+1}},$$ for any even integer $s_{j+1}$, and taking  $N= (g+2) \theta_j$   in equation \eqref{ineq}, we get 
 \begin{eqnarray*} |{\textstyle L(\frac{1}{2},\chi)}|^{\kk} &\leq& \exp(\kk (1/\theta_j+\eta)) D_{j,\kk}(\chi) 
 S_{j,\kk}(\chi)
 \Big(\frac{e^2 \kk \Re P_{I_{j+1}}(\chi;u)}{\ell_{j+1}} \Big)^{s_{j+1}}.
\end{eqnarray*}
Then, if $(3)$ holds, we have
\begin{equation}\label{eq:sum}
|{\textstyle L(\frac{1}{2},\chi)}|^\kk \leq \sum_{\substack{0\leq j \leq \JJ-1 \\ j<u\leq \JJ}} \exp(\kk(1/\theta_j+\eta)) D_{j,\kk}(\chi)  S_{j,\kk}(\chi)\Big(\frac{e^2 \kk \Re P_{I_{j+1}}(\chi;u)}{\ell_{j+1}} \Big)^{s_{j+1}},
 \end{equation}
where in the bound of the right-hand side, $j$ and $u$ are independent of $\chi$.
 \end{proof}

 \begin{rem}
  In the previous proof, we could have written $\max_{\substack{0\leq j \leq \JJ-1 \\ j<u\leq \JJ}}$ instead of $\sum_{\substack{0\leq j \leq \JJ-1 \\ j<u\leq \JJ}}$ in the bound \eqref{eq:sum}. However, this maximum depends on $\chi$, and in future applications of Proposition \ref{prop-cases} we  will need the right-hand side to be independent of $\chi$ so that we can exchange the bound with a sum over all the possible $\chi$. 
 \end{rem}

 \subsection{The mollifier} \label{mollifier}
 Let $\kappa$ be a positive real number, and we define 
 $${\textstyle M_j(\chi;\frac{1}{\kappa}) }:={\textstyle E_{\ell_j}(-\frac{1}{\kappa}P_{I_j}(\chi;\JJ))}= \sum_{\substack{P |f \Rightarrow P \in I_j \\ \Omega(f) \leq \ell_j}} \frac{a(f;\JJ)  \chi(f)  \lambda(f)\nu(f)}{\kappa^{\Omega(f)}\sqrt{|f|}},$$ where $\lambda(f)$ is the Liouville function. 
  We also define
 $${\textstyle M(\chi;\frac{1}{\kappa})} = \prod_{j=0}^\JJ {\textstyle M_j(\chi;\frac{1}{\kappa}) }.$$

 We have for any positive integer $n$
 $${\textstyle M_j(\chi;\frac{1}{\kappa}) }^{n}  = \sum_{\substack{P | f \Rightarrow P \in I_j \\ \Omega(f) \leq n \ell_j}} \frac{a(f;\JJ) \chi(f) \lambda(f)}{\kappa^{\Omega(f)}\sqrt{|f|}} \nu_{n}(f;\ell_j),$$ where
 $$ \nu_{n}(f;\ell_j) = \sum_{\substack{f= f_1 \cdot \ldots \cdot f_{n} \\ \Omega(f_1) \leq \ell_j,\ldots \Omega(f_{n}) \leq \ell_j}} \nu(f_1) \ldots \nu(f_{n}).$$
 Then, taking $\kappa$ such that $\kk \kappa $ is an even integer, 
 \begin{equation} {\textstyle |M_j(\chi;\frac{1}{\kappa})| }^{\kk \kappa} = \sum_{\substack{P | f_j h_j \Rightarrow P \in I_j \\ \Omega(f_j) \leq \frac{\kk \kappa}{2} \ell_j, \Omega(h_j) \leq \frac{\kk \kappa  }{2} \ell_j}} \frac{ a(f_j;\JJ) a(h_j;\JJ) \chi(f_j) \overline{\chi}(h_j)  \lambda(f_j) \lambda(h_j)}{{\kappa^{\Omega(f_j h_j)} }\sqrt{|f_jh_j|}} \nu_{\kk \kappa /2}(f_j; \ell_j) \nu_{\kk \kappa /2}(h_j;\ell_j).
 \label{mj}
 \end{equation}
We remark that the mollifier should be a Dirichlet polynomial approximating  $|\textstyle{L(\frac{1}{2}, \chi)|}^{-1}$. Indeed, taking $\kappa=1$ and $\kk$ an even integer in the definition above, we see that $|M_j (\chi; 1)|^\kk$ is a Dirichlet polynomial on the interval $I_j$ in view of \eqref{chandee}, approximating the exponential with the finite sum $E_{\ell_j}$ on each interval
(we do not claim that the finite sum is an upper bound, but it is close enough to work on average), i.e.  it is very close to \eqref{approximation-j} (with the added  Liouville function taking care of the cancellation).
 Taking $\kappa \neq 1$ allows us to mollify all moments, not just in the case when $\kk$ is an even integer, by taking the mollifier to be
${\textstyle |M_j(\chi;\frac{1}{\kappa})| }^{\kk \kappa }$ on each interval $I_j$ as defined above. We remark that for any $\kappa$, the term with $f_j h_j = P$ in the Dirichlet series
\eqref{mj} is
$$
- \frac{\kk}{2} \frac{a(P, \JJ) \left( \chi(P) + \overline{\chi}(P) \right)}{|P|^{1/2}}
$$
which is independent of $\kappa$, and of the correct size to approximate 
$|{\textstyle L(\frac{1}{2}, \chi)} |^{-\kk}$.

 \subsection{Averages over the family} 
  
 \begin{lem} \label{mt} \label{LZ-Lemma4.1s}
Let $I_0,I_1,\ldots, I_\JJ$ be intervals such that $I_0=(0,(g+2) \theta_0], I_1=((g+2) \theta_0,(g+2)\theta_1],\ldots I_\JJ=((g+2) \theta_{\JJ-1},(g+2) \theta_\JJ]$. Let $B, C, b$, and $c$ be any functions supported on $\mathbb{F}_q[T]$. Suppose $s_j$ and $\ell_{j}$ are nonnegative integers for $j=0,\dots,\JJ$ such that 
\begin{equation}\label{nicecondition}
2 \sum_{j=0}^\JJ \theta_j s_{j}+ 3 \sum_{j=0}^\JJ \theta_j \ell_{j} \leq 1/2.
\end{equation}
Then we have
 \begin{align*}
 &\sum_{R \in \mathcal{M}_{q^2,g/2+1}} \prod_{j=0}^\JJ \sum_{\substack{ P\mid F_jH_j \Rightarrow P\in I_j\\ P|  f_j h_j \Rightarrow P \in I_j \\
 \Omega(F_jH_j) \leq  s_j \\ \Omega(f_{j}) \leq \ell_{j} \\ \Omega(h_{j}) \leq \ell_{j}}} B(F_j) C(H_j) b(f_j) c(h_j) \chi_R(F_j H_j^2f_jh_j^2) \\
 =& q^{g+2} \prod_{j=0}^\JJ \sum_{\substack{ P\mid F_jH_j \Rightarrow P\in I_j\\ P|  f_j h_j \Rightarrow P \in I_j \\
 \Omega(F_jH_j) \leq  s_j \\ \Omega(f_{j}) \leq \ell_{j} \\ \Omega(h_{j}) \leq \ell_{j}\\F_jH_j^2 f_jh_j^2 = \tinycube}}  {B(F_j) C(H_j) b(f_{j})c(h_{j})} \frac{\phi_{q^2}(F_j H_j^2f_jh_j^2)}{|F_j H_j^2f_jh_j^2|_{q^2}}.
 \end{align*}
\end{lem}

\kommentar{\ccom{I would like to change the notation to avoid things like $\chi_F(F_i)$. Maybe $\chi_D(F_i)$? Or it looks too quadratic?}
\acom{I think $\chi_D$ is ok. Or maybe $\chi_R$.} \mcom{$R$ for the win! Please remove comments if you agree with the changes} \acom{I agree with the changes. Another question: in the statement we say: Fix $0 \leq r \leq \JJ$. I suppose that's forgotten from before? We don't seem to use anything about $r$.} \mcom{We should delete it. It was there because of the applications in the proof of Lemma 4.1, we were trying to write the most general version, in the end we decided to write Remark 3.6 instead.}}
\begin{rem}
 We will also use the above lemma in slightly different cases. We will allow the following variations: 
 \begin{itemize}
  \item That the condition  $\Omega(F_jH_j) \leq  s_j $ be replaced by $\Omega(F_jH_j) =  s_j$. 
\item That the condition $ P\mid F_jH_j \Rightarrow P\in I_j$ be replaced by  $P\mid F_jH_j \Rightarrow \deg(P)=m_j$ where $m_j$ is a fixed element in $I_j$. 
  \end{itemize}
  The above variations will happen for some values of $j$ and may happen both at the same time.  In all cases, the results are analogous. 
  
\end{rem}

 \begin{proof}
 Expanding the left-hand side of the identity above and exchanging the order of summation, we need to evaluate sums of the form
 $$ \sum_{R \in \mathcal{M}_{q^2,g/2+1}} \chi_R\left(\prod_{j=0}^\JJ F_jH_j^2f_jh_j^2\right)  =
\sum_{R \in \mathcal{M}_{q^2,g/2+1}} \chi_{\prod_{j=0}^\JJ F_jH_j^2f_jh_j^2}(R),$$
since $q^2 \equiv 1 \pmod 6$, and we have cubic reciprocity over $\F_{q^2}[T]$.
 If $\prod_{j=0}^\JJ F_jH_j^2f_jh_j^2 \neq \cube$, then 
 \[\sum_{j=0}^\JJ\deg(F_jH_j^2f_jh_j^2)\leq   (g+2)\left( 2\sum_{j=0}^\JJ\theta_j s_j+3\sum_{j=0}^\JJ  \theta_j  \ell_j\right)\leq (g+2)/2 = \deg(R)\]
and  the character sum above vanishes. We are then left with the contribution of those terms with $\prod_{j=0}^\JJ F_jH_j^2f_jh_j^2 = \cube$. Since $F_jH_j^2f_jh_j^2$ is only divisible by primes in $I_j$ and the intervals $I_j$ are disjoint, it follows that we must have $F_jH_j^2f_jh_j^2 = \cube$ for each $j\leq \JJ $. 
For any $c \in \F_{q^2}[T]$, $c = \cube$,  and $\deg{c} \leq g/2 + 1$, we have that
$$
 \sum_{R \in \mathcal{M}_{q^2,g/2+1}} \chi_R( c) = \sum_{\substack{d \in \mathcal{M}_{q^2} \\ d \mid c}} \mu(d)  \sum_{\substack {R \in \mathcal{M}_{q^2,g/2+1}\\ d \mid R}} 1
= q^{g + 2} \sum_{\substack{d \in \mathcal{M}_{q^2} \\ d \mid c}} \mu(d) q^{-2 \deg{d}} = q^{g+2} \frac{\phi_{q^2}(c)}{|c|_{q^2}},
$$
and using the above with $c = \prod_{j=0}^\JJ  F_jH_j^2f_jh_j^2$, the conclusion follows.
 \end{proof}

 \subsection{Proof of Theorem \ref{moll_ub}}

\begin{proof}
 We write
 \begin{equation} \label{2-terms}
   \sum_{\chi\in \mathcal{C}(g) } |{\textstyle L(\frac{1}{2},\chi) }|^{\kk} |{\textstyle M(\chi;\frac{1}{\kappa})}|^{\kk \kappa}   =  \sum_{\chi\in \mathcal{C}(g)\cap \mathcal{T}_0} |{\textstyle L(\frac{1}{2},\chi)}|^{\kk} |{\textstyle M(\chi;\frac{1}{\kappa})}|^{\kk \kappa} +  \sum_{\chi \in \mathcal{C}(g)\setminus \mathcal{T}_0} |{\textstyle L(\frac{1}{2},\chi) }|^{\kk} |{\textstyle M(\chi;\frac{1}{\kappa})}|^{\kk \kappa},
   \end{equation}
   where $\mathcal{T}_r$ is defined in \eqref{tr}. 
 We first focus on the second term above. Since $\chi \not\in \mathcal{T}_0$, there exists $u=u(\chi)$ such that $0 \leq u \leq \JJ $ and
$$ |\Re P_{I_0}(\chi;u)| > \frac{\ell_0}{\kk e^2}. $$ 
Choosing  $s_0$ even and multiplying by $ \Big(\frac{ke^2 \Re P_{I_0}(\chi;u)}{\ell_0} \Big)^{s_0} >1$, 
 completing the sum for all $\chi \in \mathcal{C}(g)$ (since all the involved terms are positive), and applying Cauchy--Schwarz, we obtain,
 \begin{align} \nonumber
 \sum_{u=0}^\JJ  \sum_{\substack{\chi \in \mathcal{C}(g)\setminus \mathcal{T}_0\\ u(\chi)=u}}& |{\textstyle L(\frac{1}{2},\chi)}|^{\kk}  |{\textstyle M(\chi;\frac{1}{\kappa})}|^{ \kk \kappa}
 \leq \sum_{u=0}^\JJ \sum_{\chi\in \mathcal{C}(g)} |{\textstyle L(\frac{1}{2},\chi)}|^{\kk}  |{\textstyle M(\chi;\frac{1}{\kappa})}|^{\kk \kappa} \Big( \frac{\kk e^2 \Re P_{I_0}(\chi;u)}{\ell_0} \Big)^{s_0} \\ \label{t2}
  & \leq \JJ ^{1/2} \Big( \sum_{\chi \in \mathcal{C}(g)} |{\textstyle L(\frac{1}{2},\chi)}|^{2 \kk} \Big)^{1/2} \Big( \sum_{u=0}^\JJ  \sum_{\chi \in \mathcal{C}(g)} \Big( \frac{ \kk e^2}{\ell_0} \Big)^{2s_0} |{\textstyle M(\chi;\frac{1}{\kappa})}|^{2\kk \kappa } ( \Re P_{I_0}(\chi;u))^{2s_0}   \Big)^{1/2},
 \end{align}
where we choose $s_0$ to be an even integer such that $$a \kk \kappa  \ell_0 \leq s_0 \leq \frac{1}{d \theta_0},$$ with $a$ and $d$ as in Lemma \ref{estimates}.  

 For the first sum in \eqref{t2}, we have an upper bound of size $q^{g/2} g^{O(1)}$ using Lemma \ref{lemma_sound}.
 We aim to obtain some saving from the second sum above. Using Lemma \ref{estimates}(i) and Stirling's formula, we get for $c = 2-4/a$ 
 \begin{eqnarray*}
&&\Big( \sum_{u=0}^\JJ  \sum_{\chi\in \mathcal{C}(g)} \Big( \frac{ \kk e^2}{\ell_0} \Big)^{2s_0} |{\textstyle M(\chi;\frac{1}{\kappa})}|^{2\kk \kappa } ( \Re P_{I_0}(\chi;u))^{2s_0} \Big)^{1/2}  \\ &&\hspace{1cm} \ll q^{g/2} g^{O(1)}  \Big( \frac{\kk e^{1+c/6} \theta_0^{b} s_0^{1-c/6} 5^{c/6}  }{2c^{c/6}} \Big)^{ s_0}   (\log g)^{s_0/2}\\  &&\hspace{1cm} \ll
\frac{q^{g/2}}{q^{(\log g)^\delta}} \ll \frac{q^{g/2}}{g^A} 
\end{eqnarray*}
for $\delta > 1$ and all $A \geq 1$, 
where the last line is obtained by setting $s_0 = 2[a \kk  \kappa \ell_0/2]+2$. 
We also used the bound \eqref{hr} for $H(0)$ from Lemma  \ref{estimates}(i) in the second line. 
Replacing the two estimates in \eqref{t2}, we get that
$$
  \sum_{\chi \in \mathcal{C}(g)\setminus \mathcal{T}_0} |{\textstyle L(\frac{1}{2},\chi) }|^{\kk} |{\textstyle M(\chi;\frac{1}{\kappa})}|^{\kk \kappa}
  = o \left( q^{g} \right),
$$
and the sum over the characters $\chi \not\in \mathcal{T}_0$ does not contribute to the sharp upper bound.

For the first sum of \eqref{2-terms} over the characters $\chi \in \mathcal{T}_0$, we use Proposition \ref{prop-cases}. As before, we first bound the sum by the completed sum over all $\chi \in \mathcal{C}(g)$ since all the extra terms are positive. We have
 \begin{align}
& \sum_{\chi \in \mathcal{C}(g)\cap \mathcal{T}_0} |{\textstyle L(\frac{1}{2},\chi)}|^{\kk} |{\textstyle M(\chi;\frac{1}{\kappa})}|^{\kk \kappa} \nonumber \\
\leq &\exp \left( \kk \left(1/\theta_J + \eta \right) \right) \sum_{\chi\in \mathcal{C}(g)} D_{\JJ ,  \kk}(\chi)S_{\JJ ,  \kk}(\chi)|{\textstyle M(\chi;\frac{1}{\kappa})}|^{\kk \kappa} \nonumber \\  
&+\sum_{\substack{{0\leq j \leq \JJ -1}\\ j<u \leq \JJ }} \exp \left( \kk \left(1/\theta_j + \eta \right) \right)  \sum_{\chi\in \mathcal{C}(g)}   D_{j, \kk}(\chi)  S_{j, \kk}(\chi) \Big( \frac{ \kk e^2 \Re P_{I_{j+1}}(\chi;u)}{\ell_{j+1}} \Big)^{s_{j+1}} |{\textstyle M(\chi;\frac{1}{\kappa})}|^{ \kk \kappa }. \label{tb}
 \end{align}
 where $s_{j+1}$ is even. 

  Using Cauchy--Schwarz, we write
 \begin{equation}
  \sum_{\chi \in \mathcal{C}(g)} D_{\JJ ,  \kk}(\chi)S_{\JJ ,  \kk}(\chi)|{\textstyle M(\chi;\frac{1}{\kappa})}|^{\kk \kappa} \leq \Big(  \sum_{\chi \in \mathcal{C}(g)} D_{\JJ ,  \kk}(\chi)^2 |{\textstyle M(\chi;\frac{1}{\kappa})}|^{2\kk \kappa} \Big)^{1/2} \Big(
   \sum_{\chi \in \mathcal{C}(g)}S_{\JJ ,  \kk}(\chi)^2 \Big)^{1/2} ,
   \label{i1}
   \end{equation}
   and similarly 
 \begin{eqnarray} \nonumber
&&   \sum_{\chi \in \mathcal{C}(g)}D_{j,\kk}(\chi)S_{j,\kk}(\chi) \Big( \frac{ \kk e^2 \Re P_{I_{j+1}}(\chi;u)}{\ell_{j+1}} \Big)^{s_{j+1}} |{\textstyle M(\chi;\frac{1}{\kappa})}|^{\kk \kappa}\\
&& \leq \Big( \sum_{\chi \in \mathcal{C}(g)} D_{j,\kk}(\chi)^2 \Big( \frac{ \kk e^2 \Re P_{I_{j+1}}(\chi;u)}{\ell_{j+1}} \Big)^{2s_{j+1}}
 |{\textstyle M(\chi;\frac{1}{\kappa})}|^{2\kk \kappa} \Big)^{1/2} \Big(
   \sum_{\chi \in \mathcal{C}(g)}S_{j,\kk}(\chi)^2 \Big)^{1/2} .
   \label{i2}
   \end{eqnarray}

 To bound equation \eqref{i1},  we use Lemma \ref{estimates}(ii) and Lemma \ref{estimates-squares} which give
 \begin{align} \label{easy-part}
  \sum_{\chi \in \mathcal{C}(g)}D_{\JJ ,  \kk}(\chi)S_{\JJ ,  \kk}(\chi)|{\textstyle M(\chi;\frac{1}{\kappa})}|^{\kk \kappa} 
  \leq q^{g+2} \mathcal{D}_k^{1/2} \mathcal{S}_k^{1/2} \exp{(k^2)},
 \end{align}
 where the constants $\mathcal{D}_k,  \mathcal{S}_k$ come from Lemmas \ref{estimates}  and \ref{lemma-square-primes} respectively.
 
 Similarly, to bound \eqref{i2}, we use Lemmas \ref{estimates}(iii) and \ref{estimates-squares}. 
 When we bound the first term in \eqref{i2} with  Lemma \ref{estimates}(iii), we use Stirling's formula and note that the sum over primes is bounded by $\log( \theta_{j+1}/\theta_j)=1$. Now we pick $s_{j+1} = 2 [1/(2d\theta_{j+1})]$, and then when $g \to \infty$, we have
 
 \begin{align}
 & \sum_{\chi \in \mathcal{C}(g)}D_{j,\kk}(\chi)^2  \Big( \frac{ \kk e^2 \Re P_{I_{j+1}}(\chi;u)}{\ell_{j+1}} \Big)^{2s_{j+1}} |{\textstyle M(\chi;\frac{1}{\kappa})}|^{2\kk \kappa} \nonumber \\
  &\leq_\varepsilon  \frac{2\sqrt{6}}{\sqrt{c}}  q^{g+2} \mathcal{D}_k \exp (3\kk^2+2\kk ) e^{k^2(J-j-1)} \left( \frac{k^2 e^{2+c/3} \theta_{j+1}^{2b} s_{j+1}^{2-c/3} 5^{c/3}}
  {4 c^{c/3} }    \right)^{s_{j+1}} \nonumber \\
  &=\frac{2 \sqrt{6}}{\sqrt{c}} q^{g+2} \mathcal{D}_k \exp (3\kk^2+2\kk )
  \times \exp \Big(\kk^2(\JJ -j-1)+ \frac{ \alpha \log \theta_{j+1}}{ d \theta_{j+1}}+ \frac{\log F}{d \theta_{j+1}} \Big), \label{dj}
 \end{align}
 where $$\alpha=2b-2+\frac{c}{3}, \, F= \frac{\kk^2e^{2+c/3}5^{c/3}}{4d^{2-c/3}c^{c/3}},$$ with $c=2-4/a$ and $a$ and $d$ as in Lemma \ref{estimates}.
 
We now replace in \eqref{i2}, and using Lemma \ref{estimates-squares}, the sum over $j,u,\chi$  in \eqref{tb}, is bounded by
$$
\exp(k \eta) \left( \frac{2 \sqrt{6}}{\sqrt{c}} \mathcal{D}_k \exp (3\kk^2+2\kk ) \mathcal{S}_k  \right)^{1/2} \;C_J \;q^{g+2},
$$
where
      \begin{align}
       C_J := \sum_{\substack{0 \leq j \leq \JJ -1 \\ j<u \leq \JJ }} & \exp \Big(\frac{\kk}{\theta_j} +\frac{\kk^2(\JJ -j-1)}{2}+\frac{\alpha \log \theta_{j+1}}{ 2d \theta_{j+1}}+ \frac{\log F}{2d \theta_{j+1}} \Big) \nonumber \\
       &= \sum_{0 \leq j\leq \JJ -1} (\JJ -j) \exp \Big(\frac{\kk}{\theta_j} +\frac{\kk^2(\JJ -j-1)}{2}+\frac{ \alpha \log \theta_{j+1}}{ 2d \theta_{j+1}}+ \frac{\log F}{2d \theta_{j+1}} \Big) \nonumber \\
       &= \sum_{0 \leq u \leq \JJ -1}  (u+1) \exp \Big(  \frac{\kk e^{u+1}}{\theta_\JJ }+\frac{\kk^2u}{2} - \frac{\alpha ue^u}{2d\theta_\JJ }+\frac{ \alpha e^u \log \theta_\JJ }{2d \theta_\JJ }+\frac{e^u \log F}{2d \theta_\JJ } \Big) \nonumber  \\
       &= O(1) . \label{tobound2} 
       \end{align} 
Now using also \eqref{easy-part}, and the fact that the characters in $\mathcal{T}_0$ do not contribute to the upper bound, we finally have
\begin{align} \label{before-explicit}
\sum_{\substack{ \chi \in  \mathcal{C}(g) }} |{\textstyle L(\frac{1}{2},\chi)}|^{\kk} |{\textstyle M(\chi;\frac{1}{\kappa})}|^{\kk \kappa} \leq_\varepsilon 
\mathcal{D}_k^{1/2} \mathcal{S}_k^{1/2} \exp{\left( \frac{3}{2} k^2 + (1+\eta) k \right)}\left( \exp(k/\theta_J) +  \sqrt[4]{\frac{24}{c}}  \;C_J  \right) \;q^{g+2}.
\end{align}
This completes the proof of Theorem \ref{second-moment}. In Section \ref{explicit-UB},
 we find an explicit numerical value for the constant in the upper bound \eqref{before-explicit} when $k=2$, which depends on the bound for $C_J$.

\end{proof}

\section{A technical lemma}
\label{friendly-lemma}
 \begin{lem}
 \label{estimates}
 Let $j=0,\ldots,\JJ -1$, $0 \leq u \leq \JJ $ for (i), 
 and $j < u \leq \JJ $ for (iii).  Let $s_j$ be an integer with $a \kk \kappa  \ell_j \leq s_j \leq \frac{1}{d \theta_j},$ where $a$ and $d$  are such that $a >2$, $d>8$, and $4ad \theta_\JJ ^{1-b} \leq 1$, with $0<b<1$. 
 
 Then
 we have
\begin{align*}
 \text{(i)} \sum_{\chi \in \mathcal{C}(g)}& |{\textstyle M(\chi;\frac{1}{\kappa})}|^{2 \kk \kappa} ( \Re P_{I_0}(\chi;u))^{2s_0} \leq_{\varepsilon} 2 q^{g+2} e^{\kk^2\JJ } H(0) \Big( \sum_{P \in I_0} \frac{1}{|P|} \Big)^{s_0}  \frac{\left(\frac{5}{3}\right)^{(2-4/a)s_0/3}(2s_{0})!}{ 4^{s_{0}} \lfloor\frac{(2-4/a)s_{0}}{3}\rfloor!}, \\
\text{(ii)} \sum_{\chi \in \mathcal{C}(g)} &  D_{\JJ ,\kk}(\chi)^2|{\textstyle M(\chi;\frac{1}{\kappa})}|^{2  \kk \kappa} 
\leq_{\varepsilon} q^{g+2} \mathcal{D}_k \exp(2\kk^2 ) ,
\\
\text{(iii)}\sum_{\chi \in \mathcal{C}(g)}  & D_{j, \kk}(\chi)^2 \Big( \Re P_{I_{j+1}}(\chi;u) \Big)^{2s_{j+1}} |{\textstyle M(\chi;\frac{1}{\kappa})}|^{2  \kk  \kappa} \leq_{\varepsilon} 2  q^{g+2} \mathcal{D}_k \exp(3\kk^2  +2\kk) \\
& \times  e^{\kk^2(\JJ -j-1)} \Big( \sum_{P \in I_{j+1}} \frac{1}{|P|} \Big)^{s_{j+1}}   \frac{\left(\frac{5}{3}\right)^{(2-4/a)s_{j+1}/3}(2s_{j+1})!}{ 4^{s_{j+1}} \lfloor\frac{(2-4/a)s_{j+1}}{3}\rfloor!} ,
\end{align*}
where $\mathcal{D}_k$ is given in \eqref{c01} and where $H(0)$ is bounded by \eqref{h0}. 
 \end{lem}

 \begin{proof}
Following \cite{BSM, DFL}, the sum over $\chi \in \mathcal{C}(g)$ can be rewritten as the sum over the cubic residue symbols $\chi_R$, for monic square-free polynomials $R \in \F_{q^2}[T]$ of degree ${g/2+1}$, with the property that if $P|R$ then $P \notin \mathbb{F}_q[T]$. Since all the summands in the expressions above are positive, we first bound the sums over $\chi \in \mathcal{C}(g)$ 
by the sum over all $R \in \mathcal{M}_{q^2,g/2+1}$.

We prove the last upper bound, and the first two are just simpler cases of that one. 
We note that $D_{j, \kk}(\chi)^2$ contributes primes from the intervals $I_0, \dots, I_j$, $\Re P_{I_{j+1}}(\chi;u)$ contributes primes from $I_{j+1}$ and the mollifier contributes primes from all the intervals $I_0, \dots, I_\JJ $. 
 To prove $\text{(iii)}$, we have to bound 
\begin{equation}\label{before-average} \sum_{R \in \mathcal{M}_{q^2,g/2+1}}
\prod_{r=0}^j  (1+e^{-\ell_r/2})^2 \mathcal{E}_R(r) \times \mathcal{E}_R(j+1)\times \prod_{r=j+2}^\JJ  \mathcal{E}_R(r),\end{equation}
where the $\mathcal{E}_R(r)$ are defined as follows. For $r=0,\dots, j$,
\begin{align*}
\mathcal{E}_R(r)&=
\sum_{\substack{P|f_r h_r F_{r1}F_{r2} H_{r1} H_{r2} \Rightarrow P \in I_r \\ \Omega(F_{r1} H_{r1}) \leq \ell_r \\\Omega(F_{r2} H_{r2}) \leq \ell_r \\  \Omega(f_r) \leq (\kk \cdot \kappa)\ell_r \\ \Omega(h_r) \leq (\kk  \cdot \kappa) \ell_r }} \frac{(\kk/2)^{\Omega(F_{r1} F_{r2}  H_{r1} H_{r2})}
a(F_{r1} F_{r2} H_{r1} H_{r2}; j) \nu(F_{r1}) \nu(F_{r2}) \nu(H_{r1}) \nu(H_{r2})  }{ \kappa^{\Omega(f_r h_r)} \sqrt{  |f_r h_r F_{r1}F_{r2} H_{r1} H_{r2}|}} \\
& \times a (f_r h_r ; \JJ ) \lambda(f_r h_r) \nu_{\kk \kappa} (f_r;\ell_r) \nu_{\kk \kappa}(h_r;\ell_r)\chi_R(f_rh_r^2F_{r1} F_{r2}H_{r1}^2H_{r2}^2).
\end{align*}

For $r=j+1$,
\begin{eqnarray*}
\mathcal{E}_R(r) = \frac{(2s_r)!}{4^{s_r}}
\sum_{\substack{P|f_r h_r F_r H_r\Rightarrow P \in I_r \\ \Omega(F_r H_r) = 2s_r \\ \Omega(f_r) \leq (\kk \cdot \kappa)\ell_r \\ \Omega(h_r) \leq ( \kk \cdot \kappa) \ell_r }}
\frac{a(F_r H_r; u) \nu(F_r) \nu(H_r) a (f_r h_r ; \JJ ) \lambda(f_r h_r) \nu_{\kk \kappa} (f_r;\ell_r) \nu_{\kk \kappa }(h_r;\ell_r) \chi_R(f_rh_r^2F_rH_r^2)}{ \kappa^{\Omega(f_r h_r)} \sqrt{  |f_r F_r h_r H_r|}},
\end{eqnarray*}

For $r=j+2,\dots, \JJ $,
\begin{eqnarray*}
\mathcal{E}_R(r)=
\sum_{\substack{P|f_r h_r \Rightarrow P \in I_r \\ \Omega(f_r) \leq (\kk \cdot \kappa)\ell_r \\ \Omega(h_r) \leq ( \kk  \cdot \kappa) \ell_r }} \frac{
a (f_r h_r ; \JJ ) \lambda(f_r h_r) \nu_{\kk \kappa} (f_r;\ell_r) \nu_{\kk \kappa}(h_r;\ell_r) \chi_R(f_rh_r^2)}{ \kappa^{\Omega(f_r h_r)} \sqrt{  |f_r h_r|}}.
\end{eqnarray*}

For $\theta_\JJ $ small enough (depending of $d, \kk, \kappa$), note that we can apply Lemma \ref{mt} to evaluate \eqref{before-average} because from our choice of parameters, we have, for any $j \leq \JJ -1$, that
 \begin{equation}
  4\sum_{r \leq j}  \theta_r\ell_r  + 4\theta_{j+1}  s_{j+1} + 3 \sum_{r=0}^\JJ  \theta_r \kk \kappa \ell_r \leq 1/2.
 \label{important_condition}
  \end{equation} 
We then obtain that \eqref{before-average} is bounded by
\begin{eqnarray} \label{after-lemma}
q^{g+2} \left( \prod_{r=0}^j  \left(1 + e^{-\ell_r/2} \right)^2  E(r) \times E({j+1}) \times \prod_{r=j+2}^\JJ  E(r) \right),
\end{eqnarray}
where the $E(r)$ are the factors obtained after doing the average over $R$ from Lemma \ref{mt}.
We proceed to address the three cases, depending on the value of $r$. 

For  $r=0, \dots, j$, we have
\begin{align*}
E(r)& =
\sum_{\substack{P|f_r h_r F_{r1}F_{r2} H_{r1} H_{r2} \Rightarrow P \in I_r \\ \Omega(F_{r1} H_{r1}) \leq \ell_r \\\Omega(F_{r2} H_{r2}) \leq \ell_r \\  \Omega(f_r) \leq (\kk \cdot \kappa) \ell_r \\ \Omega(h_r) \leq (\kk  \cdot \kappa) \ell_r \\ f_rh_r^2F_{r1} F_{r2} H_{r1}^2 H_{r2}^2= \tinycube }} \frac{(\kk/2)^{\Omega(F_{r1} F_{r2}  H_{r1} H_{r2})}
a(F_{r1} F_{r2} H_{r1} H_{r2}; j) \nu(F_{r1}) \nu(F_{r2}) \nu(H_{r1}) \nu(H_{r2})  }{ \kappa^{\Omega(f_r h_r)} \sqrt{  |f_r h_r F_{r1}F_{r2} H_{r1} H_{r2}|}} \\
& \times a (f_r h_r ; \JJ ) \lambda(f_r h_r) \nu_{\kk \kappa} (f_r;\ell_r) \nu_{\kk \kappa}(h_r;\ell_r)\frac{\phi_{q^2}(f_rh_r^2F_{r1} F_{r2}H_{r1}^2H_{r2}^2)}{| f_rh_r^2F_{r1} F_{r2}H_{r1}^2H_{r2}^2|_{q^2}}.
 \end{align*}

Notice that if $\max\{\Omega(f_r),\Omega(h_r),\Omega(F_{r1}H_{r1}), \Omega(F_{r2} H_{r2})\}\geq \ell_r$, we have $2^{\Omega(f_rh_rF_{r1}H_{r1} F_{r2} H_{r2})}\geq 2^{\ell_r}$. We write $F_r=F_{r1} F_{r2}$,  $H_r=H_{r1}H_{r2}$, and we recall
that $\nu_2(F_r) = (\nu * \nu)(F_r)$. 
We have 
\begin{align}
E(r) &\leq 
\sum_{\substack{P|f_r h_r F_r H_r\Rightarrow P \in I_r \\ f_r h_r^2 F_r H_r^2=\tinycube}} \frac{(\kk/2)^{\Omega(F_r H_r)}
 a(F_r H_r; j)  \nu_2(F_r) \nu_2(H_r) a (f_r h_r ; \JJ )  \lambda(f_r h_r) \nu_{\kk \kappa} (f_r) \nu_{\kk \kappa}(h_r) \phi_{q^2} (f_r h_r^2 F_r H_r^2)}{ \kappa^{\Omega(f_r h_r)} \sqrt{  |f_r F_r h_r H_r|} |f_r h_r^2 F_r H_r^2|_{q^2}} \nonumber\\
&\qquad +\frac{1}{2^{\ell_r}}\sum_{\substack{P|f_r h_r F_r H_r\Rightarrow P \in I_r \\ f_r h_r^2 F_r H_r^2=\tinycube}} \frac{ 2^{\Omega(f_r h_rF_rH_r)}(\kk/2)^{\Omega(F_r H_r)} \nu_2(F_r) \nu_2(H_r) \nu_{\kk \kappa}(f_r) \nu_{\kk \kappa}(h_r)}{\kappa^{\Omega(f_r h_r)} \sqrt{|f_r h_rF_rH_r|}},
\label{eq:errorterm}
\end{align}
where we have used the bounds $\phi_{q^2}(f_r h_r^2 F_r H_r^2)/|f_r h_r^2 F_r H_r^2|_{q^2}, \lambda(f_r h_r), \nu(F_r), \nu(H_r) \leq 1$, $ a(F_rH_r;j), a(f_r h_r; \JJ ) \leq 1$ in the second term above. Now using that  
$\nu_{\kk \kappa }(f_r) \leq (\kk \kappa)^{\Omega(f_r)}$ and $\nu_2(F_r) \leq 2^{\Omega(F_r)}$ we get that the second term in \eqref{eq:errorterm} is
\begin{equation}\label{eq:errorterm2}
\leq \frac{1}{2^{\ell_r}}\sum_{\substack{P|f_r h_r F_r H_r\Rightarrow P \in I_r \\ f_r h_r^2 F_r H_r^2=\tinycube}} \frac{ (2\kk)^{\Omega(f_rh_rF_rH_r)}}{\sqrt{|f_rh_rF_rH_r|}}.
\end{equation}

Now write $(f_r,h_r)=X$ and $(F_r,H_r)=Y$ and let $f_r=f_{r,0}X$, $h_r=h_{r,0}X$, $F_r=F_{r,0}Y$, and $H_r=H_{r,0}Y$. Then $f_{r,0}h_{r,0}^2F_{r,0}H_{r,0}^2=\cube$. Write $(f_{r,0},H_{r,0})=S$ and  $(h_{r,0},F_{r,0})=T$, and write $f_{r,0}=f_{r,1}S$, $h_{r,0}=h_{r,1}T$, $F_{r,0}=F_{r,1}T$, and $H_{r,0}=H_{r,1}S$. Then $f_{r,1}h_{r,1}^2F_{r,1}H_{r,1}^2=\cube$ with $(f_{r,1}F_{r,1}, h_{r,1}H_{r,1})=1$, and it follows that $f_{r,1}F_{r,1}=\cube$, $h_{r,1}H_{r,1}=\cube$. Let $(f_{r,1},F_{r,1})=M$, $f_{r,1}=f_{r,2}M$, $F_{r,1}=F_{r,2}M$. 
Then $M^2f_{r,2}F_{r,2}=\cube$. Write $M=CD^2\times \cube$ with $(C,D)=1$ and $C, D$ square-free. Then 
$C^2Df_{r,2}F_{r,2}=\cube$ and it follows that $f_{r,2}=C_1D_1^2\times \cube$, $F_{r,2}=C_2D_2^2 \times \cube$ where $C_1C_2=C$ and $D_1D_2=D$. Then we replace 
\[f_r\rightarrow X SCD^2C_1D_1^2f_r^3, \quad F_r\rightarrow YTCD^2C_2D_2^2F_r^3,\]
and similarly
\[h_r\rightarrow X TAB^2A_1B_1^2h_r^3, \quad H_r\rightarrow YSAB^2A_2B_2^2H_r^3,\]
with $A_1A_2=A$ and $B_1B_2=B$. 
We ignore the co-primality conditions when bounding the second term of \eqref{eq:errorterm}, and for the first term in \eqref{eq:errorterm} we keep the condition $(S,T)=1$, which we need to get the cancellation between the mollifier and the short Dirichlet polynomial of the $L$-function. 

\kommentar{\mcom{Old version: There are of course co-primality conditions between some of the variables, which we ignore when we upper bound the factors $E(r)$. We need to keep those conditions only in the case when $r \leq j$, to get the cancellation between the mollifier and the short Dirichlet polynomial of the $L$-function. This involves only the variables $T,S,X,Y$, where we have the condition $(S,T)=1$. }
\acom{I find this paragraph a bit confusing. We ignore the co-primality conditions when bounding the second term of \eqref{eq:errorterm}, and for the first term in \eqref{eq:errorterm} we keep the condition $(S,T)=1$.}}

Replacing in  \eqref{eq:errorterm2}, we get that the second term in \eqref{eq:errorterm} is bounded by 
\begin{eqnarray*}
&\leq& \frac{1}{2^{\ell_r}}\sum_{\substack{P|XYSTABCDf_r h_r F_r H_r\Rightarrow P \in I_r}} \frac{
(2\kk) ^{\Omega(X^2Y^2S^2T^2A^3B^6C^3D^6f_r^3h_r^3F_r^3H_r^3)}}{|XYSTB^3D^3||ACf_rh_rF_rH_r|^{3/2} } \nonumber\\
&=& \frac{1}{2^{\ell_r}} \prod_{P \in I_r} \Big(1-\frac{ (2\kk)^2}{|P|} \Big)^{-4} \Big(1- \frac{ (2\kk)^3}{|P|^{3/2}} \Big)^{-6} \Big(1- \frac{(2\kk)^6}{|P|^3} \Big)^{-2}\nonumber.
\end{eqnarray*}
Let $F(r)$ denote the expression above. Using the inequality form of the Prime Polynomial Theorem \eqref{ppt-bound}, note that for $r \neq 0$, we have that 
$$F(r) \leq_{\epsilon} \frac{1}{2^{\ell_r}}\exp 
 \Big(4(2\kk)^2+ 6(2\kk)^3 \sum_{(g+2)\theta_{r-1}<n\leq (g+2) \theta_r} \frac{1}{nq^{n/2}}+2(2\kk)^6 \sum_{(g+2)\theta_{r-1}<n\leq (g+2) \theta_r} \frac{1}{nq^{2n}} \Big),$$ and hence

\begin{equation}\label{F(r)crazy}
F(r)\leq_\varepsilon \frac{1}{2^{\ell_r}} \exp(16 \kk^2).
 \end{equation}
 For $r=0$, we have that
 $$F(r) \leq_{\epsilon} \frac{1}{2^{\ell_0}} ((g+2) \theta_0)^{O(1)},$$ and we remark that

\[\lim_{g \rightarrow \infty} F(0)= 0.\]

For the first term in  \eqref{eq:errorterm}, using the change of variable from before, we get 
\begin{eqnarray*}
&&\sum_{\substack{P|ABCDf_r h_r F_r H_r\Rightarrow P \in I_r \\ C_1C_2=C, D_1D_2=D\\A_1A_2=A, B_1B_2=B}}\frac{ (\kk/2)^{\Omega(CD^2C_2D_2^2AB^2A_2B_2^2F_r^3H_r^3)}
 \lambda(CC_1AA_1f_rh_r)}{\kappa^{\Omega(CD^2C_1D_1^2AB^2A_1B_1^2f_r^3 h_r^3)}  \sqrt{  |C^3D^6 A^3B^6f_r^3 h_r^3F_r^3 H_r^3|}}\\
&& \times a(f_r^3h_r^3;\JJ )a(AB^2A_1 B_1^2CD^2C_1D_1^2;\JJ ) a(F_r^3H_r^3;j) a(AB^2A_2B_2^2CD^2C_2D_2^2;j) \\
&& \times \sum_{\substack{P|STXY\Rightarrow P \in I_r\\ (S, T)=1}} \frac{ (\kk/2)^{\Omega(Y^2ST)}
a(S;\JJ ) a(S; j) a(T;\JJ ) a(T;j) a(X; \JJ )^2 a(Y; j)^2 \lambda(ST) }{\kappa^{\Omega(X^2 ST)}  |YX ST| } \\
&&\;\;\;\;\;\;\;\;\;\; \;\;\;\;\;\;\;\;\;\;\times   \nu_2( YTCD^2C_2D_2^2F_r^3) \nu_2(YSAB^2A_2B_2^2H_r^3) \\
&&\;\;\;\;\;\;\;\;\;\; \;\;\;\;\;\;\;\;\;\;\times  \nu_{\kk \kappa} (X SCD^2C_1D_1^2f_r^3) \nu_{\kk \kappa}(X TAB^2A_1B_1^2h_r^3) \\
&& \;\;\;\;\;\;\;\;\;\; \;\;\;\;\;\;\;\;\;\; \times \frac{\phi_{q^2}(X^3 S^3T^3Y^3C^3D^6 A^6B^{12}f_r^3h_r^6F_r^3H_r^6)}{ |X^3 S^3T^3Y^3C^3D^6 A^6B^{12}f_r^3h_r^6F_r^3H_r^6|_{q^2}}.
\end{eqnarray*}

For every fixed value of $A, B, C, D, f_r, h_r, F_r, H_r, A_1, A_2, B_1, B_2, C_1, C_2, D_1, D_2$, let
\begin{eqnarray*} &&\mathcal{F} (A, B, C, D, f_r, h_r, F_r, H_r,  A_1, A_2, B_1, B_2, C_1, C_2, D_1, D_2)
\\ &:=& \prod_{\substack{P \in I_r \\P^a \parallel A, P^b \parallel B, \dots, P^{d_1} \parallel D_1, P^{d_2} \parallel D_2}} \frac{\sigma(P; a,b,c,d,f,h,F,H, a_1, a_2, b_1, b_2, c_1, c_2, d_1, d_2)}{\sigma(P)},\end{eqnarray*}
where 
\begin{eqnarray*}
&& \sigma(P; a,b,c,d,f,h,F,H, a_1, a_2, b_1, b_2, c_1, c_2, d_1, d_2)\\
:= &&\sum_{s, t, x, y \geq 0, \; st =0} \left(  (\kk/2)^{s+t+2y} (1/\kappa)^{s+t+2x} a(P; j)^{s+t+2y} a(P; \JJ )^{s+t+2x}  (-1)^{s + t} \right.\\
&& \times \nu_2(P^{t + y+c+2d+c_2+2 d_2 +3F}) \nu_2(P^{s + y+a+2b+a_2+2b_2+3H})  \\
&& \times \nu_{\kk \kappa}(P^{s + x+c+2d+c_1+2d_1+3f}) \nu_{\kk \kappa}(P^{t+ x+a+2b+a_1+2b_1+3h})\\
&& \;\;\;\;\;\;\;\;\;\;\;\;\;\; \left. \times \frac{\phi_{q^2}(P^{3s + 3t + 3x + 3y+3c+6d+6a+12b+3f+6h+3F+6H})}{|P^{3s + 3t + 3x + 3y+3c+6d+6a+12b+3f+6h+3F+6H}|_{q^2}} \frac{1}{|P|^{s+t+x+y}} \right),
\end{eqnarray*}
and  $\sigma(P):=\sigma(P;0, \dots, 0)$.

We can rewrite the first term in  \eqref{eq:errorterm} as
\begin{eqnarray*}
&& \prod_{P \in I_r} \sigma(P)  \times \sum_{\substack{P|ABCDf_r h_r F_r H_r\Rightarrow P \in I_r \\ C_1C_2=C, D_1D_2=D\\A_1A_2=A, B_1B_2=B}}\frac{ (\kk/2)^{\Omega(CD^2C_2D_2^2AB^2A_2B_2^2F_r^3H_r^3)}
 \lambda(CC_1AA_1f_rh_r)}{\kappa^{\Omega(CD^2C_1D_1^2AB^2A_1B_1^2f_r^3 h_r^3)}  \sqrt{  |C^3D^6 A^3B^6f_r^3 h_r^3F_r^3 H_r^3|}}\\
 && \times a(f_r^3h_r^3;\JJ )a(AB^2CD^2C_1D_1^2A_1 B_1^2;\JJ ) a(F_r^3H_r^3;j) a(AB^2CD^2A_2B_2^2C_2D_2^2;j) \\
&& \;\;\;\;\;\;\;\;\;\;\;\;\;\;\;\;\;\; \times \;\mathcal{F} (A, B, C, D, f_r, h_r, F_r, H_r,  A_1, A_2, B_1, B_2, C_1, C_2, D_1, D_2).
\end{eqnarray*}

It is clear that the sum is absolutely bounded; since we need an explicit constant, we will prove that 
that the sum is $\leq 1$. We can write it as an Euler product, and we look at the coefficient of $1/|P|^{3/2}$. We need to consider $f_r=P$, $h_r=P, F_r=P, H_r=P$ and $A=P, A_1=1$ and $A=P, A_1=P,$ $C=P, C_1=1$ and $C=P, C_1=P$. When $f_r=P$ (and everything else is $1$) we get a factor of
$$ -\frac{1}{\kappa^3}a(P;\JJ )^3 \nu_{\kk \kappa} (P^3) \frac{\phi_{q^2}(P^3)}{|P|^{3/2}|P|_{q^2}^3}= -\frac{\kk^3 a(P;\JJ )^3 \phi_{q^2}(P^3)}{6|P|^{3/2} |P^3|_{q^2}} .$$ We get the same term when $h_r=P$. If $F_r=P$ we get 
$$ \frac{\kk^3 a(P;j)^3\phi_{q^2}(P^3)}{6|P|^{3/2}|P^3|_{q^2}},$$
and when $H_r=P$ we get the same factor. If $A=P, A_1=1$, we get the term 
$$ - \frac{\kk^2a(P;\JJ ) a(P;j)^2}{\kappa} \frac{\kk \kappa \phi_{q^2}(P^6)}{2|P|^{3/2}|P^6|_{q^2}}= -\frac{\kk^3 a(P;\JJ ) a(P;j)^2 \phi_{q^2}(P^6)}{2 |P|^{3/2} |P^6|_{q^2}}.$$
Similarly when $A=P, A_1=P$ we get
$$ \frac{\kk^3 a(P;\JJ )^2 a(P;j) \phi_{q^2}(P^6)}{2 |P|^{3/2} |P^6|_{q^2}}.$$
We get the same contribution for $C=P, C_1=1$ and $C=P, C_1=P$ as for the $A$'s. Overall for the sum over $A,B,C,D, A_1, C_1, B_1, D_1,F_r,h_r,F_r,H_r$ we get 
$$\prod_{P \in I_r} \Big(1+ \frac{\kk^3(a(P;j)-a(P;\JJ ))^3 \phi_{q^2} (P^3)}{3|P|^{3/2} |P^3|_{q^2}} + O \Big( \frac{1}{|P|^{5/2}} \Big) \Big) \leq 1,$$ since $a(P;j) < a(P;\JJ )$.

The Euler product is equal to
\begin{eqnarray*} \nonumber
\prod_{P \in I_r} \sigma(P) &=&  \prod_{P \in I_r} \sum_{\substack{s, t, x, y \geq 0 \\ st =0}} \left(  (\kk/2)^{s+t+2y} (1/\kappa)^{s+t+2x} a(P; j)^{s+t+2y} a(P; \JJ )^{s+t+2x}   (-1)^{s + t}  
 \right. \\ \nonumber
&& \;\;\;\;\;\;\;\;\;\;\;\;\;\; \left.  \nu_2(P^{s + y}) 
\nu_2(P^{t + y})  \nu_{\kk \kappa }(P^{s + x}) \nu_{\kk \kappa }(P^{t + x}) 
\frac{\phi_{q^2}(P^{3s + 3t + 3x + 3y})}{|P^{3s + 3t + 3x+3y}|_{q^2}} \frac{1}{|P|^{s+t+x+y}} \right) \\ \label{cancel}
&=&  \prod_{P \in I_r} \left( 1 + \frac{\alpha_j(P)}{|P|} + O \left( \frac{1}{|P|^2} \right) \right),
\end{eqnarray*}
where we define $\alpha_j(P):=\kk^2 (a(P;j)-a(P;\JJ ))^2$ for any $0\leq r \leq j$. 

We have
\begin{align*}
 a(P;\JJ )-a(P;j) &= \frac{1}{|P|^{\frac{1}{(g+2) \theta_\JJ  \log q}}} - \frac{1}{|P|^{\frac{1}{(g+2) \theta_j \log q}}}+ \frac{\deg(P)}{(g+2) \theta_j |P|^{\frac{1}{(g+2) \theta_j \log q}}} -\frac{\deg(P)}{(g+2) \theta_\JJ  |P|^{\frac{1}{(g+2) \theta_\JJ  \log q}}} \\
 & \leq 1- \Big(1- \frac{\deg(P) }{(g+2) \theta_j}\Big)+ \frac{\deg(P)}{(g+2) \theta_j} \leq \frac{2 \deg(P) }{(g+2) \theta_j},
 \end{align*} 
which gives that
\begin{align}
\prod_{\deg{P} \leq (g+2) \theta_j} 
\left( 1 +  \frac{\alpha_j(P)}{|P|} + O \left( \frac{1}{|P|^2} \right) \right)
 \leq& \exp \left(\sum_{\deg{P} \leq (g+2) \theta_j} \frac{\alpha_j(P)}{|P|}  \right) \nonumber \\
 \leq& \exp  \left(\frac{4\kk^2  }{(g+2)^2\theta_j^2}\sum_{\deg{P} \leq (g+2) \theta_j} 
 \frac{\deg(P)^2}{|P|}  \right) \nonumber  \\
 \leq& \exp \Big(2\kk^2  \Big(1+\frac{1}{(g+2) \theta_j}\Big) \Big) \label{bound}.
 \end{align}
Combining with \eqref{F(r)crazy} and incorporating everything in \eqref{eq:errorterm}, we get that the contributions from the intervals $I_0, \dots, I_j$ are bounded by
\begin{align}
&\prod_{r=0}^j (1 + e^{-\ell_r / 2})^2 E(r)  \nonumber \\
&\leq   \prod_{r=0}^j  \Big( \prod_{P \in I_r} \Big(1+\frac{\alpha_j(P)}{|P|} + O \Big(\frac{1}{|P|^2} \Big) \Big)+F(r)  \Big) (1+e^{-\ell_r/2})^2 \nonumber  \\
& \leq_\varepsilon \mathcal{D}_k\prod_{\deg{P} \leq (g+2) \theta_j} \left(
1 +  \frac{\alpha_j(P)}{|P|} + O \left( \frac{1}{|P|^2} \right) \right) \nonumber \\
& \leq_\varepsilon  \mathcal{D}_k \exp (2\kk^2 ), \label{nr1}
\end{align}
where 
\begin{equation} 
\mathcal{D}_k=(1+e^{-\ell_0/2})^2 \prod_{r=1}^\JJ  (1+e^{-\ell_r/2})^2\Big(1+\frac{e^{16\kk^2}}{2^{\ell_r}} \Big). \label{c01} 
\end{equation}


We now look at the term  $r= j+1$ from \eqref{after-lemma}, which involves the mollifier and  $(\Re P_{I_{r}}(\chi))^{2s_r}$. We first write
\begin{eqnarray} \label{r=j+1}
{E(r)} \leq \frac{(2s_r)!}{4^{s_r}} 
\sum_{\substack{P|f_r h_r F_r H_r\Rightarrow P \in I_r \\ \Omega(F_r H_r) = 2s_r \\ \Omega(f_r) \leq (\kk \cdot \kappa)\ell_r \\ \Omega(h_r) \leq ( \kk  \cdot \kappa) \ell_r \\ f_r h_r^2 F_r H_r^2=\tinycube}}
\frac{\nu(F_r) \nu(H_r) \nu_{\kk \kappa} (f_r) \nu_{\kk \kappa}(h_r) }{ \kappa^{\Omega(f_r h_r)} \sqrt{  |f_r F_r h_r H_r|}},
\end{eqnarray}
where we have bounded $\lambda(f_r h_r), a(f_r h_r; \JJ ) a(F_r H_r; u), \phi_{q^2} (f_r h_r^2 F_r H_r^2)/  | (f_r h_r^2 F_r H_r^2)|_{q^2} \leq 1$, $\nu_{\kk \kappa}(f_r; \ell_r) \leq \nu_{\kk \kappa}(f_r)$.

Using the change of variable as before, we can rewrite the sum of \eqref{r=j+1} as
\begin{eqnarray} \label{2-sums}
&&\sum_{\substack{X, S, T, C, D, A, B, f_r, h_r \\ C_1 C_2 = C, \; D_1 D_2 = D \\ A_1 A_2 = A, \; B_1 B_2 = B \\
P|XST f_r h_r ABCD \Rightarrow P \in I_r \\ 
\Omega(XSCD^2C_1D_1^2f_r^3) \leq \kk \kappa  \ell_r \\ \Omega(XTAB^2A_1B_1^2h_r^3) \leq \kk \kappa \ell_r}}
\frac{ \nu_{\kk \kappa} (XSCD^2C_1D_1^2f_r^3) \nu_{\kk \kappa}(XTAB^2 A_1 B_1^2 h_r^3) }{ \kappa^{\Omega(X^2SCD^2C_1D_1^2f_r^3 TAB^2 A_1 B_1^2 h_r^3)} \sqrt{  |(X^2  S^2 T^2 C^3 A^3 B^6 D^6 f_r^3  h_r^3|}},
\\ \nonumber
&& \times \sum_{\substack{Y, F_r, H_r \\ P|Y F_r H_r \Rightarrow P \in I_r\\\Omega(Y^2 F_r^3 H_r^3) = 2s_r - \Omega(TCD^2 C_2 D_2^2 S A B^2 A_2 B_2^2
)}} \frac{ \nu(Y)^2 \nu(F_r) \nu(H_r)}{|Y| |F_r H_r|^{3/2}
 3^{\Omega(F_r)} 3^{\Omega(H_r)}},
\end{eqnarray}
 where we have used the fact that $\nu(Z^3) \leq \nu(Z)/3^{\Omega(Z)}$ and $\nu(\cdot) \leq 1$.

Now note that $ \Omega(TCD^2 C_2 D_2^2 S A B^2 A_2 B_2^2) \leq \Omega((ST) (CD^2)^2(AB^2)^2)) \leq 4 \kk \kappa \ell_r$  and by hypothesis
$ 4 \kk \kappa \ell_r \leq \frac{4}{a} s_r$, so $\Omega(Y^2 F_r^3 H_r^3) \geq (2-\frac{4}{a})s_r:=c s_r$.

 Let $\alpha=2s_r -   \Omega(TCD^2 C_2 D_2^2 S A B^2 A_2 B_2^2) $.
 Using the fact that $\nu(Y)^2 \leq \nu(Y)$ (since $\nu(Y) \leq 1$), the sum over $Y, F_r, H_r$  is   bounded by 
 \begin{align} 
 \sum_{2i+3j+3k=\alpha} \sum_{\substack{P| Y \Rightarrow P \in I_r \\ \Omega(Y) = i}} \frac{\nu(Y)}{|Y|} \sum_{\substack{P|F_r \Rightarrow P \in I_r \\ \Omega(F_r)=j}} \frac{\nu(F_r)}{3^{\Omega(F_r)} |F_r|^{3/2}} \sum_{\substack{P|H_r \Rightarrow P \in I_r \\ \Omega(H_r)=k}} \frac{\nu(H_r)}{3^{\Omega(H_r)} |H_r|^{3/2}} . \label{to_bound}
 \end{align}
 Now $$ \sum_{\substack{P|F_r \Rightarrow P \in I_r \\ \Omega(F_r)=j}} \frac{\nu(F_r)}{3^{\Omega(F_r)} |F_r|^{3/2}} = \frac{1}{j!} \Big( \sum_{P \in I_r} \frac{1}{3|P|^{3/2}} \Big)^j,$$ a similar expression holds for the sum over $H_r$,
 and
 $$ \sum_{\substack{P|Y \Rightarrow P \in I_r \\ \Omega(Y) = i}} \frac{\nu(Y)}{|Y|} = \frac{1}{i!} \Big( \sum_{P \in I_r} \frac{1}{|P|} \Big)^i.$$
 Using the inequalities above, it follows that
 \begin{align}
 \eqref{to_bound}  & \leq \Big( \sum_{P \in I_r} \frac{1}{|P|} \Big)^{\alpha/2} \sum_{2i+3j+3k=\alpha} \frac{1}{i!j!k!3^{j+k}} = \Big( \sum_{P \in I_r} \frac{1}{|P|} \Big)^{\alpha/2} \sum_{\substack{i \leq \alpha/2\\3\mid (\alpha-2i)}} \Big( \frac{2}{3} \Big)^{\frac{\alpha-2i}{3}} \frac{1}{ i!  ( \frac{\alpha-2i}{3}) !} \nonumber \\
 & \leq \Big( \sum_{P \in I_r} \frac{1}{|P|} \Big)^{\alpha/2}\Big( \frac{2}{3} \Big)^{\alpha/3} \Big( \sum_{i \leq \alpha/3} \frac{\Big( \frac{2}{3} \Big)^{\frac{-2i}{3}}}{i! \lfloor \frac{\alpha}{3}-i \rfloor!}+ \sum_{\substack{\alpha/3<i\leq \alpha/2\\3 \mid (\alpha-2i)}} \frac{\Big( \frac{2}{3} \Big)^{\frac{-2i}{3}}}{\lfloor \frac{2i}{3}\rfloor! ( \frac{\alpha-2i}{3})!} \Big)  \nonumber\\ \label{internal-sum}
 &\leq 2\Big( \sum_{P \in I_r} \frac{1}{|P|} \Big)^{\alpha/2}\Big( \frac{2}{3} \Big)^{\alpha/3} \frac{\left(\frac{5}{2}\right)^{\alpha/3}}{\lfloor \alpha/3\rfloor!}\leq 2\Big( \sum_{P \in I_r} \frac{1}{|P|} \Big)^{s_r}\frac{\left(\frac{5}{3}\right)^{c s_r/3}}{\lfloor cs_r/3\rfloor !},
 \end{align}

We now consider the exterior sum in \eqref{2-sums}. 
For the sum over $A_1$ we have (recall that $A$ is square-free)
$$ \sum_{A_1|A} \frac{\nu_{\kk \kappa}(A_1)}{\kappa^{\Omega(A_1)}} = \prod_{P|A}(1+\kk).$$ Then overall we get that
$$\sum_{P|A \Rightarrow P \in I_r} \frac{\nu_{\kk \kappa}(A)}{\kappa^{\Omega(A)}} (\kk+1)^{\omega(A)} \frac{1}{|A|^{3/2}} = \prod_{P \in I_r} \Big( 1+ \frac{\kk(\kk+1)}{|P|^{3/2}} \Big).$$
Similar expressions hold for the sums over $C, B, D$ and overall for the sum over $X,S,T,A,B,C,D, f_r, h_r$ we will get that it is 
\begin{align*}
\leq \prod_{P \in I_r} \Big( 1+ \frac{\kk(\kk+1)}{|P|^{3/2}} \Big)^2 \Big(1+ \frac{\kk^2(\kk^2/2+1)}{2|P|^3} \Big)^2 \Big(1+ \frac{\kk^3}{6|P|^{3/2}} \Big)^2 \Big( 1+\frac{\kk^2}{|P|}\Big)\Big( 1+ \frac{\kk}{|P|}\Big)^2 := H(r).
\end{align*}
Using the Prime Polynomial Theorem \eqref{ppt-bound}, we get  that 
\begin{equation*} H(r) \leq_\varepsilon
  \exp \Big( \kk^2+2\kk+ \frac{2\kk(\kk+1)}{q^{(g+2) \theta_{r-1}/2}} + \frac{ \kk^3}{3q^{(g+2) \theta_{r-1}/2}}  + \frac{\kk^2 (\frac{\kk^2}{2}+1)}{q^{2(g+2) \theta_{r-1}}}\Big)
 \label{hr}
\end{equation*}
for $r \neq 0$,
and then 
 \begin{equation} \label{external-sum} H(r) \leq _{\varepsilon} \exp(\kk^2+2\kk),\end{equation}
which is what we need to prove (iii).
For $r=0$, we have that 
\begin{equation}\label{h0}
H(0) \ll (g \theta_0)^{O(1)}.
\end{equation}

\kommentar{\ccom{ $H(0)$  belongs to the term which is $o(q^g)$ in the proof of the upper bound theorem, see page 17.  So it is enough to use $H(0) = g^{O(1)},$ which is clear and we use, and we refer to (49). But we could write it here more explicitely as $H(0) \ll g^{O(1)}.$ in the equation below.}
\acom{Same comment about the $q-k^2$.} \mcom{Added it.}}

 In (i), the bound will depend on $H(0)$.
Replacing  \eqref{internal-sum} and \eqref{external-sum} in  \eqref{2-sums} and finally in \eqref{r=j+1}, it follows that 
\begin{equation}\label{nr2}E(j+1) \leq_{\varepsilon} 2 \exp(\kk^2+2\kk) \Big( \sum_{P \in I_{j+1}} \frac{1}{|P|} \Big)^{s_{j+1}} \; \frac{\left(\frac{5}{3}\right)^{cs_{j+1}/3}(2s_{j+1})!}{ 4^{s_{j+1}} \lfloor\frac{cs_{j+1}}{3}\rfloor!}.
\end{equation}

Finally, we consider the case where $r \geq j+2$ and in this case, only the mollifier contributes primes in this interval in the factors 
of \eqref{after-lemma}. It is easy to see that
$$
\prod_{r=j+2}^{\JJ } E(r) \leq
\prod_{r=j+2}^\JJ  \sum_{\substack{P \mid f_r  h_r \Rightarrow P \in I_r\\ f_r h_r^2 = \tinycube}} 
\frac{\kk^{\Omega(f_r)} \kk^{\Omega(h_r)}}{ \sqrt{|f_r h_r|}},
$$
where we used the same bound as above on the functions appearing in the mollifier, and 
we also used the fact that $\nu_{\kk \kappa}(g_r) \leq (\kk \kappa)^{\Omega(g_r)}$.
Note that $f_r h_r^2 = \cube$ is equivalent to $f_r = g_r S_r^3$ and $h_r = g_r T_r^3$ for $(S_r, T_r)=1$. Then, the term corresponding to a fixed $r$ in the product above is bounded by

\begin{align*}
 \leq \sum_{\substack{P \mid g_r \Rightarrow P \in I_r}} \frac{\kk^{2 \Omega(g_r)}}{|g_r|} 
 \sum_{\substack{P \mid S_r \Rightarrow P \in I_r}} \frac{\kk^{\Omega(S_r^3)}}{|S_r|^{3/2}} 
 \sum_{\substack{P \mid T_r \Rightarrow P \in I_r}} \frac{\kk^{\Omega(T_r^3)}}{|T_r|^{3/2} } 
=\prod_{P \in I_r} \Big( 1-\frac{\kk^2}{|P|}\Big)^{-1} \Big(1- \frac{\kk^3}{|P|^{3/2}} \Big)^{-2}.
 \end{align*}

Using the fact that $-\log(1-x) < \frac{x}{1-x}$ we get that
\begin{align*}
\prod_{P \in I_r}  \Big(1- \frac{\kk^2}{|P|} \Big)^{-1} & \leq \exp \Big( \sum_{P \in I_r} \frac{\kk^2}{|P|-\kk^2} \Big) \leq \exp \Big(\kk^2 \Big( \sum_{n=(g+2) \theta_{r-1}}^{(g+2) \theta_r} \frac{1}{n} + \sum_{n=(g+2) \theta_{r-1}}^{(g+2) \theta_r} \frac{\kk^2}{n(q^n-\kk^2)} \Big) \Big) \\
& \leq_{\epsilon} \exp \Big( \kk^2+\frac{\kk^4}{q^{(g+2) \theta_{r-1}}-\kk^2} \Big).
\end{align*}
Similarly 
\begin{align*}
\prod_{P \in I_r}  \Big( 1-\frac{\kk^3}{|P|^{3/2}} \Big)^{-2}  & \leq \exp \Big(2 \sum_{P \in I_r} \frac{\kk^3}{|P|^{3/2}-\kk^3} \Big) \leq \exp \Big(2\kk^3 \sum_{n=(g+2) \theta_{r-1}}^{(g+2) \theta_r} \Big( \frac{1}{nq^{n/2}}+O \Big(\frac{1}{nq^{2n}} \Big) \Big) \Big) \\
& \leq_{\epsilon} \exp \Big( \frac{2\kk^3}{q^{(g+2) \theta_{r-1}/2}}+O \Big(\frac{1}{q^{2(g+2) \theta_{r-1}}} \Big) \Big).
\end{align*}

Then the contribution from $r \geq j+2$ will be bounded by
\begin{align} &\leq_{\epsilon} e^{\kk^2(\JJ -j-1)} \prod_{r=j+2}^\JJ  \exp \Big( \frac{\kk^4}{q^{(g+2) \theta_{r-1}}-\kk^2}+ \frac{2\kk^3}{q^{(g+2) \theta_{r-1}/2}}+ O \Big( \frac{1}{q^{2(g+2) \theta_{r-1}}} \Big) \Big) \nonumber \\
& \leq_\varepsilon e^{\kk^2(\JJ -j-1)}.
\label{nr3}
\end{align}

Combining the contribution of the intervals $I_r$ with $r \leq j$ from \eqref{nr1},  the contribution of the interval $I_{ j+1}$ from  \eqref{nr2} and  the contribution of the intervals $I_r$ with $j+2\leq r \leq \JJ $ from \eqref{nr3}, we get the bound of the last inequality.

We prove the first inequality corresponding to ``$j=-1$" in the same way, except that the bound for $H(r)$ in \eqref{external-sum}  is not valid for $r=0$, so we just keep $H(0)$ on the right-hand side. The second inequality corresponds to $j=J$.

\kommentar{\acom{I find the paragraph above a bit confusing. The first inequality corresponds to $j=-1$ and the only difference is in the bound for $H(r)$ that we get in \eqref{hr}. The second inequality corresponds to $j=\JJ $.}
\ccom{Is that better?}}
 
\end{proof}

\section{Squares of the primes}
\label{square-primes}

 In this section we prove an upper bound for the average over the square of the primes appearing in the $k^{\text{th}}$ moment. Our proof is similar to \cite{Harper}, but it is simpler because we separate the primes and the square of the primes from the start by using Cauchy--Schwatz in order to deal with the mollifier.

 We recall that
$$ S_{j,\kk}(\chi)=\exp  \left(\kk
\Re\left( \sum_{\deg(P) \leq (g+2) \theta_j/2} \frac{   \chi(P)b(P;j)}{|P|}\right)\right),$$
where the positive weights $b(P;j)$ are defined by \eqref{def-bp}. Then, $b(P;j) \leq \frac12$, which is the only property that we use in this section.

    \begin{lem} \label{lemma-square-primes} Let  $S_{j,\kk}$ be the sum defined by \eqref{S_{j,\kk}} and let $\beta >1$. For $j=0,\ldots,\JJ $ we have
$$\sum_{\chi \in \mathcal{C}(g)}S_{j,\kk}(\chi)^2 \leq q^{g+2} \Bigg( \exp\Big(\kk+\frac{2\kk}{\beta-1} \Big)+ \frac{3 e^{\kk (\gamma+1)} }{4}  \sum_{m=1}^{\infty} \exp \Big(\kk \log m+ \frac{2\kk}{\beta^{m}(\beta-1)} \Big) \frac{ \beta^{4m}}{ q^{2m}} \Bigg) .$$ 
In particular, choosing $\beta =2$ and using that $q\geq 5$, we have 
\begin{equation}
 \sum_{\chi \in \mathcal{C}(g)}S_{j,k}(\chi)^2 \leq q^{g+2}  \mathcal{S}_k, 
 \label{sk}
 \end{equation}
where
\[\mathcal{S}_k:=e^{3\kk}+ \frac{4 k!e^{\kk (\gamma+2)} }{3} \left(\frac{25}{9}\right)^{\kk},\]
and in particular 
\[\mathcal{S}_2\approx 3967.15\dots.\]
\kommentar{\acom{I agree with the bound above, but then for $\mathcal{S}_k$, I get
$$ \frac{3}{4} e^{k(\gamma+1)} \Big(\frac{25}{9} \Big)^{k} \frac{16k!}{9}.$$ There should be some $k!$ in the bound for $\mathcal{S}_k$, right?}\mcom{You're right, so sorry!!!!!}

\mcom{Old thing: For the second moment ($k=2$) we have
$$ \sum_{\chi \in \mathcal{C}(g)}S_{j,2}(\chi)^2 \leq q^{g+2} \Big( e^6+ \frac{12 e^{2 (\gamma+1)} q^2(q^2+16) }{(q^2-16)^3} \Big) \leq S q^{g+2},$$
where we choose $S = e^6+ \displaystyle \frac{4100 e^{2(\gamma+1)}}{243}$ independently of $q$. }\acom{See the constant. We're using the fact that $q \geq 5$.} 
}
\label{estimates-squares}
\end{lem}
 \begin{proof}
 Let 
 $$F_m(\chi;j) = \sum_{P \in \mathcal{P}_m} \frac{\chi(P) b(P;j)}{|P|},$$
 where the sum is over the monic irreducible polynomials of degree $m$.
 For ease of notation, we will simply denote the sum above by $F_m(\chi)$. Let
 $$\mathcal{F}(m) = \left\{ \chi \in \mathcal{C}(g) \, : |  \, \Re F_m(\chi) | > \frac{1}{\beta^m}, \;\mbox{but}\; | \Re F_n(\chi) | \leq \frac{1}{\beta^n}, \; \forall m+1 \leq n \leq (g+2) \theta_j/2 \right\}.$$
 Note that $F_0(\chi)=\frac{1}{2}$, so the set $\mathcal{F}(0)$ is empty. Since the sets $\mathcal{F}(m)$ are disjoint, note that we have
 \begin{equation}
  \sum_{\chi \in \mathcal{C}(g)}S_{j,\kk}(\chi)^2  \leq \sum_{m=1}^{(g+2) \theta_j/2} \sum_{\chi \in \mathcal{F}(m)} S_{j,\kk}(\chi)^2 + \sum_{\chi \notin \mathcal{F}(m),  \forall m}  S_{j,\kk}(\chi)^2.
  \label{sum_char}
  \end{equation}
  If $\chi$ does not belong to any of the sets $\mathcal{F}(m)$, then 
  $$ | \Re F_n(\chi)| \leq \frac{1}{\beta^n},$$ for all $1 \leq n \leq (g+2) \theta_j/2$, so in this case we have
  $$S_{j,\kk}(\chi)^2 \leq \exp\Big(\kk+\frac{2\kk}{\beta-1} \Big).$$
  Now assume that $\chi \in \mathcal{F}(m)$ for some $1 \leq m \leq (g+2) \theta_j/2$. Then we have
  \begin{align*}
  \sum_{l=1}^{(g+2) \theta_j/2} \Re F_l(\chi) \leq \sum_{l=1}^m \frac{1}{2l} + \sum_{l=m+1}^{(g+2) \theta_j/2} \frac{1}{\beta^l} \leq \frac{1}{2} \Big( \log m+\gamma +1 \Big)+\frac{1}{\beta^{m+1}(1-\frac{1}{\beta})},
  \end{align*}
  where $\gamma$ is the Euler--Mascheroni constant. 
  Therefore in this case we have
  $$S_{j,\kk}(\chi)^2 \leq \exp \Big(\kk (\log m+\gamma+1)+ \frac{2\kk}{\beta^{m+1}(1-\frac{1}{\beta})} \Big).$$
  If $\chi \in \mathcal{F}(m)$, also note that $(\beta^m \Re F_m(\chi))^4>1$, so combining this with the inequality above, we get
 \begin{align}
 \sum_{\chi \in \mathcal{F}(m)} S_{j,\kk}(\chi)^2 \leq  \exp \Big(\kk (\log m+\gamma+1)+ \frac{2\kk}{\beta^{m+1}(1-\frac{1}{\beta})} \Big)  \sum_{\chi \in \mathcal{C}(g)} (\beta^m \Re F_m(\chi))^4. \label{interm1}
 \end{align} 
Note that by Lemma \ref{combinatorics}, 
$$  \sum_{\chi \in \mathcal{C}(g)}  (\beta^m \Re F_m(\chi))^4= \frac{4! \beta^{4m}}{2^4}  \sum_{\chi \in \mathcal{C}(g)}\sum_{\substack{ P |fh \Rightarrow P \in \mathcal{P}_m \\ \Omega(fh)=4}} \frac{b(f;j)b(f;\JJ ) \chi(f) \overline{\chi}(h) \nu(f) \nu(h)}{|fh|}.$$
Using Lemma \ref{mt} (note that $8m \leq (g+2)/2$ since $\theta_\JJ $ is small enough), we get that
\begin{align}
 \sum_{\chi \in \mathcal{C}(g)}  (\beta^{m} \Re F_m(\chi))^4 \leq q^{g+2} \frac{4!\beta^{4m}}{2^4} \sum_{\substack{ P |fh \Rightarrow P \in \mathcal{P}_m \\ \Omega(fh)=4 \\ fh^2=\tinycube}} \frac{\nu(f) \nu(h) }{|fh|}. \label{interm2}
\end{align}
When $fh^2=\cube$ we can write $f=b f_1^3, h=b h_1^3$ with $(f_1,h_1)=1$.
Since $\Omega(b^2f_1^3h_1^3)=4$ it follows that $f_1=h_1=1$ and $\Omega(b)=2$. Then using the fact that $\nu(b)^2 \leq \nu(b)$ we get
\begin{align*}
 \sum_{\substack{ P |fh \Rightarrow P \in \mathcal{P}_m \\ \Omega(fh)=4 \\ fh^2=\tinycube}} \frac{\nu(f) \nu(h) }{|fh|} \leq \sum_{\substack{P|g \Rightarrow P \in \mathcal{P}_m \\ \Omega(b)=2}} \frac{ \nu(b)}{|b|^2}= \frac{1}{2} \Big( \sum_{P \in \mathcal{P}_m} \frac{1}{|P|^2} \Big)^2 \leq \frac{1}{2q^{2m}},
\end{align*}
where for the last inequality we used the Prime Polynomial Theorem \eqref{ppt-bound}. Combining the above and equations \eqref{interm1} and \eqref{interm2}, we get that 
\begin{align}
\label{interm3}
\sum_{m=1}^{(g+2) \theta_j/2} \sum_{\chi \in \mathcal{F}(m)} S_{j,\kk}(\chi)^2 \leq q^{g+2} \sum_{m=1}^{(g+2)\theta_j/2} \exp \Big(\kk (\log m+\gamma+1)+ \frac{2\kk}{\beta^{m+1}(1-\frac{1}{\beta})} \Big) \frac{4! \beta^{4m}}{2^5 q^{2m}}.
\end{align}
Note that for any $1<\beta<\sqrt{q}$, the expression above will be $\ll q^{g+2}$. 

Now we take $\beta=2$. We use the fact that $\exp( 2k/(\beta^m(\beta-1))) \leq  e^k$, and the fact that $$\sum_{m=1}^{\infty} m^k x^m \leq \frac{xk!}{(1-x)^{k+1}},$$ for $x<1$. Since $q \geq 5$, using the inequality above and \eqref{interm3}, inequality \eqref{sk} follows.

\kommentar{\acom{We should find $b$ which minimizes the bound above. Also, here we take $a=4$ using the notation we have in the previous version; it seems a bit suspicious that the choice of $a$ is so easy, and it also seems we don't need to split into ranges for $m$. Maybe I'm missing something. Let me know what you think.}}
 \end{proof}

 
 \section{Upper bounds for moments of $L$--functions}
 \label{ub_sound}
 Here, we will prove the following upper bound. 
 \begin{prop} 
 \label{lemma_sound}
 For any positive real number $k$ and any $\varepsilon>0$, we have
 $$\sum_{\chi \in \mathcal{C}(g)} |{\textstyle L(\frac{1}{2},\chi)}|^{2k}  \ll q^{g+2} g^{k^2+\varepsilon}. $$
 \end{prop}
 We first prove the following result.

\begin{lem}
 \label{dir_pol_2} \label{dir_pol}
 Let $\el$ and $y$ be integers such that $3 \el y \leq g/2+1$. For any complex numbers $a(P)$ with $|a(P)| \ll 1$ we have
 \begin{equation}
  \sum_{\chi \in \mathcal{C}(g)} \Big|   \sum_{\deg(P) \leq y} \frac{ \chi(P) a(P)}{\sqrt{|P|}} \Big|^{2\el} \ll q^g  \frac{ (\el!)^2 5^{2\el/3} }{\lfloor 2\el/3 \rfloor ! 9^{\el/3}} \Big(  \sum_{\deg(P) \leq y} \frac{|a(P)|^2}{|P|} \Big)^\el.
 \label{first}
 \end{equation}
 If we also assume that $\el \leq \displaystyle \Big(\sum_{\deg(P) \leq y} \frac{|a(P)|^2}{|P|}\Big)^{3-\varepsilon}$, then we have
 \begin{equation}
   \sum_{\chi \in \mathcal{C}(g)} \Big|   \sum_{\deg(P) \leq y} \frac{ \chi(P) a(P)}{\sqrt{|P|}} \Big|^{2\el} \ll q^g  \el! \Big(  \sum_{\deg(P) \leq y} \frac{|a(P)|^2}{|P|} \Big)^\el.    
   \label{second}
   \end{equation}
 \end{lem}
 \begin{proof}
 We extend $a(P)$ to a completely multiplicative function. We have
 \begin{equation}
  \Big|   \sum_{\deg(P) \leq y} \frac{ \chi(P) a(P)}{\sqrt{|P|}} \Big|^{2\el} = (\el!)^2 \sum_{\substack{P|fh \Rightarrow \deg(P) \leq y \\ \Omega(f)=\el \\ \Omega(h)=\el}} \frac{a(f) \overline{a(h)} \nu(f) \nu(h) \chi(fh^2)}{\sqrt{|fh|}}.
  \label{expand}
  \end{equation}
 Note that
 \begin{equation*}
  \sum_{\chi \in \mathcal{C}(g)} \Big|   \sum_{\deg(P) \leq y} \frac{ \chi(P) a(P)}{\sqrt{|P|}} \Big|^{2\el} \leq \sum_{F \in \mathcal{M}_{q^2,g/2+1}} \Big|   \sum_{\deg(P) \leq y} \frac{ \chi_F(P) a(P)}{\sqrt{|P|}} \Big|^{2\el}.
 \end{equation*}
 Using the above and equation \eqref{expand}, note that if $fh^2$ is not a cube, then the character sum over $F \in \mathcal{M}_{q^2,g/2+1}$ vanishes since $\deg{(fh^2}) \leq 3 \el y \leq g/2+1$ by hypothesis. Then
 \begin{equation*}
 \sum_{\chi \in \mathcal{C}(g)} \Big|   \sum_{\deg(P) \leq y} \frac{ \chi(P) a(P)}{\sqrt{|P|}} \Big|^{2\el} \leq q^{g+2} (\el!)^2 \sum_{\substack{P|fh \Rightarrow \deg(P) \leq y \\ \Omega(f)=\el \\ \Omega(h)=\el \\ fh^2=\tinycube}} \frac{a(f) \overline{a(h)} \nu(f) \nu(h) \phi_{q^2}(fh^2)}{\sqrt{|fh|}|fh^2|_{q^2}}.
 \end{equation*}
 The condition $fh^2=\cube$ can be rewritten as $f=bf_1^3$ and $h=bh_1^3$ with $(f_1,h_1)=1$. Then we get that
 \begin{align}
 \sum_{\chi \in \mathcal{C}(g)} & \Big|   \sum_{\deg(P) \leq y} \frac{ \chi(P) a(P)}{\sqrt{|P|}} \Big|^{2\el} \leq q^{g+2} (\el!)^2 \sum_{\substack{P|b \Rightarrow \deg(P) \leq y \\ \Omega(b) \leq \el \\ \Omega(b) \equiv \el \pmod 3}} \frac{ |a(b)|^2 \nu(b)}{|b|} \Big( \sum_{\substack{ P|f \Rightarrow \deg(P) \leq y \\ \Omega(f) = (\el- \Omega(b))/3 }} \frac{ |a(f)|^3 \nu(f)}{|f|^{3/2} 3^{\Omega(f)}} \Big)^2.  \nonumber \\
 &= q^{g+2} (\el!)^2 \sum_{\substack{P|b \Rightarrow \deg(P) \leq y \\ \Omega(b) \leq \el \\ \Omega(b) \equiv \el \pmod 3}} \frac{ |a(b)|^2 \nu(b) }{|b|} \frac{1}{(( (\el-\Omega(b))/3)!)^2} \Big( \sum_{\deg(P) \leq y}  \frac{|a(P)|^3}{3|P|^{3/2}} \Big)^{2(\el-\Omega(b))/3}, \label{in2}
 \end{align}
 where we used the fact that $\nu(ab) \leq \nu(a) \nu(b)$, $\nu(f^3) \leq \nu(f)/3^{\Omega(f)}$ and $\nu(b)^2 \leq \nu(b)$, and we ignored the condition that $(f_1,h_1)=1$.

 We further get that the above is
\begin{align} \label{common}
&\ll q^{g} (\el!)^2 \sum_{\substack{i=0 \\ i \equiv \el \pmod 3}}^\el \Big( \sum_{\deg(P) \leq y} \frac{|a(P)|^2}{|P|} \Big)^i \frac{1}{  i! ( (\el-i)/3)!)^23^{2(\el-i)/3}} \\
& \ll q^{g} \frac{ (\el!)^2}{9^{\el/3}} \Big( \sum_{\deg(P) \leq y} \frac{|a(P)|^2}{|P|} \Big)^\el \sum_{j=0}^{\lfloor \el/3\rfloor} \frac{9^j}{(3j)! ( \frac{\el}{3}-j)!^2}, 
\label{common2}
\end{align}
where we use the fact that $|a(P)| \ll 1$ and $\sum_{n=1}^{\infty} \frac{1}{nq^{n/2}}<1$, and that the sum over primes in \eqref{in2} is bounded, to get the first line above. Using the trinomial expansion formula, we get
\[ \sum_{j=0}^{\lfloor \el/3\rfloor} \frac{9^j}{(3j)! ( \frac{\el}{3}-j)!^2}\leq \sum_{j=0}^{\lfloor \el/3\rfloor} \frac{3^{2j}}{(2j)! ( \frac{\el}{3}-j)!^2}\leq \sum_{a+b+c=\lfloor 2\el/3\rfloor } \frac{3^a}{a!b!c!}\leq \frac{5^{2\el/3}}{\lfloor 2\el/3 \rfloor !}\]
Replacing in \eqref{common2} and then \eqref{in2}, we then get that
\begin{equation*}
  \sum_{\chi \in \mathcal{C}(g)} \Big|   \sum_{\deg(P) \leq y} \frac{ \chi(P) a(P)}{\sqrt{|P|}} \Big|^{2\el} \ll q^g \frac{ (\el!)^2 5^{2\el/3} }{\lfloor 2\el/3 \rfloor ! 9^{\el/3}} \Big(  \sum_{\deg(P) \leq y} \frac{|a(P)|^2}{|P|} \Big)^\el.
 \end{equation*}

Now let
$$x = \sum_{\deg(P) \leq y} \frac{ |a(P)|^2}{|P|},$$ and we assume that $\el \leq x^{3-\varepsilon}$. We claim that for $i \leq \el$ with $i \equiv \el \pmod 3$, we have
\begin{equation} \frac{x^i 3^{2i/3}}{i! ( (\el-i)/3)!^2} \ll \frac{x^\el}{\el!}.\label{in1}
\end{equation}
Using Stirling's formula, we need to show that for $\el \leq x^{3-\varepsilon}$, we have 
$$ \frac{2i}{3} \log 3+ \el \log \el - \el -i \log i +i - \frac{2(\el-i)}{3} \log \Big( \frac{\el-i}{3} \Big)+\frac{2(\el-i)}{3} \leq  (\el-i) \log x+\log C,$$ for some constant $C$.
Now let $$f(i)= i \log x+\frac{2i}{3} \log 3+ \el \log \el - \el -i \log i +i - \frac{2(\el-i)}{3} \log \Big( \frac{\el-i}{3} \Big)+\frac{2(\el-i)}{3}.$$ Then
$$f'(i) = \log(3^{2/3}x) - \log i+ \frac{2}{3} \log \Big(\frac{\el-i}{3} \Big),$$
and $f$ attains its maximum on $[0,\el]$ at $i$ with $i^3=x^3(\el-i)^2$. Since $\el \leq x^{3-\varepsilon}$ it follows that $f$ attains its maximum at some $i_0$ with $i_0>\el/2$. Indeed, if we suppose that $i_0 \leq \el/2$ then $\el-i_0 \geq \el/2$, and since $x^3>\el$ it follows that $i_0^3 > \el^3/4$ which is a contradiction since we assumed that $i_0^3 \leq \el^3/8$. Let $i_1=i_0/\el$. 
We have $1/2<i_1<1$. Then
$$f(i_0)= \el i_1 \log x+ \frac{\el(1-i_1)}{3} \log \el+ \frac{2i_0}{3} \log 3-\el-\el i_1 \log i_1+i_0-\frac{2\el(1-i_1)}{3} \log \Big(\frac{1-i_1}{3} \Big)+\frac{2\el(1-i_1)}{3}.$$
Since $1/2<i_1<1$ it follows that $$f(i_0) \leq \el \log x,$$ which establishes \eqref{in1}. Combining \eqref{common} and \eqref{in1} and since $\el /3^{2\el/3}<1$, the conclusion follows.
 \end{proof}

 \begin{proof}[Proof of Proposition \ref{lemma_sound}]
 The proof is similar to the proof of Corollary $A$ in \cite{Sound-again}. Let
 $$N(V) = \Big |\Big \{ \chi \text{ primitive cubic } \text{ genus}(\chi)=g : \log |{\textstyle L(\frac{1}{2}, \chi)} | \geq V  \Big \} \Big|.$$
 Then 
 \begin{align}
 \sum_{\chi \in \mathcal{C}(g)} |{\textstyle L(\frac{1}{2},\chi)} |^{2k} =2k \int_{-\infty}^{\infty} \exp (2kV) N(V) \, dV. \label{mom}
 \end{align}
 In equation \eqref{chandee} with $k=1$, note that we can bound the contribution from primes square by $O(\log \log g)$. Indeed, we split the sum over $P$ with $\deg(P) \leq N/2$ into primes $P$ with $\deg(P) \leq 4 \log_q g$ and primes $P$ with $4 \log_q g<\deg(P) \leq N/2$. For the first term, we use the trivial bound which gives the bound $O(\log \log g)$. For the second term, we use the Weil bound \eqref{weil} yielding an upper bound of size $o(1)$. 
 So we have
 \begin{eqnarray}
\log  | {\textstyle L(\frac{1}{2},\chi) }| &\leq &     \Re \Big( \sum_{\deg(P) \leq N} \frac{ \chi(P)(N-\deg(P)}{N|P|^{\frac{1}{2} +\frac{1}{N \log q}}}  \Big) + \frac{g+2}{N}+O(\log \log g).
\end{eqnarray}
Let $$\frac{g+2}{N} = \frac{V}{A}$$ and $N_0=N/\log g$, where
\begin{equation} \label{values-A}
A = 
\begin{cases}
\frac{\log \log g}{2} & \mbox{ if } V \leq \log g,\\
\frac{\log g}{2V} \log \log g & \mbox{ if } \log g< V \leq \frac{1}{12} \log g \log \log g, \\
6 & \mbox{ if } \frac{1}{12} \log g \log \log g  <V.
\end{cases}
\end{equation}
We only need to consider $\sqrt{\log g}<V$. Indeed, note that the contribution from $V \leq \sqrt{\log g}$ in the integral on the right-hand side of \eqref{mom} is $o(q^gg^{k^2})$,  by trivially bounding $N(V) \ll q^g$.  If $\chi$ is such that $\log |{\textstyle L(\frac{1}{2}, \chi)} | \geq V$, then
$$   \Re \Big( \sum_{\deg(P) \leq N} \frac{ \chi(P)(N-\deg(P)}{N|P|^{\frac{1}{2} +\frac{1}{N \log q}}}  \Big)  
\geq V - \frac{V}{A} + O \left( \log{\log{g}} \right) \geq  V \Big(1-\frac{2}{A} \Big),$$
for $g$ large enough since $\sqrt{\log g}<V$.

Let $$S_1(\chi)=  \left| \sum_{\deg(P) \leq N_0} \frac{ \chi(P)(N-\deg(P)}{N|P|^{\frac{1}{2} +\frac{1}{N \log q}}} \right|, \;\;\;
S_2(\chi)=  \left| \sum_{N_0 < \deg(P) \leq N} \frac{ \chi(P)(N-\deg(P)}{N|P|^{\frac{1}{2} +\frac{1}{N \log q}}} \right|.$$
Then if $\log |{\textstyle L(\frac{1}{2}, \chi)} | \geq V$,
either $$S_2(\chi) \geq V/A \;\;\;\text{or} \;\;\;S_1(\chi) \geq V(1-3/A):=V_1.$$ Let 
\begin{eqnarray*}
\mathcal{F}_1 &=& \{ \chi \text{ primitive cubic, genus}(\chi)=g : S_1(\chi) \geq V_1 \}\\
\mathcal{F}_2 &=&  \{ \chi \text{ primitive cubic, genus}(\chi)=g : S_2(\chi) \geq V/A \}.
\end{eqnarray*}

Using \eqref{first} of Lemma \ref{dir_pol_2}, we get 
$$ | \mathcal{F}_2| \leq \sum_{\chi \in \mathcal{C}(g)} \left( \frac{S_2(\chi)}{V/A} \right)^{2\el} \ll q^g \Big(\frac{A}{V} \Big)^{2\el} \frac{(\el!)^2 (25/9)^{\el/3} }{\lfloor2\el/3\rfloor !} \Big( \sum_{N_0 < \deg(P) \leq N} \frac{|a(P)|^2}{|P|} \Big)^\el,$$ for any $\el$ such that $3\el N \leq g/2+1 \iff \el \leq V/(6A)$ and where $a(P) = (N-\deg(P))/(N|P|^{1/N\log q})$. Picking $\el =6\lfloor V/(36A)\rfloor$, this gives
\begin{equation}
|\mathcal{F}_2| \ll q^g  \Big(\frac{A}{V} \Big)^{2\el} \Big(\frac{\el}{e} \Big)^{4\el/3} (5/2)^{2\el/3} (\log \log g)^\el \ll q^g \exp \Big( - \frac{V}{10 A} \log V \Big).
\label{f2}
\end{equation}

If $\chi \in \mathcal{F}_1$ and $V \leq (\log g)^{2-\varepsilon}$ then we pick $\el=\lfloor V_1^2/\log g \rfloor$. 
Note that since $a(P)=(N-\deg(P))/(N |P|^{1/N \log g})$ we have that $\sum_{\deg(P ) \leq N_0} |a(P)|^2/|P| = \log g + o(\log g)$ and then $\el \leq (\sum_{\deg(P) \leq N_0} |a(P)|^2/|P|)^{3-\varepsilon}$. We can then apply \eqref{second} of Lemma \ref{dir_pol_2}, and we get that
$$| \mathcal{F}_1| \leq \sum_{\chi \in \mathcal{C}(g)} \left( \frac{S_1(\chi)}{V_1} \right)^{2\el} \ll q^g \sqrt{\el} \exp \Big(\el \log \Big( \frac{\el \log g}{eV_1^2} \Big) \Big) \ll q^g \frac{V}{\sqrt{\log g}} \exp \Big(- \frac{V_1^2}{\log g} \Big).
$$
If $V>(\log g)^{2-\varepsilon}$, then we pick $\el=18V$ and apply \eqref{first} to get that
$$  | \mathcal{F}_1| \ll q^g  \Big(\frac{\el^{4/3}  25^{1/3}  \log g}{e^{4/3} 4^{1/3} V_1^2 } \Big)^\el \ll q^g \exp(-2V \log V). $$

Using the above and the values for $A$ of \eqref{values-A}, 
we proved that:\\
if $\sqrt{\log g} \leq V \leq \log g$, then 
\begin{equation}
\label{n1}
N(V) \ll q^g \exp \Big(- \frac{V^2}{\log g} \Big(1-\frac{6}{ \log \log g} \Big)^2 \Big);
\end{equation}
if $\log g < V \leq \frac{1}{12} \log g \log \log g$, then
\begin{equation}
\label{n2}
N(V) \ll q^g \exp \Big(-\frac{V^2}{ \log g} \Big(1- \frac{6V}{ \log g \log \log g} \Big)^2  \Big);
\end{equation}
if $V> \frac{1}{12} \log g \log \log g$, then
\begin{equation}
N(V) \ll q^g \exp \Big(- \frac{V \log V}{60} \Big).
\label{n3}
\end{equation}
Now we use the bounds \eqref{n1}, \eqref{n2}, \eqref{n3} in the form $N(V) \ll q^g g^{o(1)} \exp(-V^2/\log g)$ if $V \leq 4 k \log g$ and $N(V) \ll q^g g^{o(1)} \exp(-4kV)$ if $V>4k \log g$ in equation \eqref{mom} to prove Proposition \ref{lemma_sound}.
Indeed, we have
\begin{eqnarray*}
 \sum_{\chi \in \mathcal{C}(g)} |{\textstyle L(\frac{1}{2},\chi)} |^{2k} &\ll_k& q^g g^{o(1)} \int_{\sqrt{\log{g}}}^{4k \log{g}} \exp (2kV - V^2/\log{g}) \, dV + q^g g^{o(1)} \int_{4 k \log{g}}^{\infty} \exp (-2kV) \, dV \\
 &\ll_k& q^g g^{o(1)}  \exp{(k^2 \log{g})},
\end{eqnarray*}
and the desired upper bound follows. As mentioned in \cite{Sound-again}, it is interesting to remark that the proof suggests that the dominant contribution for the $2k^{\text{th}}$ moment comes from the characters $\chi$ such that $|{\textstyle L(\frac{1}{2}, \chi)} |$ has size $g^k$, and the measure of this set is about $q^g g^{-k^2}$.
\end{proof}

 \section{Explicit upper bound for mollified moments} \label{explicit-UB}

Here we will obtain an explicit upper bound for expression \eqref{before-explicit}, which means that we want to find an upper bound for $C_J$ from \eqref{tobound2}
by choosing $\theta_\JJ , a,b$ and $d$ subject to the constraints in Lemma \ref{estimates} and subject to \eqref{important_condition}.

   Let
$$f(u)=\A e^u-\B  ue^u+\frac{\kk^2u \theta_\JJ }{2},$$ with 
\begin{equation}
\A =\kk e+\frac{\alpha}{2d} \log \theta_\JJ  + \frac{\log F}{2d}, \B  = \frac{\alpha}{2d} ,
\label{ab}
\end{equation} 
where recall that
$$\alpha= 2 b-2+\frac{c}{3} , F= \frac{\kk^2 e^{2+c/3} 5^{c/3}}{4d^{2-c/3}c^{c/3}}, c=2-4/a,$$ and $a$ and $d$ are as in Lemma \ref{estimates}. 
We will pick $\theta_\JJ $ subject to the condition \eqref{important_condition} and such that $\A >0$. 

We have $$f'(u) = e^u(\A -\B  -\B  u)+ \frac{\kk^2 \theta_\JJ }{2},$$
and notice that for $u \leq (\A -\B  )/\B  $ we have $f'(u)>0$  so $f$ is increasing on $[0,(\A -\B  )/\B  ]$ i.e: $f$ is increasing on $[0, \frac{2d\kk e}{\alpha}+\log \theta_\JJ +\frac{\log F}{\alpha}-1]$.
Also note that
$$f'(\A /\B  )= -\B  e^{\A /\B  }+ \frac{\kk^2 \theta_\JJ }{2}<0,$$ so the maximum of $f$ occurs at some $m \in (\A /\B  -1,\A /\B  )$. 
With the above notation, we write 
$$C_J \sim \int_0^{\JJ -1} (u+1) \exp \Big( \frac{1}{\theta_\JJ } (\A e^u-\B   u e^u + \frac{\kk^2 u \theta_\JJ }{2} ) \Big) \, du.$$
For $u \geq 4\A /\B  $ we have $\A e^u+ \kk^2 u \theta_\JJ /2<\B  ue^u/2$, so
\begin{align}
\int_{4A/B}^{\JJ -1} (u+1) \exp \Big( \frac{1}{\theta_\JJ } (\A e^u-\B   u e^u + \frac{\kk^2 u \theta_\JJ }{2} ) \Big) \, du \leq \int_{4\A /\B  }^{\infty} e^{-u} \, du = e^{-4\A /\B  }.
\label{in21}
\end{align}
Now
\begin{align}
\int_0^{4\A /\B  } & (u+1) \exp \Big( \frac{1}{\theta_\JJ } (\A e^u-\B   u e^u + \frac{\kk^2 u \theta_\JJ }{2} ) \Big) \, du \nonumber\\ &\leq  \frac{4\A }{\B  } \Big(\frac{4\A }{\B  }+1 \Big) \exp \Big( \frac{1}{\theta_\JJ } (\A e^m-\B   m e^m + \frac{\kk^2 m \theta_\JJ }{2} ) \Big)\nonumber  \\
&  \leq \frac{4\A }{\B  }\Big(\frac{4\A }{\B  }+1 \Big) \exp \Big(\frac{\kk^2\A }{2\B  } \Big) \exp \Big(\frac{\B  e^{\frac{\A }{\B  }-1}}{\theta_\JJ } \Big), \label{in22}
\end{align}
where in the third line we used the fact that $m \in (\A /\B  -1,\A /\B  )$. Combining \eqref{in21} and \eqref{in22} we get that
$$ C_J \leq e^{-4\A /\B  }+\frac{4\A }{\B  } \Big( \frac{4\A }{\B  }+1 \Big) \exp \Big(\frac{\kk^2\A }{2\B  } \Big) \exp \Big(\frac{\B  e^{\frac{\A }{\B  }-1}}{\theta_\JJ } \Big).$$
Now using this inequality back in \eqref{before-explicit}, we get that

\begin{align}
\sum_{\substack{ \chi \in  \mathcal{C}(g) }} |{\textstyle L(\frac{1}{2},\chi)}|^{\kk} & |{\textstyle M(\chi;\frac{1}{\kappa})}|^{\kk \kappa} \leq_\varepsilon 
q^{g+2}\mathcal{D}_k^{1/2} \mathcal{S}_k^{1/2} \exp{\Big( \frac{3}{2} k^2 + (1+\eta) k \Big)}\\
&\times \left( \exp(k/\theta_J) +  \sqrt[4]{\frac{24}{c}}  \Big[e^{-4\A /\B  }+\frac{4\A }{\B  } \Big( \frac{4\A }{B}+1 \Big) \exp \Big(\frac{\kk^2\A }{2\B  } \Big) \exp \Big(\frac{\B  e^{\frac{\A }{\B}-1}}{\theta_\JJ } \Big) \Big]  \right),\label{explicit_ub}
\end{align}
where recall that $\A $ and $\B  $ are given in \eqref{ab}, and where $\mathcal{S}_k$ is defined in Lemma \ref{estimates-squares}.


From the explicit upper bound above, we remark that because of the term $\exp(\B  e^{\A /\B  -1}/\theta_\JJ )$, the upper bound we obtain is of the form $e^{e^{O(\kk)}}$.

   Now we take $\kappa=1, \kk=2$. Condition \eqref{important_condition} becomes
   $$ 10 \sum_{r=0}^\JJ \theta_r  \ell_r + \frac{4}{d} \leq \frac{1}{2}.$$ Note that any $\theta_\JJ $ with
   $$\theta_\JJ ^{1-b} \frac{e^{1-b}}{e^{1-b}-1} \leq \frac{d-8}{40d},$$ satisfies the condition above. 
   We will pick $\theta_\JJ $ such that
   \begin{equation}
     \theta_\JJ = \Big(\frac{d-8}{40d} \Big(1-\frac{1}{e} \Big)\Big)^{\frac{1}{1-b}}.
     \label{opt2}
   \end{equation}

 
   Now, in equation \eqref{explicit_ub}, in order to obtain an optimal constant, we set 
   $$ \frac{1}{\theta_\JJ }= e^{\A /\B  },$$ and the term $\log F/(2d)$ in the expression for $\A $ is small compared to the rest, so in order to optimize the constant, we set
   \begin{equation}
    \log \frac{1}{\theta_\JJ } = \frac{2de}{2b-2+\frac{c}{3}}. \label{opt1}
    \end{equation} 
Now from Lemma \ref{estimates}, we need 
$$4ad \theta_\JJ ^{1-b} \leq 1,$$ so combining this with \eqref{opt2} it follows that
$$c \leq 2-\frac{2(d-8)(e-1)}{5e}.$$     
Now, to minimize \eqref{opt1}, we need $c$ to be maximal, so we will pick
\begin{equation}
\label{c}
c= 2-\frac{2(d-8)(e-1)}{5e}.
\end{equation}
From \eqref{opt2}, \eqref{opt1} and \eqref{c}, it follows that
\begin{equation}
 b= 1- \frac{cx}{6(de+x)},
 \label{b}
 \end{equation}
  where $x= \log (40 de/((d-8)(e-1)))$. With choices \eqref{b} and \eqref{c} for $b$ and $c$, we want to minimize \eqref{opt1}, and this translates into minimizing the function of $d$ given by
  $$\frac{de+x}{1-\frac{(d-8)(e-1)}{5e}},$$ for $d>8$. 
  
  The minimum of the function above is achieved for 
  \begin{equation}
  d \approx 8.15, \label{d}
  \end{equation} and in that case
  $$ \log \frac{1}{\theta_\JJ } = \frac{2de}{2b-2+\frac{c}{3}} \approx 92.65. $$
  With the choice above for $d$, we get that
  \begin{equation}
  b \approx 0.91 , \, c \approx 1.96.
  \label{b-c}
  \end{equation}
  \kommentar{ From \eqref{opt1} and \eqref{opt2}, we get that 
   \begin{equation} 
   c= \frac{6(1-b)(de+x)}{x},\label{c}
   \end{equation} where $x= \log(40de/((d-8)(e-1)).$ Since we need $c<2$, from \eqref{c} it follows that
   $$b>1-\frac{x}{3(de+x)},$$ and we pick $b$ such that 
   \begin{equation}
   1-b= \frac{x(1-\varepsilon)}{3(de+x)}, \label{b}
   \end{equation} for a small $\varepsilon>0$.
   Then we want to minimize \eqref{opt1} subject to $d>8$. We have
   $$\frac{2de}{2b-2+\frac{c}{3}}= \frac{3(de+x)}{1-\varepsilon},$$ and when $d>8$, this attains its minimum when 
   \begin{equation}
   d= \frac{8+\sqrt{64+\frac{32}{e}}}{2}.
   \label{d}
   \end{equation}}

   Choosing $b,c,d$ as in \eqref{d}, \eqref{b-c} we obtain the upper bound
   \begin{equation}    \sum_{\chi \in \mathcal{C}(g)} |{\textstyle L(\frac{1}{2},\chi)}|^{ 2}  |M(\chi;1)|^2 \leq_{\varepsilon} e^{e^{182}} q^{g+2}.
   \label{big_bound}
   \end{equation}

\section{The mollified first moment} \label{sec:first-moment}

Here, we will prove Theorem \ref{first-moment}. We consider the mollified first moment with $\kappa=1$. We have
\begin{equation}
\label{mchi}
M(\chi):=M(\chi,1) = \sum_{\substack{h_0 \cdot \ldots \cdot h_\JJ =h\\ P|h_j \Rightarrow P \in I_j\\ \Omega(h_j) \leq  \ell_j}} \frac{a(h; \JJ ) \chi(h) \lambda(h) \nu(h_0) \cdot \ldots \cdot \nu(h_\JJ )}{\sqrt{|h|}},
\end{equation}
and then

\begin{align}
\sum_{\chi \in \mathcal{C}(g)} {\textstyle L(\frac{1}{2}, \chi)} M(\chi) = \sum_{\substack{h_0 \cdot \ldots \cdot h_\JJ =h\\ P|h_j \Rightarrow P \in I_j\\ \Omega(h_j) \leq  \ell_j}} \frac{a(h; \JJ )  \lambda(h) \nu(h_0) \cdot \ldots \cdot \nu(h_\JJ )}{ \sqrt{|h|}} \sum_{\chi \in \mathcal{C}(g)} \chi(h) {\textstyle L(\frac{1}{2}, \chi)} .\label{mom1}
\end{align} 
We will evaluate the twisted first moment in the lemma below. 
\begin{prop}\label{twist}
Let $q \equiv 2 \pmod 3$, and let $h$ be a polynomial in $\mathbb{F}_q[T]$ with $\deg(h) < g ( \frac{1}{10}-\varepsilon)$. Let $h=CS^2E^3$, where $C$ and $S$ are square-free and coprime. Then we have
\begin{align*}
 \sum_{\chi \in \mathcal{C}(g)} \chi(h) {\textstyle L(\frac{1}{2}, \chi)}  =& \frac{ q^{g+2} \zeta_q(3/2)}{\zeta_q(3)|C| \sqrt{|S|}} \mathcal{A}_{\mathrm{nK}} \Big( \frac{1}{q^2}, \frac{1}{q^{3/2}} \Big)  \prod_{\substack{R \in \mathbb{F}_q[T] \\ \deg(R) \text{ even} \\ R|h}} M_R \Big( \frac{1}{q^2},\frac{1}{q^{3/2}} \Big) \\
 &+ O \Big(q^{\frac{7g}{8}+ \frac{\deg(h)}{4}+\varepsilon g} \Big),
 \end{align*}
 where $\mathcal{A}_{\mathrm{nK}} \Big( \frac{1}{q^2}, \frac{1}{q^{3/2}} \Big)$ and $M_R \Big( \frac{1}{q^2},\frac{1}{q^{3/2}} \Big) $ are given in equations \eqref{a-euler} and \eqref{mr}.
\end{prop}

\begin{proof}[Proof of Proposition \ref{twist}]

The proof is similar to the proof of Theorem $1.1$  in \cite{DFL}. Using the explicit description of the characters $\chi \in \mathcal{C}(g)$ given by  
 \eqref{square-free-sum}, and Proposition \ref{prop-AFE}, we write
\begin{eqnarray*}
\sum_{\chi \in \mathcal{C}(g)} \chi(h) {\textstyle L(\frac{1}{2}, \chi)}  &=&  S_{1, \mathrm{principal}}  + S_{1, \mathrm{dual}}, 
\end{eqnarray*}
 where
 \begin{equation}
 S_{1, \mathrm{principal}} = \sum_{f \in \mathcal{M}_{q, \leq X}}\frac{1}{q^{\deg(f)/2}} \sum_{\substack{F \in \mathcal{H}_{q^2, g/2+1} \\ P \mid F \Rightarrow P \not\in \F_q[T]}}  \chi_F(fh)
 + \frac{1}{1-\sqrt{q}} \sum_{f \in \mathcal{M}_{q, X+1}} \frac{1}{q^{\deg(f)/2}} \sum_{\substack{F \in \mathcal{H}_{q^2, g/2+1} \\ P \mid F \Rightarrow P \not\in \F_q[T]}}  \chi_F(fh)
 \label{princ2}
 \end{equation} and
 \begin{align}
 S_{1, \mathrm{dual}} =& \sum_{f \in \mathcal{M}_{q, \leq g-X-1}}\frac{1}{q^{\deg(f)/2}}
 \sum_{\substack{F \in \mathcal{H}_{q^2, g/2+1} \\ P \mid F \Rightarrow P \not\in \F_q[T]}}  \omega(\chi_F) \overline{\chi_F}(fh^2) \\
& +\frac{1}{1-\sqrt{q}} \sum_{f \in \mathcal{M}_{q, g-X}} \frac{1}{q^{\deg(f)/2}} \sum_{\substack{F \in \mathcal{H}_{q^2, g/2+1} \\ P \mid F \Rightarrow P \not\in \F_q[T]}} \omega(\chi_F)  \overline{\chi_F}(fh^2).
 \label{dual*}
 \end{align}
We will choose $X \equiv 2 \deg(h) \pmod 3$. For the principal term, we will compute the contribution 
from polynomials $f$ such that $fh$ is a cube and bound the contribution from $fh$ non-cube. We write
 $$S_{1,\mathrm{principal}} = S_{1,\cube}+S_{1, \neq \cube},$$ where 
$S_{1,\cube}$ corresponds to the sum with $fh$ a cube in equation \eqref{princ2} and $S_{1,\neq \cube}$ corresponds to the sum with $fh$ not a cube, namely,
\begin{equation}
S_{1,\cube} = \sum_{\substack{f \in \mathcal{M}_{q, \leq X}\\fh = \tinycube}} \frac{1}{q^{\deg(f)/2}} \sum_{\substack{F \in \mathcal{H}_{q^2, g/2+1} \\ (F,fh) =1 \\ P \mid F \Rightarrow P \not\in \F_q[T]}} 1,
\label{maint}
\end{equation} 
and 
\begin{equation}\label{non-cubes}
S_{1, \neq \cube} =  \sum_{\substack{f \in \mathcal{M}_{q, \leq X} \\ fh  \neq \tinycube}} \frac{1}{q^{\deg(f)/2}} \sum_{\substack{F \in \mathcal{H}_{q^2, \frac{g}{2}+1} \\ P \mid F \Rightarrow P \not\in \F_q[T] }} {\chi_F}(fh)+ \frac{1}{1-\sqrt{q}} \sum_{f \in \mathcal{M}_{q,X+1}} \frac{1}{q^{\deg(f)/2}}\sum_{\substack{F \in \mathcal{H}_{q^2, \frac{g}{2}+1} \\ P \mid F \Rightarrow P \not\in \F_q[T]}}
 {\chi_F}(fh).
\end{equation}
 
Since $X \equiv 2 \deg(h) \pmod 3$, note that the second term in \eqref{princ2} does not contribute to the expression \eqref{maint}
for $S_{1,\cube}$. 

\kommentar{\acom{Bounding $S_{1, \neq \cube}$ is exactly analogous to the proof in \cite{DFL}, and we get that
 \begin{equation}
 S_{1, \neq \cube} \ll q^{\frac{X+g}{2}+\varepsilon g}.\label{non-cube}
 \end{equation}
 }\mcom{This is done in section 8.2, do we need to keep this comment?}}
\subsection{The main term}

Now we focus on $S_{1,\cube}$. Since $h= CS^2E^3$, where $C, S$ are square-free and $(C,S)=1$ and $fh=\cube$, it follows that we can write $f= C^2 S K^3$. Then
 \begin{equation}\label{eq:mt}
S_{1,\cube} = \sum_{K \in \mathcal{M}_{q, \leq \frac{X-\deg(C^2D)}{3}}} \frac{1}{|C|_q \sqrt{|S|_q} |K|_q^{3/2}} \sum_{\substack{F \in \mathcal{H}_{q^2, g/2+1} \\ (F,Kh) =1 \\ P \mid F \Rightarrow P \not\in \F_q[T]}} 1.
\end{equation}
We first look at the generating series of the sum over $F$.  We use the fact that 
\begin{equation}
\sum_{\substack{D \in \F_q[T]\\D \mid F}} \mu(D) = \begin{cases} 1 & \mbox{if $F$ has no prime divisor in $\F_q[T]$,} \\ 0 & \mbox{otherwise}, \end{cases}
\label{sieve_F}
\end{equation}
where $\mu$ is the M\"obius function over $\F_q[T]$. 
The generating series corresponding to the inner sum in \eqref{eq:mt} is
\begin{equation}
\sum_{\substack{F \in \mathcal{H}_{q^2} \\ (F,Kh) =1 \\ P \mid F \Rightarrow P \not\in \F_q[T]}} x^{\deg(F)} 
= \sum_{\substack{F \in \mathcal{H}_{q^2} \\ (F,Kh)=1}} x^{\deg(F)} \sum_{\substack{D \in \mathbb{F}_q[T] \\ D\mid F}} \mu(D) = \sum_{\substack{D \in \mathbb{F}_q[T] \\ (D,Kh)=1}} \mu(D) x^{\deg(D)} \sum_{\substack{F \in \mathcal{H}_{q^2} \\ (F,DKh)=1}} x^{\deg(F)}.
\label{sum_F}
\end{equation}
We evaluate the sum over $F$ and we have that
$$ \sum_{\substack{F \in \mathcal{H}_{q^2} \\ (F,KDh)=1}} x^{\deg(F)} = \prod_{\substack{P \in \mathbb{F}_{q^2}[T] \\ P \nmid DKh}} (1+x^{\deg(P)}) = \frac{ \mathcal{Z}_{q^2}(x)}{\mathcal{Z}_{q^2}(x^2)\displaystyle  \prod_{\substack{P \in \mathbb{F}_{q^2}[T] \\ P \mid DKh}} (1+x^{\deg(P)})},$$
and combining the above with equation \eqref{sum_F}, it follows that
$$ \sum_{\substack{F \in \mathcal{H}_{q^2} \\ (F,Kh) =1 \\ P \mid F \Rightarrow P \not\in \F_q[T]}} x^{\deg(F)}  = \frac{ \mathcal{Z}_{q^2}(x)}{\mathcal{Z}_{q^2}(x^2) \displaystyle \prod_{\substack{P \in \mathbb{F}_{q^2}[T] \\ P\mid Kh}} (1+x^{\deg(P)})} \sum_{\substack{D \in \mathbb{F}_q[T] \\ (D,Kh)=1}} \frac{ \mu(D)x^{\deg(D)}}{ \displaystyle \prod_{\substack{P \in \mathbb{F}_{q^2}[T] \\ P\mid D}} (1+x^{\deg(P)})}.$$

Now we write down an Euler product for the sum over $D$ and we have that
\begin{align}
\sum_{\substack{D \in \mathbb{F}_q[T] \\ (D,Kh)=1}} \frac{ \mu(D)x^{\deg(D)}}{ \displaystyle \prod_{\substack{P \in \mathbb{F}_{q^2}[T] \\ P\mid D}} (1+x^{\deg(P)})} = \prod_{\substack{R \in \mathbb{F}_q[T] \\ (R,Kh)=1 \\ \deg(R)  \,\mathrm{ odd}}} \left ( 1- \frac{x^{\deg(R)}}{1+x^{\deg(R)}} \right )\prod_{\substack{R \in \mathbb{F}_q[T] \\ (R,Kh)=1 \\ \deg(R)  \,\mathrm{ even}}} \left ( 1- \frac{x^{\deg(R)}}{(1+x^{\frac{\deg(R)}{2}})^2} \right ), \label{sum_d}
\end{align} where the product over $R$ is over monic, irreducible polynomials. Let $A_R(x)$ denote the first Euler factor above and $B_R(x)$ the second. 
\kommentar{Then we rewrite the sum over $D$ as
$$ \eqref{sum_d} =  \frac{\displaystyle \prod_{\substack{R \in \mathbb{F}_q[T] \\ \deg(R)  \,\mathrm{ odd}}} A_R(x)  \prod_{\substack{R \in \mathbb{F}_q[T] \\ \deg(R)  \,\mathrm{ even}}} B_R(x)}{\displaystyle \prod_{\substack{R \in \mathbb{F}_q[T] \\ R\mid Kh \\ \deg(R)  \,\mathrm{ odd}}} A_R(x)  \prod_{\substack{R \in \mathbb{F}_q[T] \\ R\mid Kh \\ \deg(R)  \,\mathrm{ even}}} B_R(x) },$$}Using \eqref{sum_d} and putting everything together, it follows that
\begin{equation}
\sum_{\substack{F \in \mathcal{H}_{q^2} \\ (F,Kh) =1 \\ P \mid F \Rightarrow P \not\in \F_q[T]}} x^{\deg(F)}  = \frac{ \mathcal{Z}_{q^2}(x) \displaystyle \prod_{\substack{R \in \mathbb{F}_q[T] \\ \deg(R)  \,\mathrm{ odd}}} A_R(x)  \prod_{\substack{R \in \mathbb{F}_q[T] \\ \deg(R)  \,\mathrm{ even}}} B_R(x) }{\mathcal{Z}_{q^2}(x^2) \displaystyle \prod_{\substack{P \in \mathbb{F}_{q^2}[T] \\ P\mid Kh}} (1+x^{\deg(P)}) \displaystyle \prod_{\substack{R \in \mathbb{F}_q[T] \\ R\mid Kh \\ \deg(R)  \,\mathrm{ odd}}} A_R(x)  \prod_{\substack{R \in \mathbb{F}_q[T] \\ R\mid Kh \\ \deg(R)  \,\mathrm{ even}}} B_R(x)  }  .
\label{sum_F_2}
\end{equation}

We now introduce the sum over $K$ and we get
\begin{align*}
& \sum_{K \in \mathcal{M}_q}  \frac{ u^{\deg(K)} }{ \displaystyle \prod_{\substack{P \in \mathbb{F}_{q^2}[T] \\ P\mid K, P \nmid h}} (1+x^{\deg(P)}) \displaystyle \prod_{\substack{R \in \mathbb{F}_q[T] \\ R\mid K, R \nmid h \\ \deg(R)  \,\mathrm{ odd}}} A_R(x)  \prod_{\substack{R \in \mathbb{F}_q[T] \\ R\mid K, R \nmid h \\ \deg(R)  \,\mathrm{ even}}} B_R(x) } \\
=&\prod_{\substack{R \in \mathbb{F}_q[T] \\ \deg(R)  \,\mathrm{ odd} \\ R \nmid h}} \left[ 1+ \frac{u^{\deg(R)}}{(1+x^{\deg(R)}) A_R(x) (1-u^{\deg(R)})}\right] \prod_{\substack{R \in \mathbb{F}_q[T] \\ \deg(R)  \,\mathrm{ even} \\ R \nmid h}} \left[1+ \frac{u^{\deg(R)}}{(1+x^{\frac{\deg(R)}{2}})^2 B_R(x) (1-u^{\deg(R)})} \right] \\ 
& \times \prod_{\substack{R \in \mathbb{F}_q[T] \\ R|h}} \frac{1}{1-u^{\deg(R)}},
\end{align*}
where $R$ denotes a monic irreducible polynomial in $\mathbb{F}_q[T]$. Combining the equation above and \eqref{sum_F_2}, we get  the generating series 
\begin{align} 
\sum_{K \in \mathcal{M}_q}  u^{\deg(K)}\sum_{\substack{F \in \mathcal{H}_{q^2} \\ (F,Kh) =1 \\ P \mid F \Rightarrow P \not\in \F_q[T]}} x^{\deg(F)}  &=\frac{ \mathcal{Z}_{q^2}(x)}{\mathcal{Z}_{q^2}(x^2)} \prod_{\substack{R \in \mathbb{F}_q[T] \\ \deg(R)  \,\mathrm{ odd} \\ R \nmid h}} \frac{1}{(1+x^{\deg(R)})(1-u^{\deg(R)})} 
\nonumber \\ \nonumber
& \hspace{-2in}  \times \prod_{\substack{R \in \mathbb{F}_q[T] \\ \deg(R)  \,\mathrm{ even} \\ R \nmid h}} \frac{1}{(1+x^{\frac{\deg(R)}{2}})^2} \left ( 1+2x^{\frac{\deg(R)}{2}} + \frac{u^{\deg(R)}}{1-u^{\deg(R)}}\right ) \prod_{\substack{P \in \mathbb{F}_{q^2}[T] \\ P|h}} \frac{1}{1+x^{\deg(P)}} \prod_{\substack{R \in \mathbb{F}_q[T] \\ R|h}} \frac{1}{1-u^{\deg(R)}} \nonumber  \\
=& \mathcal{Z}_q (u) \frac{ \mathcal{Z}_{q^2}(x)}{\mathcal{Z}_{q^2}(x^2)} \mathcal{A}_{\mathrm{nK}}(x,u)  \prod_{\substack{R \in \mathbb{F}_q[T] \\ \deg(R) \text{ even} \\ R|h}} M_R(x,u), \nonumber
\end{align} 
where
\begin{align}
\mathcal{A}_{\mathrm{nK}}(x,u) &=  \prod_{\substack{R \in \mathbb{F}_q[T] \\ \deg(R)  \,\mathrm{ odd}}} \frac{1}{1+x^{\deg(R)}} \prod_{\substack{R \in \mathbb{F}_q[T] \\ \deg(R)  \,\mathrm{ even}}}\frac{1}{(1+x^{\frac{\deg(R)}{2}})^2} \left ( 1+2x^{\frac{\deg(R)}{2}} (1-u^{\deg(R)}) \right ), \label{a-euler}\\
M_R(x,u) &= \frac{1}{1+2x^{\deg(R)/2}(1-u^{\deg(R)})}.\label{mr}
\end{align}
\kommentar{\begin{equation*} 
C_R(x) = \frac{1}{(1+x^{\deg(R)})(1-u^{\deg(R)})},
\end{equation*}
\begin{equation*} 
D_R(x) = \frac{1}{(1+x^{\frac{\deg(R)}{2}})^2} \left ( 1+2x^{\frac{\deg(R)}{2}} + \frac{u^{\deg(R)}}{1-u^{\deg(R)}}\right ).
\end{equation*}
Note that we can rewrite
\begin{align*}
\eqref{double_sum} = \mathcal{Z}_q (u) \frac{ \mathcal{Z}_{q^2}(x)}{\mathcal{Z}_{q^2}(x^2)} \mathcal{A}_{\mathrm{nK}}(x,u)  \prod_{\substack{R \in \mathbb{F}_q[T] \\ \deg(R) \text{ even} \\ R|h}} M_R(x,u), 
\end{align*}
where 
\begin{equation}
M_R(x,u) = \frac{1}{1+2x^{\deg(R)/2}(1-u^{\deg(R)})}.\label{mr}
\end{equation}}

We remark that if $h=1$, the generating series above is the same as in \cite[Section 4.3]{DFL}, and we compute the asymptotic for $S_{1, \cube}$ below in the exact same way, keeping the dependence on $h$. 
Using Perron's formula  (Lemma \ref{perron})  twice in \eqref{eq:mt} and the expression of the generating series above, we get that
$$ S_{1,\cube} = \frac{1}{|C|_q \sqrt{|S|_q}}\frac{1}{(2 \pi i)^2} \oint \oint \frac{ \mathcal{A}_{\mathrm{nK}}(x,u)(1-q^2x^2) \prod_{R|h} M_R(x,u)}{(1-qu)(1-q^2x)(1-q^{3/2}u) x^{\frac{g}{2}+1}(q^{3/2}u)^{\frac{X-\deg(C^2D)}{3}}} \, \frac{dx}{x} \, \frac{du}{u},$$
where we are integrating along circles of radii $|u|<1/q^{\frac{3}{2}}$ and $|x|<1/{q^2}$. As in \cite{DFL}, we have that $\mathcal{A}_{\mathrm{nK}}(x,u)$ is analytic for $|x|<1/q, |xu|<1/q, |xu^2|<1/q^2$. We initially pick $|u|=1/q^{\frac{3}{2}+\varepsilon}$ and $|x|=1/q^{2+\varepsilon}$. We shift the contour over $x$ to $|x|=1/q^{1+\varepsilon}$ and we encounter a pole at $x=1/q^2$. Note that the new double integral will be bounded by $O(q^{\frac{g}{2}+\varepsilon g})$. Then
$$ S_{1,\cube} =\frac{q^{g+2}}{\zeta_q(3) |C|_q \sqrt{|S|_q}} \frac{1}{2 \pi i} \oint \frac{\mathcal{A}_{\mathrm{nK}}( \tfrac{1}{q^2},u) \prod_{R|h} M_R ( \tfrac{1}{q^2},u) }{(1-qu)(1-q^{3/2}u) (q^{3/2}u)^{\frac{X- \deg(C^2D)}{3}}} \, \frac{du}{u}+ O(q^{\frac{g}{2}+\varepsilon g}).$$
We shift the contour of integration to $|u|=q^{-\varepsilon}$ and we encounter two simple poles: one at $u=1/q^{\frac{3}{2}}$ and one at $u=1/q$. Evaluating the residues, we get that
\begin{align}
S_{1,\cube} =& \frac{ q^{g+2} \zeta_q(3/2)}{\zeta_q(3)|C|_q \sqrt{|S|_q}} \mathcal{A}_{\mathrm{nK}} \left ( \frac{1}{q^2}, \frac{1}{q^{3/2}} \right )  \prod_{\substack{R \in \mathbb{F}_q[T] \\ \deg(R) \text{ even} \\ R|h}} M_R \Big( \frac{1}{q^2},\frac{1}{q^{3/2}} \Big) \nonumber \\
&+ \frac{q^{g+2-\frac{X}{6}} \zeta_q(1/2)}{\zeta_q(3)|C^2S|_q^{1/3}} \mathcal{A}_{\mathrm{nK}} \left (\frac{1}{q^2}, \frac{1}{q}  \right )  \prod_{\substack{R \in \mathbb{F}_q[T] \\ \deg(R) \text{ even} \\ R|h}} M_R \Big( \frac{1}{q^2},\frac{1}{q} \Big) +O(q^{g-\frac{X}{2}+\varepsilon g}). \label{scube}
\end{align}


\subsection{The contribution from non-cubes} \label{non-cube}
 
 Let $S_{11}$ be the first term in equation \eqref{non-cubes} and $S_{12}$ the second. Note that it is enough to bound $S_{11}$, since bounding $S_{12}$ will follow in a similar way. We use equation \eqref{sieve_F} again for the sum over $F$ and we have
\begin{equation}
S_{11}=  \sum_{\substack{f \in \mathcal{M}_{q, \leq X} \\ fh \neq \tinycube}} \frac{1}{q^{\deg(f)/2}} \sum_{\substack{D \in \mathcal{M}_{q, \leq \frac{g}{2}+1} \\ (D,f)=1}} \mu(D) \sum_{\substack{F \in \mathcal{H}_{q^2,\frac{g}{2}+1-\deg(D)} \\ (F,D)=1}} \chi_F(fh).
\label{non_cube_term}
\end{equation}
Remark that we used  that $\chi_D(fh)=1$ because $D,f,h \in \mathbb{F}_q[T]$. Looking at the generating series of the sum over $F$, we have
\begin{align*}
 \sum_{\substack{F \in \mathcal{H}_{q^2} \\ (F,D)=1}} \chi_F(fh) u^{\deg(F)} = \prod_{\substack{P \in \mathbb{F}_{q^2}[T] \\ P \nmid Dfh}} \left ( 1+\chi_P (fh) u^{\deg(P)} \right ) = \frac{ \mathcal{L}_{q^2} \left(u, \chi_{fh}\right) }{\mathcal{L}_{q^2}(u^2, \overline{\chi_{fh}})} \prod_{\substack{P \in \mathbb{F}_{q^2}[T] \\ P \nmid fh \\ P\mid D}} \frac{  1- \chi_P(fh) u^{\deg(P)} }{   1- \overline{\chi_P}(fh) u^{2 \deg(P)} }.
\end{align*}
Using Perron's formula  (Lemma \ref{perron}) and the generating series above, we have
\begin{align*}
 \sum_{\substack{F \in \mathcal{H}_{q^2,\frac{g}{2}+1-\deg(D)} \\ (F,D)=1}} \chi_F(fh) = \frac{1}{2 \pi i} \oint \frac{ \mathcal{L}_{q^2} \left(u,\chi_{fh}\right) }{\mathcal{L}_{q^2}(u^2, \overline{\chi_{fh}}) u^{\frac{g}{2}+1-\deg(D)}} \prod_{\substack{P \in \mathbb{F}_{q^2}[T] \\ P \nmid fh \\ P\mid D}} \frac{  1- \chi_P(fh) u^{\deg(P)} }{   1- \overline{\chi_P}(fh) u^{2 \deg(P)} } \, \frac{du}{u},
\end{align*}
where the integral takes place along a circle of radius $|u|= 1/q$ around the origin. Now we use the Lindel\"of bound for the $L$--function in the numerator and a lower bound for the $L$--function in the denominator (equations \eqref{lindelof} and \eqref{folednil}) and we obtain
$$ \left|  \mathcal{L}_{q^2} \left(u,\chi_{fh}\right) \right| \ll q^{2 \varepsilon \deg(fh)}, \, \, \, \left| \mathcal{L}_{q^2} (u^2, \overline{\chi_{fh}})  \right| \gg q^{-2 \varepsilon\deg(fh)}.$$

Therefore, 
$$  \sum_{\substack{F \in \mathcal{H}_{q^2,\frac{g}{2}+1-\deg(D)} \\ (F,D)=1}} \chi_F(fh) \ll q^{\frac{g}{2}-\deg(D)} q^{4 \varepsilon \deg(fh)+2 \varepsilon \deg(D)}.$$
Trivially bounding the sums over $D$ and $f$ in \eqref{non_cube_term} gives a total upper bound of
\begin{equation*}
S_{11}\ll q^{\frac{X+g}{2}+\varepsilon g},
\end{equation*}
and similarly for $S_{12}$.

\subsection{The dual term}

Now we focus on $S_{1  , \text{dual}}$.  From \eqref{dual*}, using \eqref{Gauss-root} and \eqref{gauss-sum-prop},
we have
 \begin{align}
 S_{1, \mathrm{dual}} =& q^{-\frac{g}{2}-1}\sum_{f \in \mathcal{M}_{q, \leq g-X-1}}\frac{1}{q^{\deg(f)/2}}
 \sum_{\substack{F \in \mathcal{H}_{q^2,\frac{g}{2}+1} \\ (F,fh)=1 \\ P \mid F \Rightarrow P \not\in \F_q[T] }}G_{q^2}(fh^2,F) \label{dual21}\\
& +\frac{q^{-\frac{g}{2}-1}}{1-\sqrt{q}} \sum_{f \in \mathcal{M}_{q, g-X}} \frac{1}{q^{\deg(f)/2}} 
\sum_{\substack{F \in \mathcal{H}_{q^2,\frac{g}{2}+1} \\ (F,fh)=1 \\ P \mid F \Rightarrow P \not\in \F_q[T] }}G_{q^2}(fh^2,F).\label{dual22}
 \end{align}
We write  $S_{1,\text{dual}}= S_{11,\text{dual}}+S_{12,\text{dual}}$ for the terms \eqref{dual21} and \eqref{dual22} respectively on the right-hand side of the above equation. 

We have 
\begin{eqnarray}
 \sum_{\substack{F \in \mathcal{H}_{q^2, \frac{g}{2}+1} \\ (F,fh)=1 \\ P \mid F \Rightarrow P \not\in \F_q[T]}} G_{q^2}(fh^2, F) &=& \sum_{\substack{N \in \F_q[T]\\ \deg(N) \leq \frac{g}{2} + 1\\(N,fh)=1}} \mu(N) \sum_{\substack{F \in \mathcal{M}_{ q^2,\frac{g}{2}+1-\deg(N)}\\
(F, fh)=1}} G_{q^2}(fh^2, NF) \nonumber \\
&=& \sum_{\substack{N \in \F_q[T]\\ \deg(N) \leq \frac{g}{2} + 1\\(N,fh)=1}} \mu(N) G_{q^2}(fh^2, N)   \sum_{\substack{F \in \mathcal{M}_{ q^2,\frac{g}{2}+1-\deg(N)}\\
(F, Nfh)=1}} G_{q^2}(fh^2N, F). \label{sieve_D}
\end{eqnarray}
Now let $(f,h)=B$ and write $f= B \tilde{f}$ and $h = B \tilde{h}$ where $\tilde{f}= f_1 f_2^2 f_3^3$ and $\tilde{h} = h_1 h_2^2 h_3^3$ with $(f_1,f_2)=1$, $(h_1,h_2)=1$ and $f_1,f_2,h_1,h_2$ square-free. Using Proposition \ref{big-F-tilde-corrected}, we get that
\begin{align*}
 \sum_{\substack{F \in \mathcal{M}_{ q^2,\frac{g}{2}+1-\deg(N)}\\
(F, fhN)=1}} & G_{q^2}(fh^2N, F) = \delta_{f_2h_1=1} \frac{q^{\frac{4g}{3}+\frac{8}{3}-4 \deg(N)-\frac{4}{3} \deg(f_1)-\frac{4}{3} \deg(h_2)-\frac{8}{3} [\frac{g}{2}+1+\deg(f_1h_2)]_3 }}{\zeta_{q^2}(2)}  \\
& \times \overline{G_{q^2}(1,f_1h_2N)} \rho(1, [g/2+1+\deg(f_1h_2)]_3) \prod_{\substack{P \in \mathbb{F}_{q^2}[T] \\ P|fhN}} \Big(1+ \frac{1}{|P|_{q^2}} \Big)^{-1} \\
& + O \Big( \delta_{f_2h_1=1} q^{\frac{g}{3}+\varepsilon g - \deg(N) - \frac{\deg(f_1)}{3}-\frac{\deg(h_2)}{3}} \Big) + \frac{1}{2 \pi i} \oint_{|u|=q^{-2\sigma}} \frac{ \Tilde{\Psi}_{q^2}(fh^2N,u)}{u^{\frac{g}{2}+1-\deg(D)}} \, \frac{du}{u},
\end{align*}
with $2/3<\sigma<4/3$. Combining \eqref{dual21} and \eqref{sieve_D}, we write $S_{11,\text{dual}}=M_1+E_1$, where $M_1$ corresponds to the first term above. Using equation \eqref{Gauss-size} and following on similar steps as Section 4.4 in \cite{DFL}, we get that 
\begin{align*}
M_1 =& \frac{q^{5g/6+5/3}}{\zeta_{q^2}(2)} \sum_{\substack{B | h \\ \deg(B) \leq g-X-1}} \frac{1}{q^{\deg(B)/2}} \sum_{\substack{\deg(\tilde{f}) \leq g-X-1-\deg(B)\\(\tilde{f},\tilde{h})=1}} \frac{ \delta_{f_2h_1=1} q^{-\frac{8}{3} [\frac{g}{2}+1+\deg(f_1h_2)]_3}}{q^{\deg(\tilde{f})/2+\deg(f_1h_2)/3}} \\
& \times \rho(1, [g/2+1+\deg(f_1h_2)]_3) \prod_{\substack{P \in \mathbb{F}_{q^2}[T] \\ P|\tilde{f}h}} \Big(1+ \frac{1}{|P|_{q^2}} \Big)^{-1} \\
& \times \sum_{\substack{N \in \F_q[T]\\ \deg(N) \leq \frac{g}{2} + 1\\(N,fh)=1}} \mu(N)  q^{-2 \deg(N)} \prod_{\substack{P \in \mathbb{F}_{q^2}[T] \\ P|N}} \Big(1+ \frac{1}{|P|_{q^2}} \Big)^{-1}.
\end{align*}

Similarly as in \cite{DFL}, we use Perron's formula and the generating series to rewrite the sum over $N$. Again, the only difference is the presence of $h$ in the formulas below. We have
\begin{align*}
\sum_{\substack{N \in \F_q[T]\\ \deg(N) \leq \frac{g}{2} + 1\\(N,fh)=1}} \mu(N)     q^{-2\deg(N)}  & \prod_{\substack{P \in \mathbb{F}_{q^2}[T] \\ P\mid N}} \left (1+ \frac{1}{|P|_{q^2}} \right )^{-1}  = \frac{1}{2 \pi i} \oint \frac{\mathcal{J}_{\mathrm{nK}}(w)}{w^{g/2+1}(1-w)} \\
& \times \prod_{\substack{R \in \mathbb{F}_q[T] \\ \deg(R)  \,\mathrm{ odd} \\ R \mid fh}} A_{\mathrm{dual},R}(w)^{-1} \prod_{\substack{R \in \mathbb{F}_q[T] \\ \deg(R)  \,\mathrm{ even} \\ R\mid fh}} B_{\mathrm{dual},R}(w)^{-1} \, \frac{dw}{w},
\end{align*}
where 
\[\mathcal{J}_{\mathrm{nK}}(w) =  \prod_{\substack{R \in \mathbb{F}_q[T] \\ \deg(R)  \,\mathrm{ odd}}} A_{\mathrm{dual},R}(w) \prod_{\substack{R \in \mathbb{F}_q[T] \\ \deg(R)  \,\mathrm{ even}}} B_{\mathrm{dual},R}(w),\]
and 
\[A_{\mathrm{dual},R}(w)=1-\frac{w^{\deg(R)}}{q^{2\deg(R)}(1+\frac{1}{q^{2\deg(R)}})}\quad \mbox{ and }\quad B_{\mathrm{dual},R}(w)=1-\frac{w^{\deg(R)}}{q^{2\deg(R)}(1+\frac{1}{q^{\deg(R)}})^2}.\]
Introducing the sums over $B$ and $\tilde{f}$, we have
\begin{align} \nonumber
M_1=&\frac{q^{5g/6+5/3}}{\zeta_{q^2}(2)} \sum_{\substack{B | h \\ \deg(B) \leq g-X-1}} \frac{1}{q^{\deg(B)/2}} \sum_{\substack{\deg(\tilde{f}) \leq g-X-1-\deg(B) \\ (\tilde{f},\tilde{h})=1}} \frac{ \delta_{f_2h_1=1} q^{-\frac{8}{3} [\frac{g}{2}+1+\deg(f_1h_2)]_3}}{q^{\deg(\tilde{f})/2+\deg(f_1h_2)/3}} \\  \label{first-M1}
& \times \rho(1, [g/2+1+\deg(f_1h_2)]_3)
\prod_{\substack{R \in \mathbb{F}_q[T] \\ \deg(R)  \,\mathrm{ odd} \\ R \mid fh}} \left(1+\frac{1}{q^{2\deg(R)}}\right)^{-1}
\times \prod_{\substack{R \in \mathbb{F}_q[T] \\ \deg(R)  \,\mathrm{ even} \\ R \mid fh}} \left(1+\frac{1}{q^{\deg(R)}}\right)^{-2} \\
\nonumber
& \times \frac{1}{2 \pi i} \oint \frac{\mathcal{J}_{\mathrm{nK}}(w)}{w^{g/2+1}(1-w)}\prod_{\substack{R \in \mathbb{F}_q[T] \\ \deg(R)  \,\mathrm{ odd} \\ R \mid fh}} A_{\mathrm{dual},R}(w)^{-1} \prod_{\substack{R \in \mathbb{F}_q[T] \\ \deg(R)  \,\mathrm{ even} \\ R\mid fh}} B_{\mathrm{dual},R}(w)^{-1} \, \frac{dw}{w}.
\end{align}

We let
\begin{align*}
\mathcal{H}_{\mathrm{nK}}(h;u,w)= &\sum_{(\tilde{f}, \tilde{h})=1}
\frac{\delta_{f_2=1}}{q^{\deg(\tilde{f})/2+\deg(f_1)/3}} 
\prod_{\substack{R \in \mathbb{F}_q[T] \\ \deg(R)  \,\mathrm{ odd} \\ R \mid \tilde{f} \\ R \nmid h}} C_R(w)^{-1} \prod_{\substack{R \in \mathbb{F}_q[T] \\ \deg(R)  \,\mathrm{ even} \\ R\mid \tilde{f} \\ R \nmid h}} D_R(w)^{-1} u^{\deg(f)},
\end{align*}
where 
\begin{align*} C_R(w) & = A_{\mathrm{dual},R}(w) \left(1+\frac{1}{q^{2\deg(R)}}\right) =
1+\frac{1}{q^{2\deg(R)}}-\frac{w^{\deg(R)}}{q^{2\deg(R)}} \\ D_R(w) &=B_{\mathrm{dual},R}(w)\left(1+\frac{1}{q^{\deg(R)}}\right)^2=\left(1+\frac{1}{q^{\deg(R)}}\right)^2-\frac{w^{\deg(R)}}{q^{2\deg(R)}}.
\end{align*}

Then we can write down an Euler product for $\mathcal{H}_{\mathrm{nK}}(h;u,w)$ and we have that
\begin{align*}
 \mathcal{H}_{\mathrm{nK}}(h;u,w)= & \prod_{\substack{R \in \mathbb{F}_q[T] \\ \deg(R)  \,\mathrm{ odd} \\ R \nmid h}} \left[1+C_R(w)^{-1} \left ( \frac{1}{q^{\deg(R)/3}} \sum_{j=0}^{\infty} \frac{ u^{(3j+1) \deg(R)}}{q^{(3j+1) \deg(R)/2}} + \sum_{j=1}^{\infty} \frac{ u^{3j \deg(R)}}{q^{3j\deg(R)/2}}\right ) \right] \\
 & \times  \prod_{\substack{R \in \mathbb{F}_q[T] \\ \deg(R)  \,\mathrm{ even} \\ R \nmid h}} \left[1+D_R(w)^{-1} \left ( \frac{1}{q^{\deg(R)/3}} \sum_{j=0}^{\infty} \frac{ u^{(3j+1) \deg(R)}}{q^{(3j+1) \deg(R)/2}} + \sum_{j=1}^{\infty} \frac{ u^{3j \deg(R)}}{q^{3j\deg(R)/2}}\right ) \right] \\
&\times  \prod_{\substack{R \in \mathbb{F}_q[T] \\  R|B\\ R \nmid \tilde{h} }} \left[1+ \left ( \frac{1}{q^{\deg(R)/3}} \sum_{j=0}^{\infty} \frac{ u^{(3j+1) \deg(R)}}{q^{(3j+1) \deg(R)/2}} + \sum_{j=1}^{\infty} \frac{ u^{3j \deg(R)}}{q^{3j\deg(R)/2}}\right ) \right]. 
\end{align*}

Following \cite{DFL}, let
\begin{align*}
 \mathcal{H}_{\mathrm{nK}}(u,w)= & \prod_{\substack{R \in \mathbb{F}_q[T] \\ \deg(R)  \,\mathrm{ odd} }} \left[1+C_R(w)^{-1} \left ( \frac{1}{q^{\deg(R)/3}} \sum_{j=0}^{\infty} \frac{ u^{(3j+1) \deg(R)}}{q^{(3j+1) \deg(R)/2}} + \sum_{j=1}^{\infty} \frac{ u^{3j \deg(R)}}{q^{3j\deg(R)/2}}\right ) \right] \\
 & \times  \prod_{\substack{R \in \mathbb{F}_q[T] \\ \deg(R)  \,\mathrm{ even}}} \left[1+D_R(w)^{-1} \left ( \frac{1}{q^{\deg(R)/3}} \sum_{j=0}^{\infty} \frac{ u^{(3j+1) \deg(R)}}{q^{(3j+1) \deg(R)/2}} + \sum_{j=1}^{\infty} \frac{ u^{3j \deg(R)}}{q^{3j\deg(R)/2}}\right ) \right]\\
 =& \mathcal{Z} \left ( \frac{u}{q^{5/6}}\right ) \mathcal{B}_{\mathrm{nK}}(u,w),
\end{align*}
with $\mathcal{B}_{\mathrm{nK}}(u,w)$ analytic in a wider region (for example, $\mathcal{B}_{\mathrm{nK}}(u,w)$ is absolutely convergent for $|u|<q^{\frac{11}{6}}$ and $|uw|< q^{\frac{11}{6}}$). 

After simplifying and using the previous computations from \cite{DFL}, we have
\begin{align}
\mathcal{H}_{\mathrm{nK}} & (h;u,w) = \mathcal{H}_{\mathrm{nK}}(u,w) 
\prod_{\substack{R \in \mathbb{F}_q[T] \\ \deg(R)  \,\mathrm{ odd} \\ R|h}} \left[1+C_R(w)^{-1} \left ( \frac{u^{\deg(R)}}{|R|_q^{5/6} (1- \frac{u^{3 \deg(R)}}{|R|_q^{3/2}})} + \frac{ u^{3 \deg(R)}}{ |R|_q^{3/2}-u^{3 \deg(R)}}\right )\right]^{-1} \nonumber \\
& \times  \prod_{\substack{R \in \mathbb{F}_q[T] \\ \deg(R)  \,\mathrm{ even} \\ R|h}} \left[1+D_R(w)^{-1} \left ( \frac{u^{\deg(R)}}{|R|_q^{5/6} (1- \frac{u^{3 \deg(R)}}{|R|_q^{3/2}})} + \frac{ u^{3 \deg(R)}}{ |R|_q^{3/2}-u^{3 \deg(R)}}\right ) \right]^{-1} \nonumber \\
& \times \prod_{\substack{R \in \mathbb{F}_q[T]  \\ R|B \\ R \nmid \tilde{h}}} \left[1+ \left ( \frac{u^{\deg(R)}}{|R|_q^{5/6} (1- \frac{u^{3 \deg(R)}}{|R|_q^{3/2}})} + \frac{ u^{3 \deg(R)}}{ |R|_q^{3/2}-u^{3 \deg(R)}}\right )\right] \nonumber \\
=&\mathcal{Z} \left ( \frac{u}{q^{5/6}}\right ) \mathcal{B}_{\mathrm{nK}}(u,w) \prod_{\substack{R \in \mathbb{F}_q[T] \\ \deg(R)  \,\mathrm{ odd} \\ R|h}} E_R(u,w)^{-1} \prod_{\substack{R \in \mathbb{F}_q[T] \\ \deg(R)  \,\mathrm{ even} \\ R|h}} G_R(u,w)^{-1}  \prod_{\substack{R \in \mathbb{F}_q[T]  \\ R|B \\ R \nmid \tilde{h}}} F_R(u) . \label{b-def}
\end{align}


We now rewrite $M_1$ by using the generating series above and Perron's formula for the sum over $\tilde{f}$. We need to deal with the  terms
involving $[g/2+1+\deg(f_1h_2)]_3$ appearing in \eqref{first-M1}. 
We notice that
if $g/2+1+\deg(f_1h_2) \equiv 0 \pmod 3$, then $\deg(f_1) \equiv g - \deg(h_2)-1 \pmod 3$, and in that case, $\rho (1,[g/2+1+\deg(f_1h_2)]_3)=1$. If $g/2+1+\deg(f_1h_2) \equiv 1 \pmod 3$, then $\deg(f_1) \equiv g - \deg(h_2) \pmod 3$. In this case we also have $\tau(\chi_3)=q$  by Proposition \ref{big-F-tilde-corrected}, and $\rho(1, [g/2+1+\deg(f_1h_2)]_3)=q^3$, since we are working over $\F_{q^2}$. Using Perron's formula  (Lemma \ref{perron})  twice and keeping in mind that $X \equiv 2 \deg(h) \pmod 3$, we get
\begin{align*}
& M_1 =\frac{q^{5g/6+5/3}}{\zeta_{q^2}(2)} \sum_{\substack{B|h \\ \deg(B) \leq g-X-1}} \frac{\delta_{h_1=1}}{q^{\deg(B)/2+\deg(h_2)/3}} \frac{1}{(2 \pi i)^2} \oint \oint \frac{\mathcal{H}_{\mathrm{nK}}(h;u,w) \mathcal{J}_{\mathrm{nK}}(w)}{w^{g/2+1}(1-w)} \\
& \times  \prod_{\substack{R \in \mathbb{F}_q[T] \\ \deg(R)  \,\mathrm{ odd} \\ R|h}} C_R(w)^{-1}  \prod_{\substack{R \in \mathbb{F}_q[T] \\ \deg(R)  \,\mathrm{ even} \\ R|h}} D_R(w)^{-1}   \left[ \frac{1}{u^{g-X-1-\deg(B)} (1-u^3)} + \frac{ q^{1/3}}{u^{g-X-3-\deg(B)}(1-u^3)} \right] \, \frac{dw}{w} \, \frac{du}{u}.
\end{align*}
We proceed as in \cite{DFL}, shifting the contour of integration over $w$ to $|w|=q^{1-\varepsilon}$, and computing the residue at $w=1$. Writing 
\[ \mathcal{K}_{\mathrm{nK}}(u)=\mathcal{B}_{\mathrm{nK}}(u,1)\mathcal{J}_{\mathrm{nK}}(1),\]
we get that
\begin{align*}
M_1=& \frac{q^{5g/6+5/3}}{\zeta_{q^2}(2)}  \sum_{\substack{B |h \\ \deg(B) \leq g-X-1}}\frac{\delta_{h_1=1}}{q^{\deg(B)/2+\deg(h_2)/3}} \frac{1}{2 \pi i}  \oint  \frac{ \mathcal{K}_{\mathrm{nK}}(u)}{(1-uq^{1/6})(1-u^3) u^{g-X-1}} (1+ q^{1/3}u^2) \\
& \times  \prod_{\substack{R \in \mathbb{F}_q[T] \\ \deg(R)  \,\mathrm{ odd} \\ R|h}} E_R(u,1)^{-1} C_R(1)^{-1} \prod_{\substack{R \in \mathbb{F}_q[T] \\ \deg(R)  \,\mathrm{ even} \\ R|h}} G_R(u,1)^{-1} D_R(1)^{-1} \prod_{\substack{R \in \mathbb{F}_q[T] \\ R|B \\ R \nmid \tilde{h}}} F_R(u)  \, \frac{du}{u} \\
& +O\left(q^{\frac{g}{2}-\frac{X}{6}+\varepsilon g}\right).
\end{align*}

Shifting the contour of integration to $|u|=q^{-\varepsilon}$ and computing the residue at $u=q^{-\frac{1}{6}}$, 
\begin{align*}
M_1 = &2 q^{g-\frac{X}{6}+2} \frac{  \mathcal{K}_{\mathrm{nK}}(q^{-1/6})}{\zeta_{q^2}(2) (\sqrt{q}-1)}    \sum_{\substack{B |h \\ \deg(B) \leq g-X-1}}\frac{\delta_{h_1=1}}{q^{2\deg(B)/3+\deg(h_2)/3}} \prod_{\substack{R \in \mathbb{F}_q[T] \\ \deg(R)  \,\mathrm{ odd} \\ R|h}} E_R(q^{-1/6},1)^{-1} C_R(1)^{-1} \\
& \times  \prod_{\substack{R \in \mathbb{F}_q[T] \\ \deg(R)  \,\mathrm{ even} \\ R|h}} G_R(q^{-1/6},1)^{-1} D_R(1)^{-1} \prod_{\substack{R \in \mathbb{F}_q[T]  \\ R|B \\ R \nmid \tilde{h}}} F_R(q^{-1/6}) + O \Big(q^{\frac{5g}{6}+\varepsilon g} \Big).
\end{align*}
\kommentar{Now note that we have $CS^2E^3 = B h_2^2 h_3^3$ and from here it follows that $C|B$. Write $B=CB_1$. Then $S^2 E^3 = B_1 h_2^2 h_3^3$. Further write $B_1=T_1 T_2^2 T_3^3$ with $T_1, T_2$ square-free and coprime. Since $S^2E^3= T_1 T_2^2 T_3^3 h_2^2 h_3^3$ it follows that $T_1=1, S=T_2h_2, E=T_3h_3$. Then the sum over $B$ in the equation above is equal to
\begin{align*}
& \frac{1}{|C|_q^{2/3} |S|_q^{1/3}}  \sum_{T_2|S} \frac{1}{|T_2|_q} \sum_{T_3|E} \frac{1}{|T_3|_q^2}   \prod_{\substack{R \in \mathbb{F}_q[T] \\ \deg(R)  \,\mathrm{ odd} \\ R | C T_2T_3 \\ R \nmid  \frac{SE}{T_2T_3}}} F_R(q^{-1/6})  \prod_{\substack{R \in \mathbb{F}_q[T] \\ \deg(R)  \,\mathrm{ even} \\ R | C T_2T_3 \\ R \nmid  \frac{SE}{T_2T_3}}} I_R(q^{-1/6}) \\
&=  \frac{1}{|C|_q^{2/3} |S|_q^{1/3}}   \prod_{\substack{R \in \mathbb{F}_q[T] \\ \deg(R)  \,\mathrm{ odd} \\ R | C \\ R \nmid SE}} F_R(q^{-1/6})  \prod_{\substack{R \in \mathbb{F}_q[T] \\ \deg(R)  \,\mathrm{ even} \\ R | C \\ R \nmid SE}} I_R(q^{-1/6})\\
& \times  \sum_{T_2|S} \frac{1}{|T_2|_q} \sum_{T_3|E} \frac{1}{|T_3|_q^2}  \prod_{\substack{R \in \mathbb{F}_q[T] \\ \deg(R)  \,\mathrm{ odd} \\ R |  T_2T_3 \\ R \nmid  \frac{SE}{T_2T_3}}} F_R(q^{-1/6})  \prod_{\substack{R \in \mathbb{F}_q[T] \\ \deg(R)  \,\mathrm{ even} \\ R |  T_2T_3 \\ R \nmid  \frac{SE}{T_2T_3}}} I_R(q^{-1/6})
\end{align*}}
Now note that we can extend the sum over $B$ to include all $B|h$ at the expense of an error term of size $O(\tau(h) / q^{\frac{2}{3}(g-X)})$, giving a total error term of size $O(q^{\frac{g}{3}+\frac{X}{2}+\varepsilon g})$. Then
\begin{align}
 M_1 =& 2 q^{g-\frac{X}{6}+2} \frac{  \mathcal{K}_{\mathrm{nK}}(q^{-1/6})}{\zeta_{q^2}(2) (\sqrt{q}-1)}   \prod_{\substack{R \in \mathbb{F}_q[T] \\ \deg(R)  \,\mathrm{ odd} \\ R|h}} E_R(q^{-1/6},1)^{-1} C_R(1)^{-1}  \prod_{\substack{R \in \mathbb{F}_q[T] \\ \deg(R)  \,\mathrm{ even} \\ R|h}} G_R(q^{-1/6},1)^{-1} D_R(1)^{-1}
 \nonumber  \\
& \times   \sum_{B |h} \frac{\delta_{h_1=1}}{q^{2\deg(B)/3+\deg(h_2)/3}}  \prod_{\substack{R \in \mathbb{F}_q[T]  \\ R|B \\ R \nmid \tilde{h}}} F_R(q^{-1/6}) + O \Big(q^{\frac{5g}{6}+\varepsilon g}+q^{\frac{g}{3}+\frac{X}{2}+\varepsilon g} \Big). \label{m1}
\end{align}
Recall that $h=CS^2 E^3$ with $C,S$ square-free and coprime. Then for the sum over $B$ we can write an Euler product as follows:
\begin{align*}
 \sum_{B |h} & \frac{\delta_{h_1=1}}{q^{2\deg(B)/3+\deg(h_2)/3}}  \prod_{\substack{R \in \mathbb{F}_q[T]  \\ R|B \\ R \nmid \tilde{h}}} F_R(q^{-1/6}) = \prod_{\substack{ R \in \mathbb{F}_q[T] \\ R|C}} \Big(\sum_{\substack{j=1 \\ j \equiv 1 \pmod 3}}^{\ord_R(h)-1} \frac{1}{|R|_q^{2j/3}} +\sum_{\substack{j=2 \\ j \equiv 2 \pmod 3}}^{ \ord_R(h)-1} \frac{1}{|R|_q^{\frac{1}{3}+\frac{2j}{3}}}+ \frac{F_R(q^{-1/6})}{|R|_q^{ \frac{2 \ord_R(h)}{3}}}  \Big) \\
 & \times \prod_{\substack{ R \in \mathbb{F}_q[T] \\ R|S}} \Big(\sum_{\substack{j=2 \\ j \equiv 2 \pmod 3}}^{\ord_R(h)-1} \frac{1}{|R|_q^{2j/3}} +\sum_{\substack{j=0 \\ j \equiv 0 \pmod 3}}^{ \ord_R(h)-1} \frac{1}{|R|_q^{\frac{1}{3}+\frac{2j}{3}}}+ \frac{F_R(q^{-1/6})}{|R|_q^{ \frac{2 \ord_R(h)}{3}}}  \Big) \\
 & \times \prod_{\substack{ R \in \mathbb{F}_q[T] \\ R|E \\ R \nmid CS}} \Big(\sum_{\substack{j=0 \\ j \equiv 0 \pmod 3}}^{\ord_R(h)-1} \frac{1}{|R|_q^{2j/3}} +\sum_{\substack{j=1 \\ j \equiv 1 \pmod 3}}^{ \ord_R(h)-1} \frac{1}{|R|_q^{\frac{1}{3}+\frac{2j}{3}}}+ \frac{F_R(q^{-1/6})}{|R|_q^{ \frac{2 \ord_R(h)}{3}}}  \Big).
\end{align*}
Simplifying and using the fact that
$F_R(q^{-1/6}) = \frac{|R|_q}{|R|_q-1},$ we get that
\begin{align*}
 \sum_{B |h} & \frac{\delta_{h_1=1}}{q^{2\deg(B)/3+\deg(h_2)/3}}  \prod_{\substack{R \in \mathbb{F}_q[T]  \\ R|B \\ R \nmid \tilde{h}}} F_R(q^{-1/6}) = \frac{1}{|C|_q^{2/3} |S|_q^{1/3}} \prod_{\substack{R \in \mathbb{F}_q[T] \\ R|h}} \frac{|R|_q}{|R|_q-1}.
\end{align*}
Using the above and equation \eqref{m1}, it follows that
\begin{align*}
M_1 =& 2 q^{g-\frac{X}{6}+2} \frac{  \mathcal{K}_{\mathrm{nK}}(q^{-1/6})}{|C|_q^{2/3} |S|_q^{1/3} \zeta_{q^2}(2) (\sqrt{q}-1)}   \prod_{\substack{R \in \mathbb{F}_q[T] \\ \deg(R)  \,\mathrm{ odd} \\ R|h}} E_R(q^{-1/6},1)^{-1} C_R(1)^{-1}  \\
& \times \prod_{\substack{R \in \mathbb{F}_q[T] \\ \deg(R)  \,\mathrm{ even} \\ R|h}}  G_R(q^{-1/6},1)^{-1} D_R(1)^{-1}  \times \prod_{\substack{R \in \mathbb{F}_q[T] \\ R|h}} \frac{|R|_q}{|R|_q-1}+O \Big(q^{\frac{5g}{6}+\varepsilon g}+q^{\frac{g}{3}+\frac{X}{2}+\varepsilon g} \Big).
\end{align*}
Putting everything together, we get that
\begin{align*}
S_{11,\mathrm{dual}} =& \frac{ 2q^{g-\frac{X}{6}+2} \mathcal{K}_{\mathrm{nK}}(q^{-1/6})}{|C|_q^{2/3} |S|_q^{1/3} \zeta_{q^2}(2) (\sqrt{q}-1)}   \prod_{\substack{R \in \mathbb{F}_q[T] \\ \deg(R)  \,\mathrm{ odd} \\ R|h}} E_R(q^{-1/6},1)^{-1} C_R(1)^{-1}  \prod_{\substack{R \in \mathbb{F}_q[T] \\ \deg(R)  \,\mathrm{ even} \\ R|h}}  G_R(q^{-1/6},1)^{-1} D_R(1)^{-1} \\
& \times \prod_{\substack{R \in \mathbb{F}_q[T] \\ R|h}} \frac{|R|_q}{|R|_q-1} +O \Big(q^{\frac{5g}{6}+\varepsilon g}+q^{\frac{g}{3}+\frac{X}{2}+\varepsilon g} \Big) \\
& +q^{-\frac{g}{2}-1} \frac{1}{2 \pi i} \oint_{|u|=q^{-2\sigma}} \sum_{f \in \mathcal{M}_{q,\leq g-X-1}} \frac{1}{q^{\deg(f)/2}} \sum_{\substack{N \in \mathbb{F}_q[T] \\ \deg(N) \leq \frac{g}{2}+1 \\ (N,fh)=1}} \mu(N) G_{q^2}(fh^2,N) \frac{\Tilde{\Psi}_{q^2}(fh^2N,u)}{u^{g/2+1-\deg(N)}} \, \frac{du}{u}.
\end{align*} 
We treat $S_{12,\mathrm{dual}}$ similarly and since $\deg(f) = g-X$ we have $[g/2+1+ \deg(f_1h_2)]_3 = 1$. Then as before $\rho(1,1)=\tau(\chi_3) = q^3$, and
we get that
\begin{align*}
S_{12,\mathrm{dual}} =& \frac{q^{g-\frac{X}{6}+2} \mathcal{K}_{\mathrm{nK}}(q^{-1/6})}{|C|_q^{2/3} |S|_q^{1/3} \zeta_{q^2}(2) (1-\sqrt{q})}   \prod_{\substack{R \in \mathbb{F}_q[T] \\ \deg(R)  \,\mathrm{ odd} \\ R|h}} E_R(q^{-1/6},1)^{-1} C_R(1)^{-1}  \prod_{\substack{R \in \mathbb{F}_q[T] \\ \deg(R)  \,\mathrm{ even} \\ R|h}}  G_R(q^{-1/6},1)^{-1} D_R(1)^{-1} \\
& \times \prod_{\substack{R \in \mathbb{F}_q[T] \\ R|h}} \frac{|R|_q}{|R|_q-1} +O \Big(q^{\frac{5g}{6}+\varepsilon g}+q^{\frac{g}{3}+\frac{X}{2}+\varepsilon g} \Big) \\
&+  \frac{ q^{-\frac{g}{2}-1}}{1-\sqrt{q}} \frac{1}{2 \pi i} \oint_{|u|=q^{-2\sigma}} \sum_{f \in \mathcal{M}_{q, g-X}} \frac{1}{q^{\deg(f)/2}} \sum_{\substack{N \in \mathbb{F}_q[T] \\ \deg(N) \leq \frac{g}{2}+1 \\ (N,fh)=1}} \mu(N) G_{q^2}(fh^2,N) \frac{\Tilde{\Psi}_{q^2}(fh^2N,u)}{u^{g/2+1-\deg(N)}} \, \frac{du}{u}.
\end{align*}
Combining the two equations above, we get that
\begin{align}
S_{1,\mathrm{dual}} =&- \frac{q^{g-\frac{X}{6}+2} \mathcal{K}_{\mathrm{nK}}(q^{-1/6})\zeta_q(1/2)}{|C|_q^{2/3} |S|_q^{1/3} \zeta_{q^2}(2)}  \prod_{\substack{R \in \mathbb{F}_q[T] \\ \deg(R)  \,\mathrm{ odd} \\ R|h}} E_R(q^{-1/6},1)^{-1} C_R(1)^{-1}  \prod_{\substack{R \in \mathbb{F}_q[T] \\ \deg(R)  \,\mathrm{ even} \\ R|h}}  G_R(q^{-1/6},1)^{-1} D_R(1)^{-1} \label{final-dual} \\
& \times \prod_{\substack{R \in \mathbb{F}_q[T] \\ R|h}} \frac{|R|_q}{|R|_q-1} +O \Big(q^{\frac{5g}{6}+\varepsilon g}+q^{\frac{g}{3}+\frac{X}{2}+\varepsilon g} \Big) 
\nonumber \\
&+ q^{-g/2-1} \frac{1}{2 \pi i} \oint_{|u|=q^{-2\sigma}} \sum_{f \in \mathcal{M}_{q,\leq g-X-1}} \frac{1}{q^{\deg(f)/2}} \sum_{\substack{N \in \mathbb{F}_q[T] \\ \deg(N) \leq \frac{g}{2}+1 \\ (N,fh)=1}} \mu(N) G_{q^2}(fh^2,N) \frac{\Tilde{\Psi}_{q^2}(fh^2N,u)}{u^{g/2+1-\deg(N)}} \, \frac{du}{u} \nonumber \\
&+  \frac{ q^{-\frac{g}{2}-1}}{1-\sqrt{q}} \frac{1}{2 \pi i} \oint_{|u|=q^{-2\sigma}} \sum_{f \in \mathcal{M}_{q, g-X}} \frac{1}{q^{\deg(f)/2}} \sum_{\substack{N \in \mathbb{F}_q[T] \\ \deg(N) \leq \frac{g}{2}+1 \\ (N,fh)=1}} \mu(N) G_{q^2}(fh^2,N) \frac{\Tilde{\Psi}_{q^2}(fh^2N,u)}{u^{g/2+1-\deg(N)}} \, \frac{du}{u}. \nonumber
\end{align}
Now using the work from \cite{DFL}, we have that
$ \frac{ \mathcal{K}_{\mathrm{nK}}(q^{-1/6})}{ \zeta_{q^2}(2)} = \frac{ \mathcal{A}_{\mathrm{nK}}(1/q^2,1/q) }{\zeta_q(3)}.$
When $\deg(R)$ is odd, note that we have
$$ E_R(q^{-1/6},1)^{-1} C_R(1)^{-1}  \frac{|R|_q}{|R|_q-1} =1,$$ and
when $\deg(R)$ is even, we have 
$$G_R(q^{-1/6},1)^{-1}D_R^{-1} \frac{|R|_q}{|R|_q-1} =\frac{|R|_q^2}{|R|_q^2+2|R|_q-2} = M_R \Big( \frac{1}{q^2},\frac{1}{q} \Big). $$
Hence combining \eqref{final-dual} and \eqref{scube}, we get
\begin{align*}
S_{1,\cube} &+ S_{1,\mathrm{dual}}  = \frac{ q^{g+2} \zeta_q(3/2)}{\zeta_q(3)|C|_q \sqrt{|S|_q}} \mathcal{A}_{\mathrm{nK}} \left ( \frac{1}{q^2}, \frac{1}{q^{3/2}} \right )  \prod_{\substack{R \in \mathbb{F}_q[T] \\ \deg(R) \text{ even} \\ R|h}} M_R \Big( \frac{1}{q^2},\frac{1}{q^{3/2}} \Big) \\
&+ q^{-\frac{g}{2}-1} \frac{1}{2 \pi i} \oint_{|u|=q^{-2\sigma}} \sum_{f \in \mathcal{M}_{q,\leq g-X-1}} \frac{1}{q^{\deg(f)/2}} \sum_{\substack{N \in \mathbb{F}_q[T] \\ \deg(N) \leq \frac{g}{2}+1 \\ (N,fh)=1}} \mu(N) G_{q^2}(fh^2,N) \frac{\Tilde{\Psi}_{q^2}(fh^2N,u)}{u^{g/2+1-\deg(N)}} \, \frac{du}{u} \nonumber \\
&+  \frac{ q^{-\frac{g}{2}-1}}{1-\sqrt{q}} \frac{1}{2 \pi i} \oint_{|u|=q^{-2\sigma}} \sum_{f \in \mathcal{M}_{q, g-X}} \frac{1}{q^{\deg(f)/2}} \sum_{\substack{N \in \mathbb{F}_q[T] \\ \deg(N) \leq \frac{g}{2}+1 \\ (N,fh)=1}} \mu(N) G_{q^2}(fh^2,N) \frac{\Tilde{\Psi}_{q^2}(fh^2N,u)}{u^{g/2+1-\deg(N)}} \, \frac{du}{u}\\
&+ O \Big(q^{\frac{5g}{6}+\varepsilon g}+  q^{\frac{g}{3}+\frac{X}{2}+\varepsilon g}+q^{g-\frac{X}{2}+\varepsilon g} \Big) .
\end{align*}


Using Proposition \ref{big-F-tilde-corrected} and following similar steps as in the proof on page $48$ in \cite{DFL}, we get that
\begin{align*} 
& q^{-\frac{g}{2}-1} \frac{1}{2 \pi i} \oint_{|u|=q^{-2\sigma}} \sum_{f \in \mathcal{M}_{q,\leq g-X-1}} \frac{1}{q^{\deg(f)/2}} \sum_{\substack{N \in \mathbb{F}_q[T] \\ \deg(N) \leq \frac{g}{2}+1 \\ (N,fh)=1}} \mu(N) G_{q^2}(fh^2,N) \frac{\Tilde{\Psi}_{q^2}(fh^2N,u)}{u^{g/2+1-\deg(N)}} \, \frac{du}{u}  \\ & \ll g q^{\frac{3g}{2}-(2-\sigma)X + 2\deg(h) ( \tfrac{3}{2}-\sigma)}, \end{align*}
as long as $\sigma \geq 7/6$. The second integral involving the sum over $f \in \mathcal{M}_{q, g-X}$ is similarly bounded.

Collecting the estimate above for $S_{1,\cube} + S_{1,\mathrm{dual}}$ with the proper error terms, and the estimate for 
$S_{1,\not= \cube}$ of Section \ref{non-cube}, we get
\begin{align*}
\sum_{\chi \in \mathcal{C}(g)}  \chi(h) {\textstyle L(\frac{1}{2}, \chi)}  =&  \frac{ q^{g+2} \zeta_q(3/2)}{\zeta_q(3)|C|_q \sqrt{|S|_q}} \mathcal{A}_{\mathrm{nK}} \left ( \frac{1}{q^2}, \frac{1}{q^{3/2}} \right )  \prod_{\substack{R \in \mathbb{F}_q[T] \\ \deg(R) \text{ even} \\ R|h}} M_R \Big( \frac{1}{q^2},\frac{1}{q^{3/2}} \Big) \\
& +O \Big(q^{\frac{X+g}{2}+\varepsilon g} + q^{\frac{3g}{2} - (2-\sigma) X + 2 \deg(h)( \frac{3}{2}-\sigma)} +q^{\frac{5g}{6}+\varepsilon g}+q^{g-\frac{X}{2}+\varepsilon g}\Big),
\end{align*}
where $7/6 \leq \sigma<4/3$. We pick $\sigma=7/6$ and $X= \frac{3g}{4} + \frac{\deg(h)}{2}$. Then the error term above becomes $O \Big(q^{\frac{7g}{8}+\frac{\deg(h)}{4} +\varepsilon g} \Big)$. Since $\deg(h) < \frac{g}{10}-\varepsilon g$, the main term above dominates the error term, and we have a genuine asymptotic formula. 
\end{proof}

\subsection{Proof of Theorem \ref{first-moment}}

Here we will finish the proof of Theorem \ref{first-moment}. From \eqref{mom1} and Proposition \ref{twist}, it follows that the main term in the mollified first moment is equal to
\begin{equation}\label{arghh}
\frac{q^{g+2} \zeta_q(3/2)}{\zeta_q(3)} \mathcal{A}_{\mathrm{nK}}\left(\frac{1}{q^{2}},\frac{1}{ q^{3/2}}\right) 
\prod_{r = 0}^\JJ  T(r),
\end{equation}

where
\begin{eqnarray*}
T(r) &=& \sum_{\substack{ P|h_r \Rightarrow P \in I_r\\ \Omega(h_r) \leq  \ell_r\\
h_r = C_r S_r^2 E_r^3\\ (C_r, S_r) =1, \; C_r, S_r \; \text{square-free}}} \frac{a(h_r; \JJ ) \lambda(h_r) \nu(h_r)}{ |C_r|_q^{3/2} |S_r|_q^{3/2} |E_r|_q^{3/2}} 
\prod_{\substack{R \in \F_q[T] \\ \deg{R} \;\text{even} \\ R \mid h_r}} M_R\left(\frac{1}{q^{2}}, \frac{1}{q^{3/2}} \right)\\
&\geq &
\sum_{\substack{P \mid h_r \Rightarrow P \in I_r\\h_r = C_r S_r^2 E_r^3\\ (C_r, S_r) =1, \; C_r, S_r \; \text{square-free}}}
\frac{a(h_r; \JJ ) \lambda(h_r) \nu(h_r)}{ {|C_r|_q^{3/2} |S_r|_q^{3/2} |E_r|_q^{3/2}}} 
\prod_{\substack{R \in \F_q[T] \\ \deg{R} \;\text{even} \\ R \mid h_r}} M_R\left(\frac{1}{q^{2}}, \frac{1}{q^{3/2}}\right) \\
&  - &\sum_{\substack{P \mid h_r \Rightarrow P \in I_r\\h_r = C_r S_r^2 E_r^3\\ (C_r, S_r) =1, \; C_r, S_r \; \text{square-free} }}
\frac{2^{\Omega(h_r)}}{ 2^{\ell_r}{|C_r|_q^{3/2} |S_r|_q^{3/2} |E_r|_q^{3/2}}} .
\end{eqnarray*}
\kommentar{\begin{eqnarray*}
T(r) &=& \sum_{\substack{ P|h_r \Rightarrow P \in I_r\\ \Omega(h_r) \leq  \ell_r\\
h_r = C_r S_r^2 E_r^3\\ (C_r, S_r) =1, \; C_r, S_r \; \text{square-free}}} \frac{a(h_r; \JJ ) \lambda(h_r) \nu(h_r)}{ |C_r|_q^{3/2} |S_r|_q^{3/2} |E_r|_q^{3/2}} 
\prod_{\substack{R \in \F_q[T] \\ \deg{R} \;\text{even} \\ R \mid h_r}} M_R\left(\frac{1}{q^{2}}, \frac{1}{q^{3/2}} \right)\\
&\geq &
\sum_{\substack{P \mid h_r \Rightarrow P \in I_r\\h_r = C_r S_r^2 E_r^3\\ (C_r, S_r) =1, \; C_r, S_r \; \text{square-free}}}
\frac{a(h_r; \JJ ) \lambda(h_r) \nu(h_r)}{ {|C_r|_q^{3/2} |S_r|_q^{3/2} |E_r|_q^{3/2}}} 
\prod_{\substack{R \in \F_q[T] \\ \deg{R} \;\text{even} \\ R \mid h_r}} M_R\left(\frac{1}{q^{2}}, \frac{1}{q^{3/2}}\right) \\
&&-  \frac{1}{2^{\ell_r}} \sum_{P \mid h_r \Rightarrow P \in I_r} \frac{ 2^{\Omega(h_r)}}{|h_r|^{3/2}},
\end{eqnarray*}}
where in the second line we have added the $h_r$ with $\Omega(h_r) \geq \ell_r$ to the main sum, 
and we have also used the fact that $2^{\ell_r}\leq 2^{\Omega(h_r)}$ and the bound $\nu(h_r)\leq 1$.
Now we have that
\begin{align*}
 \frac{1}{2^{\ell_r}} & \sum_{\substack{P \mid h_r \Rightarrow P \in I_r\\h_r = C_r S_r^2 E_r^3\\ (C_r, S_r) =1, \; C_r, S_r \; \text{square-free} }}
\frac{2^{\Omega(h_r)}}{ {|C_r|_q^{3/2} |S_r|_q^{3/2} |E_r|_q^{3/2}}} \leq \frac{1}{2^{\ell_r}} \sum_{P|C_r \Rightarrow P \in I_r} \frac{2^{\Omega(C_r)}}{|C_r|_q^{3/2}}  \sum_{P|S_r \Rightarrow P \in I_r} \frac{4^{\Omega(S_r)}}{|S_r|_q^{3/2}}   \\
& \times \sum_{P|E_r \Rightarrow P \in I_r} \frac{8^{\Omega(E_r)}}{|E_r|_q^{3/2}} = \frac{1}{2^{\ell_r}} \prod_{P \in I_r} \Big(1- \frac{2}{|P|_q^{3/2}} \Big)^{-1}\Big(1- \frac{4}{|P|_q^{3/2}} \Big)^{-1} \Big(1- \frac{8}{|P|_q^{3/2}} \Big)^{-1},
\end{align*}
so combining the two equations above, we get that
\begin{align*}
T(r) & \geq \sum_{\substack{P \mid h_r \Rightarrow P \in I_r\\h_r = C_r S_r^2 E_r^3\\ (C_r, S_r) =1, \; C_r, S_r \; \text{square-free}}}
\frac{a(h_r; \JJ ) \lambda(h_r) \nu(h_r)}{ {|C_r|_q^{3/2} |S_r|_q^{3/2} |E_r|_q^{3/2}}} 
\prod_{\substack{R \in \F_q[T] \\ \deg{R} \;\text{even} \\ R \mid h_r}} M_R\left(\frac{1}{q^{2}}, \frac{1}{q^{3/2}}\right) \\
& - \frac{1}{2^{\ell_r}} \prod_{P \in I_r} \Big(1- \frac{2}{|P|_q^{3/2}} \Big)^{-1}\Big(1- \frac{4}{|P|_q^{3/2}} \Big)^{-1} \Big(1- \frac{8}{|P|_q^{3/2}} \Big)^{-1}.
\end{align*}

\kommentar{
\acom{Good point, I think this needs to be fixed. We have
\begin{align*}
 & \sum_{\substack{P \mid h_r \Rightarrow P \in I_r\\h_r = C_r S_r^2 E_r^3\\ (C_r, S_r) =1, \; C_r, S_r \; \text{square-free} \\ \Omega(h_r) >\ell_r}}
\frac{a(h_r; \JJ ) \lambda(h_r) \nu(h_r)}{ {|C_r|_q^{3/2} |S_r|_q^{3/2} |E_r|_q^{3/2}}} 
\prod_{\substack{R \in \F_q[T] \\ \deg{R} \;\text{even} \\ R \mid h_r}} M_R\left(\frac{1}{q^{2}}, \frac{1}{q^{3/2}}\right)\\
& \leq \sum_{\substack{P \mid h_r \Rightarrow P \in I_r\\h_r = C_r S_r^2 E_r^3\\ (C_r, S_r) =1, \; C_r, S_r \; \text{square-free} }}
\frac{2^{\Omega(h_r)}}{ 2^{\ell_r}{|C_r|_q^{3/2} |S_r|_q^{3/2} |E_r|_q^{3/2}}} \leq \frac{1}{2^{\ell_r}} \sum_{P|C_r \Rightarrow P \in I_r} \frac{2^{\Omega(C_r)}}{|C_r|_q^{3/2}}  \sum_{P|S_r \Rightarrow P \in I_r} \frac{4^{\Omega(S_r)}}{|S_r|_q^{3/2}} \\
& \times \sum_{P|E_r \Rightarrow P \in I_r} \frac{8^{\Omega(E_r)}}{|E_r|_q^{3/2}} = \frac{1}{2^{\ell_r}} \prod_{P \in I_r} \Big(1- \frac{2}{|P|_q^{3/2}} \Big)^{-1}\Big(1- \frac{4}{|P|_q^{3/2}} \Big)^{-1} \Big(1- \frac{8}{|P|_q^{3/2}} \Big)^{-1}.
\end{align*}
Then what follows needs to be changed a bit as well. Let me know if you agree.
}\mcom{I agree! Thanks to the two of you for taking care of this. I think this won't change the final bound, but some details in the middle} 
\ccom{Good for me, but some $\Omega(h_r)$ have to be  $\Omega(C_r)$, $\Omega(S_r)$ and $\Omega(E_r)$}
\mcom{No, that's why we have $\leq$} \acom{Oh, sorry, fixed the typos. Yes, we needed $\Omega(C_r)$ etc.}}

Let $U(r)$ denote the first term above. Then 
\begin{align}\label{arghh2}
\prod_{r =0}^{\JJ} T(r) \geq \prod_{r=0}^{\JJ} U(r) \prod_{r=0}^{\JJ} \left(1- \frac{1}{2^{\ell_r} U(r) \prod_{P \in I_r} \Big(1- \frac{2}{|P|_q^{3/2}} \Big)\Big(1- \frac{4}{|P|_q^{3/2}} \Big) \Big(1- \frac{8}{|P|_q^{3/2}} \Big) }\right).
\end{align}

We first focus on 
\begin{align}
\mathcal{U}:=\prod_{r=0}^{\JJ} U(r)=& \prod_{r=0}^\JJ  
\sum_{\substack{P \mid h_r \Rightarrow P \in I_r\\h_r = C_r S_r^2 E_r^3\\ (C_r, S_r) =1, \; C_r, S_r \; \text{square-free}}}
\frac{a(h_r; \JJ ) \lambda(h_r) \nu(h_r)}{ {|C_r|_q^{3/2} |S_r|_q^{3/2} |E_r|_q^{3/2}}} 
\prod_{\substack{R \in \F_q[T] \\ \deg{R} \;\text{even} \\ R \mid h_r}} M_R\left(\frac{1}{q^{2}}, \frac{1}{q^{3/2}}\right) \nonumber \\
\label{main_moll-before}
= &\prod_{r=0}^\JJ  \prod_{P\in I_r}
\left[1+\sum_{e=0}^\infty \frac{a(P;\JJ )^{3e+1}(-1)^{3e+1}}{|P|_q^{3(e+1)/2}(3e+3)!}\right.\\ \nonumber
&\times \left.\left(a(P;\JJ )^2+3(e+1)(-a(P;\JJ )+3e+2)\right)N_P\left(\frac{1}{q^{2}}, \frac{1}{q^{3/2}}\right)\right]
\end{align} 
where $N_P\left(\frac{1}{q^{2}}, \frac{1}{q^{3/2}}\right)=M_P\left(\frac{1}{q^{2}}, \frac{1}{q^{3/2}}\right)$ or 1, according to whether $\deg(P)$ is even or odd. Thus
\begin{align}
\mathcal{U} &= \prod_{\deg(P) \leq (g+2) \theta_{\JJ}}
\left[1+\left[\frac{1}{3}\left(1+\frac{1}{|P|_q^{1/2}}+\frac{1}{|P|_q}\right)\exp\left( -\frac{a(P;\JJ )}{|P|_q^{1/2}}\right)\right.\right. \nonumber \\
&\left.+\frac{1}{3}\left(1+\frac{\xi_3}{|P|_q^{1/2}}+\frac{\xi_3^2}{|P|_q}\right) \exp\left( -\frac{\xi_3a(P;\JJ )}{|P|_q^{1/2}}\right)+\frac{1}{3}\left(1+\frac{\xi_3^2}{|P|_q^{1/2}}+\frac{\xi_3}{|P|_q}\right) \exp\left( -\frac{\xi_3^2a(P;\JJ )}{|P|_q^{1/2}}\right)-1\right] \nonumber \\
&\times \left. N_P\left(\frac{1}{q^{2}}, \frac{1}{q^{3/2}}\right)\right]. \label{main_moll}
\end{align}

For the second product of \eqref{arghh2}, we have
\begin{align*}
&\prod_{r=0}^{\JJ} \left(1- \frac{1}{2^{\ell_r} U(r) \prod_{P \in I_r} \Big(1- \frac{2}{|P|_q^{3/2}} \Big)\Big(1- \frac{4}{|P|_q^{3/2}} \Big) \Big(1- \frac{8}{|P|_q^{3/2}} \Big) }\right)\\ 
\geq & \left(1- \frac{1}{2^{\ell_0}} \exp \Big( \sum_{n=1}^{\infty} \frac{q^n}{n} \Big(\frac{1}{q^{3n/2}-1}+\frac{2}{q^{3n/2}-2}+\frac{4}{q^{3n/2}-4}+\frac{8}{q^{3n/2}-8} \Big)  \Big) \right) \\
\times &  \prod_{r=1}^{\JJ} \left( 1- \frac{1}{2^{\ell_r} } \exp\Big( \sum_{n=(g+2) \theta_{r-1}}^{(g+2) \theta_r} \frac{15}{nq^{n/2}} +O\Big(\frac{1}{q^{2g \theta_{r-1}}} \Big) \Big)\right)  \\
 \geq & \Big(1-\frac{1}{2^{\ell_0}} K \Big)\prod_{r=1}^{\JJ} \Big(1-\frac{1}{2^{\ell_r}} +O \Big(\frac{1}{2^{\ell_r} q^{g \theta_{r-1}/2}} \Big) \Big) \\
 \geq &1-\frac{1}{e^{e^{84}}},
\end{align*}
where in the second line we used the inequality form of the Prime Polynomial Theorem \eqref{ppt-bound}, 
$K = \exp \Big( \sum_{n=1}^{\infty} \frac{q^n}{n} \Big(\frac{1}{q^{3n/2}-1}+\frac{2}{q^{3n/2}-2}+\frac{4}{q^{3n/2}-4}+\frac{8}{q^{3n/2}-8} \Big)  \Big),$ and the estimate in the last line is taken with the constants chosen in Section \ref{explicit-UB}.

Putting together all this information, we obtain that 
\begin{align}
\frac{q^{g+2} \zeta_q(3/2)}{\zeta_q(3)} \mathcal{A}_{\mathrm{nK}}\left(\frac{1}{q^{2}},\frac{1}{ q^{3/2}}\right) 
\prod_{r = 0}^\JJ  T(r) \geq \left(1-\frac{1}{e^{e^{84}}}\right) \frac{ q^{g+2} \zeta_q(3/2)}{\zeta_q(3)} \mathcal{A}_{\mathrm{nK}} \left ( \frac{1}{q^2}, \frac{1}{q^{3/2}} \right ) \mathcal{U}.
\label{mtmol}
\end{align}

Finally, summing the error term coming from Proposition \ref{twist} gives
\begin{align} \label{et_moll}
&q^{\frac{7g}{8}+\varepsilon g}\sum_{\deg(h) \leq w(g+2)} \frac{|h|^{1/4}}{|h|^{1/2}}\sim q^{\frac{7g}{8}+\varepsilon g+\frac{3w(g+2)}{4}} ,
\end{align}
where $w=\sum_{j=0}^\JJ  \theta_j \ell_j$. Note that because of equation \eqref{important_condition}, we have that
$$\sum_{j=0}^{\JJ} \theta_j \ell_j \leq \frac{1}{20},$$ so the above constitutes an error term. This finishes the proof of Theorem \ref{first-moment}.

\begin{proof}[Proof of Corollary \ref{lb}]
Note that from expression \eqref{main_moll-before}, we can write 
\begin{align*}
\mathcal{U} &\geq  \prod_{\deg(P) \leq (g+2) \theta_{\JJ}} \Big( 1- \frac{a(P;\JJ)}{6|P|^{3/2}} ( a(P;\JJ)^2-3a(P;\JJ)+6) \Big) \\
&\geq \prod_{\deg(P) \leq (g+2) \theta_{\JJ}}\Big( 1- \frac{1}{|P|^{3/2}}\Big)\\
&\geq \zeta_q(3/2)^{-1}.
\end{align*}
We also have
 \[\mathcal{A}_{\mathrm{nK}} \left ( \frac{1}{q^2}, \frac{1}{q^{3/2}} \right )=
 \prod_{\substack{R \in \F_q[T]\\\deg(R) \, \mathrm{odd}}}\frac{1}{1+\frac{1}{|R|^2}}\prod_{\substack{R \in \F_q[T]\\\deg(R) \, \mathrm{even}}}\frac{1+\frac{2}{|R|}(1-\frac{1}{|R|^{3/2}})}{(1+\frac{1}{|R|})^2}.\]
For the factors involving $R$ of even degree, we have that 
$$ 1-\frac{1}{(|R|+1)^2}- \frac{2}{|R|^{1/2}(|R|+1)^2}>\Big(1-\frac{1}{|R|^2} \Big)^2$$ and this leads to  
$$ \mathcal{A}_{\mathrm{nK}} \left ( \frac{1}{q^2}, \frac{1}{q^{3/2}} \right ) \geq \zeta_q(2)^{-2}.$$
Combining everything, the main term of the mollified moment in \eqref{mtmol}  satisfies
\[\geq \left(1-\frac{1}{e^{e^{84}}}\right)  \frac{q^{g+2}}{\zeta_q(2)^2\zeta_q(3)}\geq0.6143 q^{g+2},\]
where we have bounded by the worst case $q = 5$. 
\end{proof}

 \bibliographystyle{amsalpha}

\bibliography{Bibliography}

\end{document}